\documentclass{mcom-l}

\usepackage{latexsym, amsmath, amssymb}
\usepackage{amsthm}
\usepackage{color,tabularx,booktabs}
\usepackage{graphicx,subfigure} 
\usepackage{epstopdf}
\usepackage{multirow}
\usepackage{array}
\usepackage{booktabs}
\usepackage{caption}

\usepackage{array}

\usepackage{hyperref}
\usepackage{empheq}

\newcommand{\f}{\frac}

\newcommand{\mbf}{\mathbf}

\newcommand{\mt}{\mathcal{ T }}

\newcommand{\mo}{\mathcal{ O }}

\newcommand{\mm}{\mathcal{ M }}

\newtheorem{theorem}{Theorem}[section]
\newtheorem{lemma}[theorem]{Lemma}

\theoremstyle{definition}

\theoremstyle{remark}
\newtheorem{remark}[theorem]{Remark}

\numberwithin{equation}{section}

\begin{document}

\title[Dynamic-stabilization-based linear schemes]{Dynamic-stabilization-based linear schemes for the Allen--Cahn equation with degenerate mobility: MBP and energy stability}



\author{Hongfei Fu}
\address{School of Mathematical Sciences \& Laboratory of Marine Mathematics, Ocean University of China, Qingdao, Shandong 266100, China}
\email{fhf@ouc.edu.cn}
\thanks{The first author was supported in part by the NSFC grant 12131014, the NSF of Shandong Province grant ZR2024MA023, the Fundamental Research Funds for the Central Universities grant 202264006 and the OUC Scientific Research Program for Young Talented Professionals.}

\author{Dianming Hou}
\address{School of Mathematics and Statistics, Jiangsu Normal University, Xuzhou, Jiangsu 221116, China}
\email{dmhou@jsnu.edu.cn}
\thanks{The second author was supported in part by the NSFC grants 12001248 and 12571387, and the NSF of Jiangsu Province grant BK20201020.}
	
\author{Zhonghua Qiao}
\address{Corresponding author. Department of Applied Mathematics, The Hong Kong Polytechnic University, Hung Hom, Kowloon, Hong Kong}
\email{zqiao@polyu.edu.hk}
\thanks{The third author was supported in part by the Hong Kong Research Grants Council GRF grant 15305624, and CAS AMSS-PolyU Joint Laboratory of Applied Mathematics (Grant no. JLFS/P-501/24).}

\author{Bingyin Zhang}
\address{School of Mathematical Sciences, Ocean University of China, Qingdao, Shandong 266100, China}
\email{zhangbingyin@stu.ouc.edu.cn}
\thanks{The fourth author was supported in part by the Fundamental Research Funds for the Central Universities under grant 202461103.}


\date{\today}

\dedicatory{}


\begin{abstract}
In this paper, we investigate linear first- and second-order numerical schemes for the Allen--Cahn equation with a general (possibly degenerate) mobility. Compared with existing numerical methods, our schemes employ a novel dynamic stabilization approach that guarantees unconditional preservation of the maximum bound principle (MBP) and energy stability. A key advance is that the discrete energy stability remains valid even in the presence of degenerate mobility-a property we refer to as mobility robustness. Rigorous maximum-norm error estimates are also established. In particular, for the second-order scheme, we introduce a new prediction strategy with a cut-off preprocessing procedure on the extrapolation solution, and only one linear system needs to be solved per time level.
Representative numerical examples are provided to validate the theoretical findings and performance of the proposed schemes.
\end{abstract}

\maketitle

\section{Introduction}
The classic Allen–Cahn equation, first introduced by Allen and Cahn in 1979 \cite{Acta_Allen_1979}, was designed to simulate the motion of anti-phase boundaries in crystalline solids. Since its inception, it has been extensively used to model a wide array of phenomena, such as interface dynamics in multiphase fluids \cite{CiCP_2012_Kim,SISC_2010_Shen,JCP_2006_Yang}, mean curvature flow \cite{NM_2003_Feng,JDG_1993_Ilmanen}, image segmentation \cite{ANM_2004_Benes,LiuQiaoZhang}, and various applications in materials science. In numerous practical situations, the two phases present in the system (such as the components of a binary alloy) exhibit distinct physical or chemical characteristics, which naturally leads to the diffusion mobility being dependent on the phase variable \cite{SIAM_2012_Du}. Furthermore, in specific cases, accurately modeling the phase separation process necessitates that the mobility vanishes in pure phases \cite{JCP_2011_Gomez,CMS_Shen_2016}. This study focuses on the following Allen–Cahn equation with a general (potentially degenerate) mobility:
\begin{empheq}{align}
	& \phi_{t} = - \mm(\phi) \mu, & t >0, \  \mbf{x} \in \Omega , \label{Model:MAC1}\\
	& \mu = - \varepsilon^2 \Delta \phi - f( \phi ),  &  t >0, \  \mbf{x} \in \Omega,  \label{Model:MAC2}
\end{empheq}
subject to the initial condition $ \phi( \mbf{x}, 0 ) = \phi_{\text {init }} ( \mbf{x} ) $ for any $ \mbf{x} \in \bar{\Omega} $, together with either periodic or homogeneous Neumann boundary conditions. Here, $\Omega \subset \mathbb{R}^d$ $(d = 1,2,3)$ is an open, bounded Lipschitz domain. In \eqref{Model:MAC1}--\eqref{Model:MAC2}, $\phi( \mbf{x}, t ) $ represents the relative concentration difference between the two phases, $  0 < \varepsilon \ll 1 $ is a parameter related to the interfacial thickness, and $ f(\phi) = -F'(\phi) $ with $ F(\phi) = \f{1}{4} ( 1 - \phi^2 )^2 $ being the standard double-well potential. The mobility function $\mm(\phi) \geq 0$ is allowed to be nonlinear and possibly degenerate. One commonly used degenerate mobility function in materials science is as follows:
\begin{equation}\label{Def:Mobi_1}
	\mm(\phi) = ( 1 - \phi^{2} )^{m} \quad \text{with} \  m > 0,
\end{equation}
while other forms of mobility can be found in \cite{SIAM_2012_Du,JCP_2016_Du,ARMA_2016_Du,PRE_1997_Puri,JSP_1994_Taylor} or in subsection \ref{sub:Eff_mobi}.

The Allen--Cahn equation \eqref{Model:MAC1}--\eqref{Model:MAC2} satisfies following \textit{energy dissipation law}
\begin{equation}\label{Eg_law}
	\begin{aligned}
		\f{ d }{ dt } E[ \phi ] = - \int_{\Omega} \mm( \phi ) \mu^2 d \mbf{x} \leq 0 ,
	\end{aligned}
\end{equation}
where the free energy $E[\phi]$ is defined by
\begin{equation}\label{def:energy}
	\begin{aligned}
		E[\phi] = \int_{\Omega} \left( \f{\varepsilon^2}{2} \vert \nabla \phi(\mbf{x}) \vert^2 + F( \phi(\mbf{x}) ) \right) d \mbf{x}.
	\end{aligned}
\end{equation}
This shows that the free energy $E[\phi]$ decreases monotonically over time. Moreover, the Allen--Cahn equation also satisfies the maximum bound principle (MBP), that is, if the initial data satisfy $|\phi_{\text {init }}(\mbf{x})| \leq 1$, then the solution remains bounded as $|\phi(\mbf{x},t)| \leq 1$ for all $t>0$; see \cite{CMS_Shen_2016} for further details. It is noteworthy that the energy dissipation law and the MBP are two fundamental properties of the Allen-Cahn equation. Therefore, it is crucial for numerical schemes to preserve such physical features at the discrete level, particularly the MBP, since its violation may cause the mobility (e.g., \eqref{Def:Mobi_1}) to become negative, leading to ill-posedness and nonphysical numerical solutions.

In the past decade, significant progress has been made in developing numerical schemes for Allen–Cahn type equations that ensure the preservation of the discrete MBP and energy stability, particularly for cases with \textit{constant} mobility. Tang et al. \cite{JCM_Tang_2016} introduced a first-order linear stabilized scheme that was both energy-dissipative and MBP-preserving. To achieve second-order accuracy in time while maintaining (modified) energy dissipation and MBP, nonlinear variable-step second-order backward differentiation formula (BDF2) schemes \cite{SINUM_2020_Liao} and stabilized Crank-Nicolson/Adams-Bashforth schemes \cite{AML_2020_Hou} have been explored. High-order MBP-preserving time integration schemes, reaching up to third- and fourth-order, have been developed for semilinear parabolic equations using the (stabilized) integrating factor Runge-Kutta (IFRK) method \cite{JCP_2021_Ju,SISC_2021_Li,CMAME_2022_Zhang}. More recently, Liao et al. \cite{JCP_2024_Liao} proposed corrected IFRK methods using two classes of difference corrections, the telescopic correction and nonlinear-term translation correction, which preserve both steady-state solutions and the original energy dissipation law. Li et al. \cite{SISC_2020_LiBuyang} developed high-order in time numerical methods for parabolic equations, enforcing the discrete MBP at nodal points through a cut-off operation applied to the lumped mass finite element solution, although energy stability was not addressed. This cut-off approach was subsequently combined with the scalar auxiliary variable (SAV) approach \cite{JCP_2018_Shen,SIREV_2019_Shen}, resulting in (modified) energy-stable and MBP-preserving linear schemes \cite{JSC_2022_Yang}. Further advancements include the use of Lagrange multiplier techniques for positivity- and bound-preserving discretizations of parabolic equations \cite{CMAME_2022_Huang,SINUM_2022_Huang}, offering a fresh perspective on the practical cut-off method. Additionally, Ju et al. \cite{SINUM_2022_Ju} integrated the SAV method with the stabilized exponential time differencing (ETD) approach \cite{SINUM_2019_Du,SIREV_Du_2021} to design first- and second-order linear schemes that unconditionally preserve the discrete MBP and a modified energy dissipation law for the Allen–Cahn equation with constant mobility.

The presence of nonlinear and potentially degenerate \textit{variable} mobility significantly complicates the construction and analysis of numerical schemes for the generalized Allen–Cahn equation \eqref{Model:MAC1}--\eqref{Model:MAC2}, compared to the classical Allen–Cahn equation with \textit{constant} mobility. In \cite{CMS_Shen_2016}, a stabilized, MBP-preserving first-order linear scheme was developed for the generalized Allen–Cahn equation with an advection term. To circumvent the challenges posed by potential mobility degeneracy, Cai et al. \cite{CMS_2023_Cai} assumed the existence of a positive constant $\mm_{*}$ such that
\begin{equation}\label{Assu:M}
	\mm(s) \geq \mm_{*} > 0, \quad \forall s \in \mathbb{R},
\end{equation}
which excludes degenerate mobility functions like \eqref{Def:Mobi_1} from practical applications. Under this assumption, they proposed stabilized first-order ETD and second-order ETD Runge–Kutta schemes that unconditionally preserve the discrete MBP. Recently, Hou et al. \cite{MOC_2023_Ju} introduced a linear variable-step second-order scheme for the Allen–Cahn equation with general mobility, combining a prediction-correction approach with stabilization. This scheme employs a classical first-order method to predict a locally second-order and MBP-preserving solution, followed by a stabilized BDF2-type correction step. As a result, two variable-coefficient Poisson-type systems are solved per time step. Using the kernel recombination technique \cite{SINUM_2020_Liao,SISC_2016_Lv}, the discrete MBP was established for general mobility under certain constraints on the time stepsize and the ratio of adjacent time stepsizes, while energy stability was proven only for constant mobility. Subsequently, by imposing two linear stabilization terms, a family of double stabilized prediction-correction schemes was proposed in \cite{AAMM_2024_Hou,CMA_2025_Hou,ANM_2025_Hou} to unconditionally preserve the discrete MBP. However, these results rely on assumptions of either constant mobility or variable mobility with a uniformly positive lower bound, which may not reflect physical reality or practical needs. This motivates us to design and analyze efficient numerical schemes that maintain the structure-preserving properties of previous methods while relaxing restrictive mobility assumptions.

In this paper, we aim to develop and analyze unconditionally MBP-preserving and energy-stable linear schemes for the Allen–Cahn equation \eqref{Model:MAC1}--\eqref{Model:MAC2} with degenerate mobility. The main contributions of this work in constructing and analyzing structure-preserving numerical methods include the following:
\begin{itemize}
	\setlength{\itemindent}{0pt}
	\setlength{\leftskip}{0pt}
	
	\item \textbf{Mobility-robust energy stability analysis}: We introduce a novel dynamic stabilization technique, distinct from previous strategies \cite{MOC_2023_Ju,ANM_2025_Hou,JSC_2022_Ju}, to establish two linear schemes with mobility-robust energy stability. Notably, the stability analysis does not depend on $1/\mm(\phi)$, which is irrelevant when $\mm(\phi)=0$. To our knowledge, this is the first work to develop and analyze a linear scheme ensuring mobility-robust energy stability for the degenerate Allen--Cahn equation.
	
	\item \textbf{Unconditional MBP-preservation}: The new dynamic stabilization approach guarantees unconditional preservation of the discrete MBP for both schemes. A novel prediction-correction strategy is designed and analyzed for the second-order unconditionally MBP-preserving linear scheme. This approach precludes using the standard M-matrix-based analysis framework for discrete MBP \cite{MOC_2023_Ju,ANM_2025_Hou,JSC_2022_Ju,SINUM_2020_Liao,JSC_2025_Zhang}. Instead, a discrete maximum principle (DMP) analysis approach is adopted for MBP analysis, as in \cite{SIREV_Du_2021,JCP_Lan_2023} for continuous problems and \cite{CMS_Shen_2016} for time-discrete schemes.
	
	\item \textbf{Maximum-norm error analysis}:
	Utilizing the discrete MBP property and DMP analysis framework, we rigorously establish optimal maximum-norm error estimates for both schemes. Theorems \ref{thm:Conver_sBE} and \ref{thm:Conver_sCN} present first-order temporal and second-order spatial accuracy for the dynamically stabilized backward Euler (DsBE) scheme, and second-order accuracy in time and space for the dynamically stabilized Crank-Nicolson (DsCN) scheme.
	
	\item \textbf{Improved computational efficiency}: Unlike existing MBP-preserving second-order linear schemes \cite{MOC_2023_Ju,SINUM_2022_Ju,JSC_2022_Ju,JSC_2025_Zhang}, which require solving an auxiliary first-order scheme for prediction, we employ a cut-off preprocessing procedure on the extrapolation solution. This modification reduces computational expense while retaining MBP preservation and energy stability. Moreover, the construction and analysis of the scheme are based on variable time stepsizes, which allows us to employ an adaptive time-stepping strategy \cite{SISC_2011_Qiao} to further improve computational efficiency without sacrificing accuracy, as demonstrated in subsection \ref{sub:adaptive}.
\end{itemize}

The remainder of this paper is organized as follows. In Section \ref{Sec:BE}, we first present a first-order linear dynamically stabilized scheme, and then provide rigorous unconditional mobility-robust numerical analyses of MBP preservation, energy dissipation, and maximum-norm error estimates. Section \ref{Sec:CN} introduces a second-order linear dynamically stabilized Crank–Nicolson scheme using a novel prediction technique within the prediction-correction framework, and corresponding structure-preserving analysis and error estimates are also discussed. In Section \ref{sec:test}, representative numerical examples are provided to demonstrate the performance of the proposed schemes. Finally, some concluding remarks are drawn in the last section.

\section{First-order stabilized backward Euler scheme}\label{Sec:BE}
In this section, we investigate the linear dynamically stabilized backward Euler (DsBE) scheme for the Allen--Cahn equation \eqref{Model:MAC1}--\eqref{Model:MAC2}. For simplicity, we restrict the following discussion to the $d$-dimensional square domain $ \Omega = [0, L]^{d} $ for $d \leq 3$, subject to periodic boundary conditions. It is worth noting that the proposed scheme and its associated analysis can be readily extended to the case of homogeneous Neumann boundary condition.

\subsection{Dynamically stabilized backward Euler scheme}
For a given positive integer $ M $, let $ \Omega_{h} := \big\{ \mbf{x_{i}} = h \mbf{i} \mid \mbf{i} \in \{ 1, 2, \cdots , M \}^{d} \big\} $ be the set of nodes of a uniform Cartesian mesh over $ \Omega $, where the spatial grid length is $ h = L/M $. For ease of notation, we denote the grid function by $ v_{ \mbf{i} } := v (\mbf{x_i} ) $ for any $ \mbf{x_i} \in \Omega_{h} $. Let $ \mathbb{V}_h $ be the set of all $M$-periodic, real-valued grid functions on $ \Omega_{h} $, i.e.,
$$
\mathbb{V}_h:=\left\{v \mid v= \{ v_{ \mbf{i} } \}_{\mbf{x_{i}}\in\Omega_{h}}~ \text{and}~  v ~ \text{is periodic}\right\} .
$$
For any $ v, w \in \mathbb{V}_h $, define the following discrete inner product, and discrete $ L^2 $ and $ L^{\infty} $ norms by
$$
\langle v, w\rangle=h^2 \sum_{\mbf{x_i}, \mbf{x_j} \in \Omega_{h}} v_{ \mbf{i} } w_{\mbf{j}}, \quad\|v\|=\sqrt{\langle v, v\rangle}, \quad\|v\|_{\infty} = \max_{ \mbf{x_i} \in \Omega_{h} } | v_{ \mbf{i} } |.
$$
Let $ \Delta_{h} $ be the discrete Laplacian operator obtained using the standard second-order central finite difference method.
Finally, we introduce a discrete analogue of the energy functional \eqref{def:energy} defined by
\begin{equation}\label{def:dis_energy}
	E_h[v] := - \f{\varepsilon^2}{2} \langle \Delta_{h} v, v \rangle + \langle F(v), 1 \rangle, \quad \forall v \in \mathbb{V}_h.
\end{equation}

In what follows, we consider a general partition of the time interval $J:= [0,T] $ by $ 0 = t_{0} < t_{1} < \cdots < t_{N} = T$ with subinterval $J_{k} :=[t_{k-1},t_{k}]$ and variable time stepsize $ \tau_{k} := t_{k} - t_{k-1} $ for $ 1 \leq k \leq N $. Assume the partition is regular, i.e., there is a positive constant $\sigma$ such that $r_{k+1}:=\tau_{k+1}/\tau_{k}\le \sigma$. Denote by $ \tau := \max_{1 \leq k \leq N} \tau_{k}$ the maximum time stepsize. For any real-valued sequence $ \{ v^{k} \}^{N}_{k=0} $, set $ \nabla_{\tau} v^{k} = v^{k} - v^{k-1}$ and $ \delta_t v^{k} = \nabla_{\tau} v^{k}/\tau_{k} $ for $ k \geq 1 $.
Then, the first-order fully-discrete DsBE scheme takes the form
\begin{empheq}{align*}
	& \delta_t \phi^{n+1} = - \mm( \phi^{n} ) \mu^{n+1} + \mathcal{S}_{2}^{n+1}[ \phi ], \\
	& \mu^{n+1} = - \varepsilon^2 \Delta_{h} \phi^{n+1} - f( \phi^{n} ) + \mathcal{S}_{1}^{n+1}[ \phi ],
\end{empheq}
where $ \mathcal{S}_{1}^{n+1}[ \phi ] $ and $ \mathcal{S}_{2}^{n+1}[ \phi ] $ denote suitable stabilization terms, which play a crucial role in preserving the discrete MBP. To ensure a mobility-robust energy stability analysis, these stabilization terms are assumed to satisfy the boundedness condition
\begin{equation}\label{Ass:sBE}
	- ( \mathcal{S}_{1}^{n+1}[\phi], \delta_t \phi^{n+1} ) + ( \mathcal{S}_{2}^{n+1}[\phi], \mu^{n+1} )  \leq ( \mm(\phi^{n}) \mu^{n+1}, \mu^{n+1} ).
\end{equation}
A natural and effective choice is to set $ \mathcal{S}_{1}^{n+1}[\phi] = S_{1} ( \phi^{n+1} - \phi^{n} ) $ and $ \mathcal{S}_{2}^{n+1}[\phi] = 0 $, where $ S_{1} \geq 0 $ is a constant. Therefore, the first-order DsBE scheme can be reformulated as
\begin{empheq}{align}
	& \delta_t \phi^{n+1} = - \mm( \phi^{n} ) \mu^{n+1},\label{sch:sBE_1}\\
	& \mu^{n+1} = - \varepsilon^2 \Delta_{h} \phi^{n+1} - f( \phi^{n} ) + S_{1} ( \phi^{n+1} - \phi^{n} ), \label{sch:sBE_2}
\end{empheq}
which is simply denoted as $ \phi^{n+1}: = \text{DsBE}( S_{1}, \phi^{n}, \mm(\phi^{n}), \tau_{n+1} )$.

\begin{remark}
	Another first-order linear stabilization scheme was proposed as follows:
	\begin{equation}\label{sch:sBE:e2}
		\delta_t \phi^{n+1} = \mm( \phi^{n} ) \left( \varepsilon^2 \Delta_{h} \phi^{n+1} + f( \phi^{n} ) \right) - S_{1}^{*} ( \phi^{n+1} - \phi^{n} )
	\end{equation}
	in \cite{MOC_2023_Ju,CMS_Shen_2016}. The primary difference between this scheme and our approach lies in the stabilization term. By substituting \eqref{sch:sBE_2} into \eqref{sch:sBE_1}, one finds that the stabilization term in our scheme takes the form $ - S_{1} \mm( \phi^{n} ) ( \phi^{n+1} - \phi^{n} ) $, which satisfies \eqref{Ass:sBE} and features a non-negative coefficient that dynamically depends on $ \phi^{n} $. This dynamic design is particularly advantageous for establishing mobility-robust energy stability, especially for the second-order scheme discussed in next section.
\end{remark}

\begin{remark}
	Using the M-matrix-based analysis technique \cite{MOC_2023_Ju,ANM_2025_Hou}, it has been shown that the scheme \eqref{sch:sBE:e2} unconditionally preserves the MBP, provided that the stabilization parameter satisfies
	\begin{equation}\label{sch:sBE:e3}
		S_{1}^{*} \geq \max_{\xi\in[-1,1]} \left( \mm'(\xi) F'(\xi) + \mm(\xi) F''(\xi) \right).
	\end{equation}
	However, owing to the introduction of dynamic stabilization terms, it leads to significant challenges for the theoretical verification of MBP preservation (cf. Theorems \ref{thm:MBP_sBE} and \ref{thm:MBP_sCN}) and for the maximum-norm error analysis (cf. Theorems \ref{thm:Conver_sBE} and \ref{thm:Conver_sCN}) for our schemes. To address these difficulties, we employ a DMP analysis framework, which effectively circumvents the influence of the mobility variation---a key obstacle for the M-matrix-based framework commonly employed in \cite{MOC_2023_Ju,ANM_2025_Hou,JSC_2022_Ju,JSC_2025_Zhang}. The following theorem shows that the novel DsBE scheme preserves the MBP only under $S_{1} \geq\left\|f^{\prime}\right\|_{C[-1, 1]}=2$.
\end{remark}

\subsection{The MBP and energy dissipation}
In this subsection, we demonstrate that the proposed first-order DsBE scheme \eqref{sch:sBE_1}--\eqref{sch:sBE_2} unconditionally preserves both the discrete MBP and the discrete energy dissipation law.

To establish the discrete MBP of the DsBE scheme \eqref{sch:sBE_1}--\eqref{sch:sBE_2}, the following useful lemma is presented.
\begin{lemma}[\cite{SIREV_Du_2021}]\label{lem:MBP_right}
	If $\kappa \geq \left\|f^{\prime}\right\|_{C[-1, 1]} = 2$ holds for some positive constant $\kappa$, then we have $|f( \xi ) + \kappa \xi | \leq \kappa $ for any $\xi \in[-1, 1]$.
\end{lemma}

\begin{theorem}[MBP of the DsBE scheme]\label{thm:MBP_sBE}
	Assume that $S_{1} \geq\left\|f^{\prime}\right\|_{C[-1, 1]}=2$, then
	the DsBE scheme \eqref{sch:sBE_1}--\eqref{sch:sBE_2} unconditionally preserves the discrete MBP, i.e.,
	\begin{equation*}\label{MBP:sBE}
		\left\|\phi_{\text {init }}\right\|_{\infty} \leq 1 \quad  \Longrightarrow \quad \left\|\phi^{n+1}\right\|_{\infty} \leq 1 \quad \text{for} \  n \geq 0.
	\end{equation*}
\end{theorem}
\begin{proof}
	This claim will be verified by a complete mathematical induction argument. Assume that $ \| \phi^{k} \|_{\infty} \leq 1 $ holds for $ 0 \leq k \leq n $. Inserting \eqref{sch:sBE_2} into \eqref{sch:sBE_1} gives
	\begin{equation}\label{MBP_sBE:1}
		\begin{aligned}
			& \left[ \f{ 1 }{ \tau_{n+1} } I + S_{1} \mm( \phi^{n} ) - \varepsilon^{2} \mm( \phi^{n} ) \Delta_{h} \right] \phi^{n+1} \\
			& \quad = \f{ 1 }{ \tau_{n+1} } \phi^{n} + \mm( \phi^{n} ) \big( f( \phi^{n} ) + S_{1} \phi^{n} \big).
		\end{aligned}
	\end{equation}
	
	We first show that $ \max_{ \mbf{x_{i}} \in \Omega_{h} } \phi^{n+1}_{ \mbf{i} } \leq 1 $. Without loss of generality,  Suppose that the numerical solution $ \phi^{n+1} $ achieves its maximum value at some $ \mbf{x_{j}} \in \Omega_{h} $, i.e., $
	\phi^{n+1}_{ \mbf{j} } = \max_{ \mbf{x_{i}} \in \Omega_{h} } \phi^{n+1}_{ \mbf{i} } $. Noting that $ \Delta_{h} \phi^{n+1}_\mbf{j} \leq 0 $, $\mm(\phi^{n}_\mbf{j})\geq0$ and Lemma \ref{lem:MBP_right}, we deduce from taking the spatial index $ \mbf{j} $ in \eqref{MBP_sBE:1} that
	\begin{align*}
		\left[ \f{ 1 }{ \tau_{n+1} } + S_{1} \mm( \phi^{n}_\mbf{j} ) \right] \phi^{n+1}_\mbf{j}
		\leq \f{ 1 }{ \tau_{n+1} } \phi^{n}_{ \mbf{j} } + S_{1} \mm( \phi^{n}_\mbf{j} )
		\leq \f{ 1 }{ \tau_{n+1} }  + S_{1} \mm( \phi^{n}_\mbf{j}),
	\end{align*}
	which implies that $ \phi^{n+1}_\mbf{j} \leq 1 $.
	
	Next, for the minimum, suppose that $\phi^{n+1}_{ \mbf{k} } = \min_{ \mbf{x_{i}} \in \Omega_{h} } \phi^{n+1}_{ \mbf{i} } $ at some $ \mbf{x_{k}} \in \Omega_{h} $. By the similar argument, we can show that $ \phi^{n+1}_\mbf{k} \geq -1 $. Consequently, it holds that $ \| \phi^{n+1} \|_{\infty} \leq 1$. This completes the proof.
\end{proof}

\begin{theorem}[Energy dissipation of the DsBE scheme]\label{thm:energy_1}
	If $S_{1} \geq\left\|f^{\prime}\right\|_{C[-1, 1]}=2$, the DsBE scheme \eqref{sch:sBE_1}--\eqref{sch:sBE_2} is unconditionally energy dissipative, i.e.,
	\begin{equation*}\label{EDL:sch1}
		E_h[\phi^{n+1}] - E_h[\phi^{n}] \leq - \tau_{n+1} \langle \mm(\phi^{n}) \mu^{n+1}, \mu^{n+1} \rangle, \quad n \geq 0.
	\end{equation*}
\end{theorem}
\begin{proof}
	Taking the inner products of \eqref{sch:sBE_1} and \eqref{sch:sBE_2} with $\tau_{n+1} \mu^{n+1} $ and  $\nabla_{\tau}\phi^{n+1}$  respectively, together with the relation $ 2 a ( a - b ) \ge a^2 - b^2 $, we obtain
	\begin{equation*}
		\begin{aligned}
			& - \f{\varepsilon^2}{2} \langle \Delta_{h} \phi^{n+1}, \phi^{n+1} \rangle + \f{\varepsilon^2}{2} \langle \Delta_{h} \phi^{n}, \phi^{n} \rangle - \langle f( \phi^{n} ), \nabla_{\tau} \phi^{n+1} \rangle + S_{1} \| \nabla_{\tau} \phi^{n+1} \|^{2} \\
			& \quad \le - \varepsilon^2 \langle \Delta_{h} \phi^{n+1}, \nabla_{\tau} \phi^{n+1} \rangle - \langle f( \phi^{n} ), \nabla_{\tau} \phi^{n+1} \rangle + S_{1} \| \nabla_{\tau} \phi^{n+1} \|^{2}\\
			& \quad  = - \tau_{n+1} \langle \mm(\phi^{n}) \mu^{n+1}, \mu^{n+1} \rangle.
		\end{aligned}
	\end{equation*}
	For any $ a , b \in [-1,1] $, it is easy to check that
	\begin{equation*}
		\begin{aligned}
			- f(b) ( a - b ) \geq F(a) - F( b ) - 2 ( a - b )^2.
		\end{aligned}
	\end{equation*}
	Thus, we arrive at
	\begin{equation*}
		\begin{aligned}
			E_{h}[\phi^{n+1}] - E_{h}[\phi^{n}] + ( S_{1} - 2 ) \| \nabla_{\tau} \phi^{n+1} \|^{2} \leq - \tau_{n+1} \langle \mm(\phi^{n}) \mu^{n+1}, \mu^{n+1} \rangle,
		\end{aligned}
	\end{equation*}
	which completes the proof.
\end{proof}

\subsection{Maximum-norm error estimate}
In this subsection, the maximum-norm error estimate of the DsBE scheme for the Allen--Cahn equation \eqref{Model:MAC1}--\eqref{Model:MAC2} with a general mobility is established.

\begin{theorem}[Error estimate of the DsBE scheme]\label{thm:Conver_sBE}
	Assume that $S_{1} \geq\left\|f^{\prime}\right\|_{C[-1, 1]}=2$ and $ \phi \in W^{2,1} ( J; L^{\infty}( \Omega ) ) \cap L^{\infty} ( J; W^{4,\infty}( \Omega ) ) $. Then, the solution of the DsBE scheme \eqref{sch:sBE_1}--\eqref{sch:sBE_2} is convergent in the maximum norm, i.e.,
	\begin{equation*}\label{Conclu:Cover_sBE}
		\| \phi( t_{n} ) - \phi^{n} \|_{\infty} \leq C ( \tau + h^2 ) \quad \text{for} \   n \geq 0.
	\end{equation*}
\end{theorem}
\begin{proof}
	Define the error function as $ e^{n} := \phi( t_{n} ) - \phi^{n} $. Subtracting \eqref{sch:sBE_1}--\eqref{sch:sBE_2} from \eqref{Model:MAC1}--\eqref{Model:MAC2}, one gets the following error equation:
	\begin{equation}\label{err:sBE_0}
		\begin{aligned}
			\delta_t e^{n+1}
			& =  \varepsilon^2 \mm( \phi^{n} )  \Delta_{h} e^{n+1} + \mm( \phi^{n} ) \big( f( \phi( t_{n} ) ) - f( \phi^{n} ) \big) \\
			& \quad +\big( \mm( \phi( t_{n} ) ) - \mm( \phi^{n} ) \big) \big( \varepsilon^2 \Delta_{h} \phi( t_{n+1} ) + f( \phi( t_{n} ) ) \big)
			\\
			& \quad
			+ S_{1} \mm( \phi^{n} )  ( \phi^{n+1} - \phi^{n} ) + \sum_{\ell=1}^3  \mt_{1,\ell}^{n+1},
		\end{aligned}
	\end{equation}
	or equivalently,
	\begin{equation}\label{err:sBE_1}
		\begin{aligned}
			& \left[ \f{ 1 }{ \tau_{n+1} } I + S_{1}  \mm( \phi^{n} ) - \varepsilon^{2}  \mm( \phi^{n} ) \Delta_{h} \right] e^{n+1} \\
			& = \left[ \f{ 1 }{ \tau_{n+1} } + S_{1}  \mm( \phi^{n} ) \right] e^{n} +  \mm( \phi^{n} ) \big( f( \phi( t_{n} ) ) - f( \phi^{n} ) \big) \\
			& \quad +  \big( \mm( \phi( t_{n} ) ) - \mm( \phi^{n} ) \big) \big( \varepsilon^2 \Delta_{h} \phi( t_{n+1} ) + f( \phi( t_{n} ) ) \big)
			\\
			& \quad
			+ S_{1} \mm( \phi^{n} )  ( \phi( t_{n+1} ) - \phi( t_{n} ) )  + \sum_{\ell=1}^3  \mt_{1,\ell}^{n+1},
		\end{aligned}
	\end{equation}
	where the truncation errors are given by
	$$
	\mt_{1,1}^{n+1} = \delta_t \phi( t_{n+1} ) - \phi_{t} ( t_{n+1} ), \quad \mt_{1,2}^{n+1} = \varepsilon^{2} \mm( \phi( t_{n} ) ) \big( \Delta \phi( t_{n+1} )  - \Delta_{h} \phi( t_{n+1} )  \big),
	$$
	$$
	\mt_{1,3}^{n+1} =  \mm( \phi( t_{n+1} ) ) \big( \varepsilon^{2}\Delta \phi( t_{n+1} )  + f( \phi( t_{n+1} ) )  \big) - \mm( \phi( t_{n} ) ) \big(  \varepsilon^{2}\Delta \phi( t_{n+1} )  + f( \phi( t_{n} ) )  \big).
	$$
	It is easy to verify that
	\begin{equation}\label{err:sBE_2}
		\begin{aligned}
			\| \mt_{1,1}^{n+1} \|_{\infty} &\leq C \| \phi \|_{ W^{2,1}(J_{n+1};L^{\infty}(\Omega)) }, \quad \| \mt_{1,2}^{n+1} \|_{\infty} \leq C h^{2} \| \phi \|_{ L^{\infty}(J;W^{4,\infty}(\Omega)) }, \\
			\| \mt_{1,3}^{n+1} \|_{\infty} &\leq C \| \phi \|_{ W^{1,1}(J_{n+1};L^{\infty}(\Omega)) }.
		\end{aligned}
	\end{equation}
	By the mean value theorem and the boundedness of $\phi^n$, we obtain
	\begin{equation}\label{err:sBE_4}
		\begin{aligned}
			& \| \big( \mm( \phi( t_{n} ) ) - \mm( \phi^{n} ) \big) \big( \varepsilon^2 \Delta_{h} \phi( t_{n+1} ) + f( \phi( t_{n} ) ) \big) \|_{\infty} \leq C  \| e^{n} \|_{\infty},\\
			& \| \mm( \phi^{n} ) \big( f( \phi( t_{n} ) ) - f( \phi^{n} ) \big) \|_{\infty} \leq C \| e^{n} \|_{\infty}, \\
			& \| S_{1} \mm( \phi^{n} )  ( \phi( t_{n+1} ) - \phi( t_{n} ) ) \|_{\infty} \leq C \| \phi \|_{ W^{1,1}(J_{n+1};L^{\infty}(\Omega)) }.
		\end{aligned}
	\end{equation}
	
	Then,  similar to the discussion in Theorem \ref{thm:MBP_sBE}, we assume that the error function $ e^{n+1} $ achieves its maximum value at $ \mbf{ x_{j} } \in \Omega_{h} $; that is, $ e^{n+1}_{ \mbf{j} } = \max_{ \mbf{i} \in \Omega_{h} } e^{n+1}_{ \mbf{i} } $. Taking the spatial index $ \mbf{j} $ in \eqref{err:sBE_1} and using the facts that $ \Delta_{h} e^{n+1}_\mbf{j} \leq 0 $ and $\mm(\phi^{n}_\mbf{j})\geq0$,  together with the estimates \eqref{err:sBE_2}--\eqref{err:sBE_4}, we have
	\begin{equation*}
		\begin{aligned}
			& \left[\frac{1}{\tau_{n+1}} + S_{1} \mm( \phi^{n}_\mbf{j} )\right] e^{n+1}_\mbf{j} \\
			& \leq \left[ \frac{1}{\tau_{n+1}} + S_{1} \mm( \phi^{n}_\mbf{j} ) \right] e^{n}_\mbf{j} + C  ( \| e^{n} \|_{\infty} + \| \phi \|_{ W^{2,1}(J_{n+1};L^{\infty}(\Omega)) } + h^{2} \| \phi \|_{ L^{\infty}(J;W^{4,\infty}(\Omega)) } ),
		\end{aligned}
	\end{equation*}
	which further implies
	\begin{equation*}
		\begin{aligned}
			e^{n+1}_\mbf{j} \leq e^{n}_\mbf{j} + C \tau_{n+1} \big( \| e^{n} \|_{\infty} + \| \phi \|_{ W^{2,1}(J_{n+1};L^{\infty}(\Omega)) } + h^{2} \| \phi \|_{ L^{\infty}(J;W^{4,\infty}(\Omega)) } \big).
		\end{aligned}
	\end{equation*}
	Similarly, supposing that the minimum value of $ e^{n+1} $ is attained at $ \mbf{ x_{k} } \in \Omega_{h} $ and taking the spatial index $ \mbf{k} $ in \eqref{err:sBE_1}, we can derive
	\begin{equation*}
		\begin{aligned}
			e^{n+1}_\mbf{k} \geq - e^{n}_\mbf{k} - C \tau_{n+1} \big( \| e^{n} \|_{\infty} + \| \phi \|_{ W^{2,1}(J_{n+1};L^{\infty}(\Omega)) } + h^{2} \| \phi \|_{ L^{\infty}(J;W^{4,\infty}(\Omega)) } \big).
		\end{aligned}
	\end{equation*}
	As a consequence, the error function $ e^{n+1} $ is bounded by
	\begin{equation}\label{err:sBE_6}
		\| e^{n+1} \|_{\infty}  \leq \| e^{n} \|_{\infty} + C \tau_{n+1} \big( \| e^{n} \|_{\infty} + h^{2} + \| \phi \|_{ W^{2,1}(J_{n+1};L^{\infty}(\Omega)) } \big).
	\end{equation}	
	Summing it up from $0$ to $n$ and together with the discrete Gr\"onwall's inequality yields the desired estimate.
\end{proof}

\section{Second-order stabilized Crank-Nicolson scheme}\label{Sec:CN}
In this section, we construct a second-order linear dynamically stabilized Crank-Nicolson (DsCN) scheme for model \eqref{Model:MAC1}--\eqref{Model:MAC2}. For given real sequence $ \{ v^{k} \}^{N}_{k=0} $, define the intermediate time level $ t_{k-\f{1}{2}} := (t_{k} + t_{k-1})/{2} $ and the mean value $ \bar{v}^{k-\f{1}{2}} := (v^{k} + v^{k-1})/{2}$ for $ k \geq 1 $. The DsCN scheme is then formulated as follows:

\paragraph{ \bf Step 1}
For $ n = 0 $, compute $ \phi^{1} $ using the DsBE scheme \eqref{sch:sBE_1}--\eqref{sch:sBE_2}, that is,
\begin{equation}\label{sch:2_1_1}
	\phi^{1} = \text{DsBE}( S_{1}, \phi^{0}, \mm(\phi^{0}), \tau_{1} ).
\end{equation}

\paragraph{ \bf Step 2} For $ n \geq 1 $,  given $ \phi^{n}$ and $ \phi^{n-1} $, we find $ \phi^{n+1} \in \mathbb{V}_{h} $ such that
\begin{empheq}{align}
	& \hat{\phi}^{n+\f{1}{2}} = \min\left\{  \max\left\{ \phi^{*,n}  , - 1 \right\} , 1 \right\} \mbox{ with } \phi^{*,n}:=( 1 + \f{r_{n+1}}{2}  ) \phi^{n} - \f{r_{n+1}}{2}  \phi^{n-1}, \label{sch:2_n_1} \\
	& \delta_{t} \phi^{n+1} = - \mm( \hat{\phi}^{n+\f{1}{2}} ) \mu^{n+\f{1}{2}} + \mathcal{S}_{2}^{n+\f{1}{2}} [\phi], \label{sch:2_n_2} \\
	& \mu^{n+\f{1}{2}} = - \varepsilon^2 \Delta_{h} \bar{\phi}^{n+\f{1}{2}} - f( \hat{\phi}^{n+\f{1}{2}} ) + \mathcal{S}_{1}^{n+\f{1}{2}} [\phi],  \label{sch:2_n_3}
\end{empheq}
where the dynamic stabilization terms are defined as $ \mathcal{S}_{1}^{n+\f{1}{2}} [\phi] := S_{1} \big( \bar{\phi}^{n+\f{1}{2}} - \hat{\phi}^{n+\f{1}{2}} \big) $ and $ \mathcal{S}_{2}^{n+\f{1}{2}} [\phi] := - S_{2} \tau_{n+1} ( \phi^{n+1} - \phi^{n} ) $ with $ S_{1}, S_{2} \geq 0 $ being two constant parameters.

\begin{remark}
	The first stabilization term $ \mathcal{S}_{1}^{n+\f{1}{2}} [\phi] $, introduced in \eqref{sch:2_n_3}, indicates that the scheme incorporates a dynamic stabilization term, namely
	$$
	- S_{1} \mm( \hat{\phi}^{n+\f{1}{2}} ) \big( \bar{\phi}^{n+\f{1}{2}} - \hat{\phi}^{n+\f{1}{2}} \big).
	$$
	This term is designed to dominate the nonlinear term and thereby ensure the discrete MBP. Moreover, its dynamic nature enables a mobility-robust energy stability, with the orginal discrete energy uniformly bounded by the initial energy up to a fourth-order perturbation (cf. Theorem \ref{thm:energy_2}). Alternatively, if a classical constant -- rather than dynamic -- stabilization term is employed, as suggested in \cite{MOC_2023_Ju,AAMM_2024_Hou,ANM_2025_Hou} (i.e., $ \mathcal{S}_{1}^{n+\f{1}{2}} [\phi] $ is added directly to \eqref{sch:2_n_2} instead of \eqref{sch:2_n_3}), the resulting scheme can still preserve the MBP, provided that condition \eqref{sch:sBE:e3} is satisfied. However, in such situation, the energy stability will additionally require a strictly positive lower bound assumption for the variable mobility function (see \eqref{Assu:M}), thereby ruling out degenerate mobility such as \eqref{Def:Mobi_1}. 
\end{remark}

\begin{remark}
	The second stabilization term $ \mathcal{S}_{2}^{n+\f{1}{2}} [\phi] $  in \eqref{sch:2_n_2} plays a key role in unconditionally preserving the discrete MBP (cf. Theorem \ref{thm:MBP_sCN}).
	It can also be regarded as a high-order perturbation of $ \delta_{t} \phi^{n+1} $, since \eqref{sch:2_n_2} is equivalent to
	\begin{equation*}\label{sch:2_n_2_2}
		( 1 + S_{2} \tau_{n+1}^{2} ) \delta_{t} \phi^{n+1} = - \mm( \hat{\phi}^{n+\f{1}{2}} ) \mu^{n+\f{1}{2}} \quad \text{or} \quad  \delta_{t} \phi^{n+1} = - \f{\mm( \hat{\phi}^{n+\f{1}{2}} )}{ 1 + S_{2} \tau_{n+1}^{2} } \mu^{n+\f{1}{2}}.
	\end{equation*}
	Thus, it follows that
	\begin{align*}
			( \mathcal{S}_{2}^{n+\f{1}{2}}[\phi], \mu^{n+\f{1}{2}} ) &= - \f{ S_{2} \tau_{n+1}^{2} }{1 + S_{2} \tau_{n+1}^{2}} ( \mm( \hat{\phi}^{n + \f{1}{2}} ) \mu^{n+\f{1}{2}}, \mu^{n+\f{1}{2}} ) \\
    &\leq  0 \leq ( \mm( \hat{\phi}^{n + \f{1}{2}} ) \mu^{n+\f{1}{2}}, \mu^{n+\f{1}{2}} ),
	\end{align*}
	which demonstrates that the inclusion of this stabilization term does not compromise the mobility-robust energy stability, see Theorem \ref{thm:energy_2} for details.
\end{remark}

\begin{remark}
	For the Allen–Cahn equation, second-order linear MBP-preserving schemes are often constructed by first employing an auxiliary first-order method to generate a predicted solution with local second-order accuracy while preserving the MBP \cite{MOC_2023_Ju,ANM_2025_Hou,SINUM_2022_Ju,JSC_2022_Ju}.
	This predicted solution is then used to explicitly handle the nonlinear term, which typically requires solving two linear systems per time level. In contrast, the proposed DsCN scheme obtains the predicted solution through a straightforward cut-off preprocessing step applied to the extrapolation solution. This approach enhances computational efficiency by requiring the solution of only one linear system per time level, while still preserving the MBP as shown below.
\end{remark}

\subsection{Discrete MBP}
Now, we are ready to establish the discerte MBP for the linear second-order DsCN scheme \eqref{sch:2_1_1}--\eqref{sch:2_n_3}.

\begin{theorem}[MBP of the DsCN scheme]\label{thm:MBP_sCN}
	Assume that $S_{1} \geq\left\|f^{\prime}\right\|_{C[-1, 1]}=2$. If $ S_{2} = 0 $ and the time stepsize $ \tau_{n+1} $ satisfies
	\begin{equation}\label{MBP:Condi_sCN}
		\tau_{n+1} \leq \f{ 2 }{ S_{1} K_{\mm} + 2 d  K_{\mm} \varepsilon^{2}/h^{2} } \quad \text{with} \ K_{\mm} := \max_{ \xi \in [-1,1] } \mm( \xi ),
	\end{equation}
	then the DsCN scheme \eqref{sch:2_1_1}--\eqref{sch:2_n_3} preserves the discrete MBP.
	Moreover, if $ S_{2} $ satisfies
	\begin{equation}\label{MBP:S2_sCN}
		S_{2} \geq \left( \f{S_{1}K_{\mm}}{4} + \f{ d\varepsilon^{2} K_{\mm} }{ 2 h^2 } \right)^{2},
	\end{equation}
	then the DsCN scheme \eqref{sch:2_1_1}--\eqref{sch:2_n_3} unconditionally preserves the discrete MBP.
\end{theorem}
\begin{proof}
	For $ n = 0 $, Theorem \ref{thm:MBP_sBE} directly yields $ \| \phi^{1} \|_{\infty} \leq 1 $. It remains to consider the case of $ n \geq 1 $. For any $ 1 \leq n \leq N-1 $, we assume that $ \| \phi^{k} \|_{\infty} \leq 1 $ holds for $ 0 \leq k \leq n $. From \eqref{sch:2_n_1}, we immediately have $ \| \hat{\phi}^{n+\f{1}{2}} \|_{\infty} \leq 1 $. Moreover, equations \eqref{sch:2_n_2}--\eqref{sch:2_n_3} can be equivalently rewritten as
	\begin{equation}\label{MBP_sCN:1}
		\begin{aligned}
			& \left[ \left( \f{1}{ \tau_{n+1} } + S_{2} \tau_{n+1} \right)I + \f{S_{1}}{2} \mm( \hat{\phi}^{n+\f{1}{2}} ) - \f{\varepsilon^2}{2} \mm( \hat{\phi}^{n+\f{1}{2}} ) \Delta_{h} \right] \phi^{n+1} \\
			& \quad = Q^{n} \phi^{n} + \mm( \hat{\phi}^{n+\f{1}{2}} ) \big( f( \hat{\phi}^{n+\f{1}{2}} ) + S_{1} \hat{\phi}^{n+\f{1}{2}} \big),
		\end{aligned}
	\end{equation}
	where the operator $Q^{n}$ is defined by
	\begin{equation}\label{MBP_sCN:1b}
		Q^{n}  = \left( \f{1}{ \tau_{n+1} } + S_{2} \tau_{n+1} \right) I  - \f{S_{1}}{2} \mm( \hat{\phi}^{n+\f{1}{2}} ) + \f{\varepsilon^2}{2} \mm( \hat{\phi}^{n+\f{1}{2}} ) \Delta_{h}.
	\end{equation}
	
	Suppose that the numerical solution $ \phi^{n+1} $ achieves its maximum value at $ \mbf{x_{j}} \in \Omega_{h} $. Taking the spatial index $ \mbf{j} $ in \eqref{MBP_sCN:1} and noting that $ \Delta_{h} \phi^{n+1}_\mbf{j} \leq 0 $ together with Lemma \ref{lem:MBP_right}, we get
	\begin{equation}\label{MBP_sCN:2}
		\left[ \f{1}{ \tau_{n+1} } + S_{2} \tau_{n+1}
		+ \f{S_{1}}{2} \mm( \hat{\phi}^{n+\f{1}{2}}_\mbf{j} ) \right] \phi^{n+1}_\mbf{j}
		\leq  Q_{\mbf{j}}^{n} \phi^{n}_\mbf{j}  + S_{1} \mm( \hat{\phi}^{n+\f{1}{2}}_\mbf{j} ).
	\end{equation}
	We then give the estimate of $Q_{\mbf{j}}^{n} \phi^{n}_\mbf{j}$ in \eqref{MBP_sCN:2}. It follows from the definition of $\Delta_h$ and the inductive hypothesis $ \| \phi^{n} \|_{\infty} \leq 1 $ that
	\begin{equation}\label{MBP_sCN:3a}	
		\f{ 2 d }{ h^2 } ( - 1 - \phi^{n}_{ \mbf{j} } )
		\le \f{ 2 d }{ h^2 } ( - \| \phi^{n} \|_{\infty} - \phi^{n}_{ \mbf{j} } )
		\le \Delta_{h} \phi^{n}_{ \mbf{j} }
		\le \f{ 2 d }{ h^2 } ( \| \phi^{n} \|_{\infty} - \phi^{n}_{ \mbf{j} } )
		\le \f{ 2 d }{ h^2 } ( 1 - \phi^{n}_{ \mbf{j} } ).
	\end{equation}
	Then, together with the definition of $Q^{n}$ in \eqref{MBP_sCN:1b}, we obtain
	\begin{equation*}\label{MBP_sCN:3b}	
		\begin{aligned}
			Q_{\mbf{j}}^{n} \phi^{n}_\mbf{j} \leq \left[ \f{1}{ \tau_{n+1} } + S_{2} \tau_{n+1} - \f{S_{1}}{2} \mm( \hat{\phi}^{n+\f{1}{2}}_\mbf{j} ) - \f{d \varepsilon^{2}}{ h^{2} } \mm( \hat{\phi}^{n+\f{1}{2}}_\mbf{j} ) \right] \phi^{n}_\mbf{j} + \f{d \varepsilon^{2}}{ h^{2} } \mm( \hat{\phi}^{n+\f{1}{2}}_\mbf{j} ).
		\end{aligned}
	\end{equation*}
	If $ S_{2} = 0 $ and  the time-step condition \eqref{MBP:Condi_sCN}, or $S_{2}$ satisfies \eqref{MBP:S2_sCN},  we have
	\begin{align*}
	&\f{1}{ \tau_{n+1} } + S_{2} \tau_{n+1} - \f{S_{1}}{2} \mm( \hat{\phi}^{n+\f{1}{2}}_\mbf{j} ) - \f{d \varepsilon^{2}}{ h^{2} } \mm( \hat{\phi}^{n+\f{1}{2}}_\mbf{j} ) \\
    &\geq 2 \sqrt{ S_{2} } - \f{S_{1}}{2} \mm( \hat{\phi}^{n+\f{1}{2}}_\mbf{j} ) - \f{d \varepsilon^{2}}{ h^{2} } \mm( \hat{\phi}^{n+\f{1}{2}}_\mbf{j} ) \geq 0,
	\end{align*}
	which together with $\phi^{n}_\mbf{j}\leq1$ gives
	\begin{equation}\label{MBP_sCN:5}
		Q_{\mbf{j}}^{n} \phi^{n}_\mbf{j} \leq \f{1}{ \tau_{n+1} } + S_{2} \tau_{n+1} - \f{S_{1}}{2} \mm( \hat{\phi}^{n+\f{1}{2}}_\mbf{j} ).
	\end{equation}
	By inserting \eqref{MBP_sCN:5} into \eqref{MBP_sCN:2}, it yields $ \phi^{n+1}_{\mbf{j}} \leq 1 $.
	
	Similarly, if the minimum value of $ \phi^{n+1} $ is achieved at $ \mbf{x_{k}} \in \Omega_{h} $, we have
	\begin{equation}\label{MBP_sCN:7}
		\begin{aligned}
			\left[ \f{1}{ \tau_{n+1} } + S_{2} \tau_{n+1} + \f{S_{1}}{2} \mm( \hat{\phi}^{n+\f{1}{2}}_\mbf{k} ) \right] \phi^{n+1}_\mbf{k} \geq  Q_{\mbf{k}}^{n} \phi^{n}_\mbf{k} - S_{1} \mm( \hat{\phi}^{n+\f{1}{2}}_\mbf{k} ).
		\end{aligned}
	\end{equation}
	When $ S_{2} = 0 $ and $ \tau_{n+1} $ satisfy \eqref{MBP:Condi_sCN}, or  $ S_{2}  $ satisfies \eqref{MBP:S2_sCN}, we can also derive that
	\begin{equation*}
		\begin{aligned}
			Q_{\mbf{k}}^{n} \phi^{n}_\mbf{k} & \geq \left[ \f{1}{ \tau_{n+1} } + S_{2} \tau_{n+1} - \f{S_{1}}{2} \mm( \hat{\phi}^{n+\f{1}{2}}_\mbf{k} ) - \f{d \varepsilon^{2}}{ h^{2} } \mm( \hat{\phi}^{n+\f{1}{2}}_\mbf{k} ) \right] \phi^{n}_\mbf{k} - \f{d \varepsilon^{2}}{ h^{2} } \mm( \hat{\phi}^{n+\f{1}{2}}_\mbf{k} ) \\
			& \geq - \f{1}{ \tau_{n+1} } - S_{2} \tau_{n+1} + \f{S_{1}}{2} \mm( \hat{\phi}^{n+\f{1}{2}}_\mbf{k} ),
		\end{aligned}
	\end{equation*}
	where \eqref{MBP_sCN:3a} and $ \| \phi^{n} \|_{\infty} \leq 1 $ have been used. Inserting this estimate in to \eqref{MBP_sCN:7}, one gets $ \phi^{n+1}_\mbf{k} \geq -1 $. As a result, the discrete MBP is proved.
\end{proof}

\begin{remark}
	The conditions \eqref{MBP:Condi_sCN} and \eqref{MBP:S2_sCN} imply that the admissible choice of the time stepsize $ \tau_{n+1} $ or the stabilization parameter $ S_{2} $ is constrained by the spatial mesh size $ h $. For the Allen--Cahn equation, the parameter $ \varepsilon \ll 1 $ characterizes the width of the diffusive interface. In practice, one usually selects a spatial resolution $ h = \mo(\varepsilon) $ to accurately capture the interface dynamics. Hence, the mesh-dependent restrictions on $ \tau_{n+1} $ or $ S_{2} $ prescribed by \eqref{MBP:Condi_sCN} or \eqref{MBP:S2_sCN} are practically reasonable.
\end{remark}

\subsection{Maximum-norm error estimate}
We establish the maximum-norm error estimate of the DsCN scheme \eqref{sch:2_1_1}--\eqref{sch:2_n_3} for the Allen--Cahn equation \eqref{Model:MAC1}--\eqref{Model:MAC2} with a general mobility as follows.
\begin{theorem}[Error estimate of the DsCN scheme]\label{thm:Conver_sCN}
	Assume that the conditions in Theorem \ref{thm:MBP_sCN} hold and $ \phi \in W^{3,1} ( J; L^{\infty}( \Omega ) ) \cap W^{2,1}(J;W^{2,\infty}(\Omega)) \cap L^{\infty} (J; W^{4,\infty}( \Omega ) ) $. Then, the solution of the DsCN scheme \eqref{sch:2_1_1}--\eqref{sch:2_n_3} is convergent in the maximum norm, that is,
	\begin{equation*}\label{Conclu:Cover_sCN}
		\| \phi( t_{n} ) - \phi^{n} \|_{\infty} \leq C ( \tau^{2} + h^2 ) \quad \text{for} \   n \geq 0.
	\end{equation*}
\end{theorem}
\begin{proof}
	Note that a useful estimate \eqref{err:sBE_6} for $ n = 0 $ indicates that
	\begin{equation}\label{err:sCN_1}
		\| e^{1} \|_{\infty}  \leq C \tau_{1} \big( \tau_{1} + h^2 \big).
	\end{equation}
	For $ n \geq 1 $, we subtract \eqref{sch:2_n_2}--\eqref{sch:2_n_3} from \eqref{Model:MAC1}--\eqref{Model:MAC2} to derive the error equation:
	\begin{equation}\label{err:sCN_2}
		\begin{aligned}
			\delta_t e^{n+1}
			& = \varepsilon^2 \mm( \hat{\phi}^{n+\f{1}{2}} ) \Delta_{h} \bar{e}^{n+\f{1}{2}} + \mm( \hat{\phi}^{n+\f{1}{2}} ) \big(  f( \phi( t_{n+\f{1}{2}} ) ) - f( \hat{\phi}^{n+\f{1}{2}} ) \big) \\
			& \quad +\big( \mm( \phi( t_{n+\f{1}{2}} ) ) - \mm( \hat{\phi}^{n+\f{1}{2}} ) \big) \left( \varepsilon^2 \Delta_{h} \bar{\phi}( t_{n+\f{1}{2}} ) + f( \phi( t_{n+\f{1}{2}} ) ) \right) \\
			& \quad
			+ S_{1} \mm( \hat{\phi}^{n+\f{1}{2}} )  ( \bar{\phi}^{n+\f{1}{2}} - \hat{\phi}^{n+\f{1}{2}} ) + S_{2} \tau_{n+1} ( \phi^{n+1} - \phi^{n} ) \\
			& \quad
			+ \mt_{2,1}^{n+1} + \mt_{2,2}^{n+1},
		\end{aligned}
	\end{equation}
	where $ \bar{e}^{n+\f{1}{2}} := \phi( t_{n+\f{1}{2}} ) - \bar{\phi}^{n+\f{1}{2}} $ and the truncation errors are given by
	$$
	\mt_{2,1}^{n+1} = \delta_t \phi( t_{n+1} ) - \phi_{t} ( t_{n+\f{1}{2}} ), \quad \mt_{2,2}^{n+1} = \varepsilon^{2} \mm( \phi( t_{n+\f{1}{2}} ) ) \left( \Delta \phi( t_{n+\f{1}{2}} ) - \Delta_{h} \bar{\phi}( t_{n+\f{1}{2}} ) \right).
	$$
	Moreover, we have following estimates for the truncation errors
	\begin{equation}\label{err:sCN_3}
		\begin{aligned}
			\| \mt_{2,1}^{n+1} \|_{\infty} &\leq C \tau_{n+1} \| \phi \|_{ W^{3,1}(J_{n+1};L^{\infty}(\Omega)) },\\
			\| \mt_{2,2}^{n+1} \|_{\infty} &\leq C \big( \tau_{n+1} \| \phi \|_{ W^{2,1}(J_{n+1};W^{2,\infty}(\Omega)) } + h^{2} \| \phi \|_{ L^{\infty}(J;W^{4,\infty}(\Omega)) } \big).
		\end{aligned}
	\end{equation}
	Observe that
	\begin{equation}\label{err:phi_phihat}
		\bar{\phi}^{n+\f{1}{2}} - \hat{\phi}^{n+\f{1}{2}} = - \bar{e}^{n+\f{1}{2}} + \bar{\phi}( t_{n+\f{1}{2}} ) - \phi( t_{n+\f{1}{2}} ) + \hat{e}^{n+\f{1}{2}},
	\end{equation}
	with $ \hat{e}^{n+\f{1}{2}} := \phi( t_{n+\f{1}{2}} ) - \hat{\phi}^{n+\f{1}{2}}.$
	Then, substituting \eqref{err:phi_phihat} into \eqref{err:sCN_2} yields
	\begin{equation}\label{err:sCN_4}
		\begin{aligned}
			& \left[ \left( \frac{1}{\tau_{n+1}} + S_{2} \tau_{n+1} \right) I + \f{S_{1}}{2} \mm( \hat{\phi}^{n+\f{1}{2}} ) - \f{\varepsilon^{2}}{2} \mm( \hat{\phi}^{n+\f{1}{2}} ) \Delta_{h} \right] e^{n+1} \\
			& = Q^{n} e^{n} + \big( \mm( \phi( t_{n+\f{1}{2}} ) ) - \mm( \hat{\phi}^{n+\f{1}{2}} ) \big) \big( \varepsilon^2 \Delta_{h} \bar{\phi}( t_{n+\f{1}{2}} ) + f( \phi( t_{n+\f{1}{2}} ) ) \big) \\
			& \quad +  \mm( \hat{\phi}^{n+\f{1}{2}} ) \big( f( \phi( t_{n+\f{1}{2}} ) ) - f( \hat{\phi}^{n+\f{1}{2}} ) \big)
			+ S_{2} \tau_{n+1} ( \phi( t_{n+1} ) - \phi( t_{n} ) ) \\
			& \quad + S_{1} \mm( \hat{\phi}^{n+\f{1}{2}} )  \big( \bar{\phi}( t_{n+\f{1}{2}} ) - \phi( t_{n+\f{1}{2}} ) + \hat{e}^{n+\f{1}{2}} \big)
			+ \mt_{2,1}^{n+1} + \mt_{2,2}^{n+1}.
		\end{aligned}
	\end{equation}
	For the right-hand side terms of \eqref{err:sCN_4}, by applying the mean value theorem and the discrete MBP law, it holds that
	\begin{equation*}\label{err:sCN_5}
		\begin{aligned}
			& \| \big( \mm( \phi( t_{n+\f{1}{2}} ) ) - \mm( \hat{\phi}^{n+\f{1}{2}} ) \big) \big( \varepsilon^2 \Delta_{h} \bar{\phi}( t_{n+\f{1}{2}} ) + f( \phi( t_{n+\f{1}{2}} ) ) \big) \|_{\infty} \leq C \| \hat{e}^{n+\f{1}{2}} \|_{\infty},\\
			& \| \mm( \hat{\phi}^{n+\f{1}{2}} ) \big( f( \phi( t_{n+\f{1}{2}} ) ) - f( \hat{\phi}^{n+\f{1}{2}} ) \big) \|_{\infty} \leq C \| \hat{e}^{n+\f{1}{2}} \|_{\infty}, \\
			& \|  S_{2} \tau_{n+1} ( \phi( t_{n+1} ) - \phi( t_{n} ) ) + S_{1} \mm( \hat{\phi}^{n+\f{1}{2}} )  \big( \bar{\phi}( t_{n+\f{1}{2}} ) - \phi( t_{n+\f{1}{2}} ) \|_{\infty} \\
			& \leq C \tau_{n+1} ( \| \phi \|_{ W^{1,1}(J_{n+1};L^{\infty}(\Omega)) } + \| \phi \|_{ W^{2,1}(J_{n+1};L^{\infty}(\Omega)) } ).
		\end{aligned}
	\end{equation*}
	
	Assume that the maximum value of $ e^{n+1} $ is achieved at $ \mbf{ x_{j} } \in \Omega_{h} $, i.e., $ e^{n+1}_{ \mbf{j} } = \max_{ \mbf{i} \in \Omega_{h} } e^{n+1}_{ \mbf{i} } $. By inserting \eqref{err:sCN_3} and the above estimates into \eqref{err:sCN_4}, and noting $ \Delta_{h} \phi^{n+1}_\mbf{j} \leq 0 $,  one gets
	\begin{equation*}
		\begin{aligned}
			& \left[ \frac{1}{\tau_{n+1}} + S_{2} \tau_{n+1} + \f{S_{1}}{2} \mm( \hat{\phi}^{n+\f{1}{2}}_{ \mbf{j} } )
			\right] e^{n+1}_{\mbf{j}} \\
			& \leq Q^{n}_{ \mbf{j} } e^{n}_{\mbf{j}} + C ( \| \hat{e}^{n+\f{1}{2}} \|_{\infty} + h^{2} \| \phi \|_{ L^{\infty}(J;W^{4,\infty}(\Omega)) } ) \\
			& \quad  + C \tau_{n+1} ( \| \phi \|_{ W^{3,1}(J_{n+1};L^{\infty}(\Omega)) } + \| \phi \|_{ W^{2,1}(J_{n+1};W^{2,\infty}(\Omega)) } ).
		\end{aligned}
	\end{equation*}
	Furthermore, under the condition \eqref{MBP:Condi_sCN} or \eqref{MBP:S2_sCN}, we can deduce a similar result of \eqref{MBP_sCN:3b} that
	\begin{equation*}
		\begin{aligned}
			Q_{\mbf{j}}^{n} e^{n}_\mbf{j} & \leq \left[ \f{1}{ \tau_{n+1} } + S_{2} \tau_{n+1} - \f{S_{1}}{2} \mm( \hat{\phi}^{n+\f{1}{2}}_\mbf{j} ) - \f{d \varepsilon^{2}}{ h^{2} } \mm( \hat{\phi}^{n+\f{1}{2}}_\mbf{j} ) \right] e^{n}_\mbf{j}
			+ \f{d \varepsilon^{2}}{ h^{2} } \mm( \hat{\phi}^{n+\f{1}{2}}_\mbf{j} ) \| e^{n} \|_{\infty} \\
			& \leq \left[ \f{1}{ \tau_{n+1} } + S_{2} \tau_{n+1} - \f{S_{1}}{2} \mm( \hat{\phi}^{n+\f{1}{2}}_\mbf{j} ) \right] \| e^{n} \|_{\infty},
		\end{aligned}
	\end{equation*}
	where the property \eqref{MBP_sCN:3a} of $ \Delta_{h} $ has been used. Furthermore, this yields
	\begin{equation*}
		\begin{aligned}
			& \left[ \frac{1}{\tau_{n+1}} + S_{2} \tau_{n+1} + \f{S_{1}}{2} \mm( \hat{\phi}^{n+\f{1}{2}}_{ \mbf{j} } ) \right] e^{n+1}_{\mbf{j}} \\
			& \leq \left[ \frac{1}{\tau_{n+1}} + S_{2} \tau_{n+1} - \f{S_{1}}{2} \mm( \hat{\phi}^{n+\f{1}{2}}_{ \mbf{j} } ) \right] \| e^{n} \|_{\infty}  + C ( \| \hat{e}^{n+\f{1}{2}} \|_{\infty} + h^{2} \| \phi \|_{ L^{\infty}(J;W^{4,\infty}(\Omega)) } ) \\
			& \quad  + C \tau_{n+1} ( \| \phi \|_{ W^{3,1}(J_{n+1};L^{\infty}(\Omega)) } + \| \phi \|_{ W^{2,1}(J_{n+1};W^{2,\infty}(\Omega)) } )
		\end{aligned}
	\end{equation*}
	which implies
	\begin{equation*}
		\begin{aligned}
			e^{n+1}_{\mbf{j}} &\leq \| e^{n} \|_{\infty} + C \tau_{n+1} [ \| \hat{e}^{n+\f{1}{2}} \|_{\infty} + h^{2} \| \phi \|_{ L^{\infty}(J;W^{4,\infty}(\Omega)) } \\
			& \quad +\tau^{2}( \| \phi \|_{ W^{3,1}(J_{n+1};L^{\infty}(\Omega)) } + \| \phi \|_{ W^{2,1}(J_{n+1};W^{2,\infty}(\Omega)) } )].
		\end{aligned}
	\end{equation*}
	Similarly, the minimum value $ e^{n+1}_{ \mbf{k} } = \min_{ \mbf{x_i} \in \Omega_{h} } e^{n+1}_{ \mbf{i} } $ of $ e^{n+1} $ can be estimated by
	\begin{equation*}
		\begin{aligned}
			e^{n+1}_{\mbf{k}} & \geq - \| e^{n} \|_{\infty} - C \tau_{n+1} [ \| \hat{e}^{n+\f{1}{2}} \|_{\infty} + h^{2} \| \phi \|_{ L^{\infty}(J;W^{4,\infty}(\Omega)) }  \\
			& \quad -  \tau^{2} ( \| \phi \|_{ W^{3,1}(J_{n+1};L^{\infty}(\Omega)) } + \| \phi \|_{ W^{2,1}(J_{n+1};W^{2,\infty}(\Omega)) } )],
		\end{aligned}
	\end{equation*}
	which further implies
	\begin{equation}\label{err:sCN_10}
		\begin{aligned}
			\| e^{n+1} \|_{\infty} &\leq \| e^{n} \|_{\infty} + C \tau_{n+1} [ \| \hat{e}^{n+\f{1}{2}} \|_{\infty} + h^{2} \| \phi \|_{ L^{\infty}(J;W^{4,\infty}(\Omega)) } \\
			&\quad+ \tau^{2} ( \| \phi \|_{ W^{3,1}(J_{n+1};L^{\infty}(\Omega)) } + \| \phi \|_{ W^{2,1}(J_{n+1};W^{2,\infty}(\Omega)) } )].
		\end{aligned}
	\end{equation}
	
	Since the cut-off operator in \eqref{sch:2_n_1} is contractive (due to the MBP-preserving property of the exact solution), it follows that
	\begin{equation}\label{err:sCN_11}
		\begin{aligned}
			\| \hat{e}^{n+\f{1}{2}} \|_{\infty} & \leq \| \phi( t_{ n+ \f{1}{2} } ) - \left( 1 + \f{r_{n+1}}{2} \right) \phi ( t_{n} ) + \f{r_{n+1}}{2} \phi ( t_{n-1} ) \|_{\infty} \\
			& \quad + \left( 1 + \f{r_{n+1}}{2} \right) \| e^{n} \|_{\infty} + \f{r_{n+1}}{2} \| e^{n-1} \|_{\infty} \\
			& \leq C ( \| e^{n} \|_{\infty} + \| e^{n-1} \|_{\infty} + \tau^{2} \| \phi \|_{ W^{2,1}(J_{n}\cup J_{n+1};L^{\infty}(\Omega)) } ).
		\end{aligned}
	\end{equation}
	Thus, inserting \eqref{err:sCN_11} into \eqref{err:sCN_10}, one further get
	\begin{equation*}
		\begin{aligned}
			\| e^{n+1} \|_{\infty} & \leq \| e^{n} \|_{\infty} + C \tau_{n+1} [ \| e^{n} \|_{\infty} + \| e^{n-1} \|_{\infty} + h^{2} \| \phi \|_{ L^{\infty}(J;W^{4,\infty}(\Omega)) } \\
			& \quad + \tau^{2} ( \| \phi \|_{ W^{3,1}(J_{n}\cup J_{n+1};L^{\infty}(\Omega)) } + \| \phi \|_{ W^{2,1}(J_{n+1};W^{2,\infty}(\Omega)) } )].
		\end{aligned}
	\end{equation*}
	Summing up the above inequality from 1 to $n$ gives
	\begin{equation*}
		\begin{aligned}
			\| e^{n+1} \|_{\infty} \leq \| e^{1} \|_{\infty} + C \sum^{n}_{k=1} \tau_{k+1} ( \| e^{k} \|_{\infty} + \| e^{k-1} \|_{\infty}  ) + C ( \tau^{2} + h^{2} ).
		\end{aligned}
	\end{equation*}
	Then, we apply the discrete Gr\"onwall inequality and the estimate \eqref{err:sCN_1} to obtain
	\begin{equation*}
		\begin{aligned}
			\| e^{n+1} \|_{\infty} \leq  C ( \| e^{1} \|_{\infty} + \tau^{2} + h^{2} ) \leq C ( \tau^{2} + h^{2} ) \quad \text{for} \  n \geq 1.
		\end{aligned}
	\end{equation*}
	Thus, the proof is completed.
\end{proof}

\begin{remark}\label{rem:bar_hat} Based upon the final estimate in Theorem \ref{thm:Conver_sCN}, an improved error estimate for $ \hat{e}^{ n + \f{1}{2} } $ in \eqref{err:sCN_11} can be followed
	\begin{equation*}
		\| \hat{e}^{ n + \f{1}{2} } \|_{\infty} \leq C ( \tau^{2} + h^{2} ) \quad \text{for} \   n \geq 1,
	\end{equation*}
	which indicates that $ \hat{\phi}^{n+\frac{1}{2}} $ obtained in the prediction step also possesses second-order accuracy.
	This result, together with \eqref{err:phi_phihat}, further implies that
	\begin{equation}\label{err:phi_phihat2}
		\| \bar{\phi}^{n+\f{1}{2}} - \hat{\phi}^{n+\f{1}{2}} \|_{\infty} \leq C ( \tau^{2} + h^{2} )\quad \text{for} \   n \geq 1.
	\end{equation}
\end{remark}

\subsection{Discrete energy stability}
In this subsection, we establish the energy stability of the DsCN scheme \eqref{sch:2_1_1}--\eqref{sch:2_n_3},
by making use of the discrete MBP (cf. Theorem \ref{thm:MBP_sCN}) together with the maximum-norm error estimate (cf. Theorem \ref{thm:Conver_sCN}).
\begin{theorem}[Energy stability of the DsCN scheme]\label{thm:energy_2}
	Assume that the conditions in Theorem \ref{thm:Conver_sCN} hold. Moreover, if $ \tau_{n+1} \leq \min\left\{ 1, \f{1}{4 K_{\mm} ( 1 + S_{2} ) } \right\} $, the DsCN scheme \eqref{sch:2_1_1}--\eqref{sch:2_n_3} is energy-stable in the sense that
	\begin{equation}\label{EDL:sch2_1}
		E_h[\phi^{1}] \le E_h[\phi^{0}]  \quad \mbox{and} \quad E_h[\phi^{n+1}] - E_h[\phi^{n}] \leq C \tau_{n+1} ( \tau^{4} + h^{4} )\  \  \  \text{for} \   n \geq 1,
	\end{equation}
	and consequently,
	\begin{equation}\label{EDL:sch2_2}
		E_h[\phi^{n+1}] \le E_h[\phi^{0}] + C T ( \tau^{4} + h^{4} )\  \  \  \text{for} \   n \geq 1,
	\end{equation}
	i.e., the discrete original energy is uniformly bounded by the initial energy plus a constant of order $ \mo ( \tau^{4} + h^{4} ) $.
\end{theorem}
\begin{proof} Note that \eqref{EDL:sch2_2} is a direct consequent of \eqref{EDL:sch2_1} by summing over $ n $. Therefore, we only pay attention to the proof of \eqref{EDL:sch2_1}.
	
	First, Theorem \ref{thm:energy_1} directly implies that $ E_h[\phi^{1}] \leq E_h[\phi^{0}] $. Next, for $ n \geq 1 $, we take the inner product of \eqref{sch:2_n_2}--\eqref{sch:2_n_3} with $\tau_{n+1} \mu^{n+\f{1}{2}} $ and  $\nabla_{\tau}\phi^{n+1}=\phi^{n+1}-\phi^{n}$, respectively, to get
	\begin{equation}\label{energy:0}
		\begin{aligned}
			&- \varepsilon^2 \big\langle \Delta_{h} \bar{\phi}^{n+\f{1}{2}}, \nabla_{\tau} \phi^{n+1} \big\rangle - \big\langle f( \hat{\phi}^{n+\f{1}{2}} ) , \nabla_{\tau} \phi^{n+1} \big\rangle \\
			& \qquad + S_{1} \big\langle \bar{\phi}^{n+\f{1}{2}} - \hat{\phi}^{n+\f{1}{2}},  \nabla_{\tau} \phi^{n+1} \big\rangle \\
			& = \langle \nabla_{\tau} \phi^{n+1}, \mu^{n+\f{1}{2}} \rangle = - \f{ \tau_{n+1} }{ 1 + S_{2} \tau_{n+1}^{2} } \big\langle \mm(\hat{\phi}^{n+\f{1}{2}}) \mu^{n+\f{1}{2}}, \mu^{n+\f{1}{2}} \big\rangle,
		\end{aligned}
	\end{equation}
	which implies
	\begin{equation}\label{energy:1}
		\begin{aligned}
			& - \f{\varepsilon^2}{2} \big\langle \Delta_{h} \phi^{n+1}, \phi^{n+1} \big\rangle + \f{\varepsilon^2}{2} \big\langle \Delta_{h} \phi^{n}, \phi^{n} \big\rangle - \big\langle f( \bar{\phi}^{n+\f{1}{2}} ), \nabla_{\tau} \phi^{n+1} \big\rangle \\
			& =  \big\langle f( \hat{\phi}^{n+\f{1}{2}} ) - f( \bar{\phi}^{n+\f{1}{2}} ), \nabla_{\tau} \phi^{n+1} \big\rangle - S_{1} \big\langle \bar{\phi}^{n+\f{1}{2}} - \hat{\phi}^{n+\f{1}{2}},  \nabla_{\tau} \phi^{n+1} \big\rangle \\
			& \qquad - \f{ \tau_{n+1} }{ 1 + S_{2} \tau_{n+1}^{2} } \big\langle \mm(\hat{\phi}^{n+\f{1}{2}}) \mu^{n+\f{1}{2}}, \mu^{n+\f{1}{2}} \big\rangle.
		\end{aligned}
	\end{equation}
	
	For any $ a , b \in [-1,1] $, it is easy to check
	\begin{equation*}\label{energy:l3_1}
		\f{1}{4} \left[ ( a^2 - 1 )^2 - ( b^2 - 1 )^2 \right] \leq \left[ \left( \f{a+b}{2} \right)^3 - \left( \f{a+b}{2} \right) \right] ( a - b ) + 2 ( a - b )^2,
	\end{equation*}
	which gives us
	\begin{equation}\label{energy:l3_2}
		\begin{aligned}
			- \big\langle f(\bar{\phi}^{n+\f{1}{2}}), \nabla_{\tau} \phi^{n+1} \big\rangle \geq \langle F( \phi^{n+1} ) - F( \phi^{n} ), 1 \rangle - 2 \| \nabla_{\tau} \phi^{n+1} \|^2.
		\end{aligned}
	\end{equation}
	Moreover, it follows from \eqref{sch:2_n_2} that $ \nabla_{\tau} \phi^{n+1}= - \f{ \tau_{n+1} }{ 1 + S_{2} \tau_{n+1}^{2} }  \mm(\hat{\phi}^{n+\f{1}{2}}) \mu^{n+\f{1}{2}}$. Thus,
	\begin{equation}
		\begin{aligned}
			\| \nabla_{\tau} \phi^{n+1} \|^{2}
			\le\tau_{n+1}^{2} \| \mm(\hat{\phi}^{n+\f{1}{2}}) \mu^{n+\f{1}{2}}\|^2
			\leq K_{\mm} \tau_{n+1}^{2} \big\langle \mm(\hat{\phi}^{n+\f{1}{2}}) \mu^{n+\f{1}{2}}, \mu^{n+\f{1}{2}} \big\rangle,
		\end{aligned}
	\end{equation}
	and the first two terms on the right-hand side of \eqref{energy:1} can be estimated by
	\begin{equation}\label{energy:r1_1}
		\begin{aligned}
			& \big\langle f( \hat{\phi}^{n+\f{1}{2}} ) - f( \bar{\phi}^{n+\f{1}{2}} ), \nabla_{\tau} \phi^{n+1} \big\rangle - S_{1} \big\langle \bar{\phi}^{n+\f{1}{2}} - \hat{\phi}^{n+\f{1}{2}},  \nabla_{\tau} \phi^{n+1} \big\rangle \\
			& = - \f{ \tau_{n+1} }{ 1 + S_{2} \tau_{n+1}^{2} } \big\langle f( \hat{\phi}^{n+\f{1}{2}} ) - f( \bar{\phi}^{n+\f{1}{2}} ), \mm( \hat{\phi}^{n+\f{1}{2}} ) \mu^{n+\f{1}{2}} \big\rangle
			\\
			& \quad + \f{ S_{1} \tau_{n+1} }{ 1 + S_{2} \tau_{n+1}^{2} } \big\langle \bar{\phi}^{n+\f{1}{2}} - \hat{\phi}^{n+\f{1}{2}},  \mm( \hat{\phi}^{n+\f{1}{2}} ) \mu^{n+\f{1}{2}} \big\rangle\\
			& \leq \f{ 2 S_{1} \tau_{n+1} }{ 1 + S_{2} \tau_{n+1}^{2} } \| M^{\f{1}{2}}( \hat{\phi}^{n+\f{1}{2}} ) ( \bar{\phi}^{n+\f{1}{2}} - \hat{\phi}^{n+\f{1}{2}} ) \| \| M^{\f{1}{2}}( \hat{\phi}^{n+\f{1}{2}} ) \mu^{n+\f{1}{2}} \| \\
			& \leq C \tau_{n+1} \| \bar{\phi}^{n+\f{1}{2}} - \hat{\phi}^{n+\f{1}{2}} \|^{2} + \f{\tau_{n+1}}{ 2 ( 1 + S_{2} \tau_{n+1}^{2} ) } \big\langle \mm(\hat{\phi}^{n+\f{1}{2}}) \mu^{n+\f{1}{2}}, \mu^{n+\f{1}{2}} \big\rangle,
		\end{aligned}
	\end{equation}
	where the mean value theorem and Cauchy-Schwarz inequality have been used in the estimate
	\begin{align*}
			&\big\langle f( \hat{\phi}^{n+\f{1}{2}} ) - f( \bar{\phi}^{n+\f{1}{2}} ), \mm( \hat{\phi}^{n+\f{1}{2}} ) \mu^{n+\f{1}{2}} \big\rangle\\
			&\le S_{1}  \| M^{\f{1}{2}}( \hat{\phi}^{n+\f{1}{2}} ) ( \bar{\phi}^{n+\f{1}{2}} - \hat{\phi}^{n+\f{1}{2}} ) \| \| M^{\f{1}{2}}( \hat{\phi}^{n+\f{1}{2}} ) \mu^{n+\f{1}{2}} \|
	\end{align*}
	for $ S_{1} \geq\left\|f^{\prime}\right\|_{C[-1, 1]} $, $ \| \bar{\phi}^{n+\f{1}{2}} \|_{\infty} \leq 1 $ and $ \| \hat{\phi}^{n+\f{1}{2}} \|_{\infty} \leq 1$.
	Then, inserting these estimates \eqref{err:phi_phihat2} and \eqref{energy:l3_2}--\eqref{energy:r1_1} into \eqref{energy:1}, gives
	\begin{equation*}
		\begin{aligned}
			&	E_{h}[ \phi^{n+1} ] - E_{h}[ \phi^{n} ] \\
			& \leq C \tau_{n+1} ( \tau^{4} + h^{4} )
			+ \Big(2K_{\mm} \tau_{n+1}- \f{1}{ 2 ( 1 + S_{2} \tau_{n+1}^{2} ) }\Big) \tau_{n+1} \big\langle \mm(\hat{\phi}^{n+\f{1}{2}}) \mu^{n+\f{1}{2}}, \mu^{n+\f{1}{2}} \big\rangle,
		\end{aligned}
	\end{equation*}
	and thus, if $2K_{\mm} \tau_{n+1} \le \f{1}{ 2 ( 1 + S_{2} \tau_{n+1}^{2} ) }$, which requires that $ \tau_{n+1} \le \min\left\{ 1, \f{1}{4 K_{\mm} ( 1 + S_{2} ) } \right\}$, we can derive the conclusion
	\begin{equation*}
		E_{h}[ \phi^{n+1} ] - E_{h}[ \phi^{n} ] 
		\leq  C \tau_{n+1} ( \tau^{4} + h^{4} ) \quad \text{for} \  n \geq 1,
	\end{equation*}
	which proves the theorem.
\end{proof}

\begin{remark}
	It is worth emphasizing that the equality \eqref{energy:0} plays a pivotal role in the mobility-robust energy stability analysis. This consideration motivates our adoption of the Crank–Nicolson method for constructing the second-order scheme. In contrast, for other time discretization approaches, such as the BDF2 method \cite{MOC_2023_Ju,ANM_2025_Hou,SINUM_2020_Liao,JSC_2025_Zhang}, it is nontrivial to determine whether this equality continues to hold.
\end{remark}

\section{Numerical experiments}\label{sec:test}
In this section, we present several numerical experiments for the Allen--Cahn equation \eqref{Model:MAC1}--\eqref{Model:MAC2} with general mobility to validate the theoretical results of the proposed DsBE scheme \eqref{sch:sBE_1}--\eqref{sch:sBE_2} and DsCN scheme \eqref{sch:2_1_1}--\eqref{sch:2_n_3}, including error accuracy, energy stability, and preservation of the MBP. Throughout all numerical simulations, 
the stabilization parameter $ S_{1}$ is consistently set to $ S_{1} = 2 $.

\subsection{Temporal convergence}\label{subsec:temporal}
In this example, we test the time accuracy of the proposed schemes by considering the exterior-forced Allen--Cahn model, given by $ \phi_{t} = - \mm(\phi) \mu + g( \mbf{x}, t ) $ for $(\mbf{x},t) \in  ( 0, 2 \pi )^{2} \times (0,1] $. The interface width is chosen as $ \varepsilon = 0.01 $, and two different mobility functions are examined: one is the constant mobility $ \mm(\phi) \equiv 1 $, and the other is the degenerate mobility $ \mm(\phi) = 1 - \phi^{2} $. The linear part $ g( \mbf{x}, t ) $ is chosen such that the exact solution is $ \phi( \mbf{x}, t ) = e^{-t} \sin(x) \sin(y) $.

\begin{table}[!t]
	\caption{\small Time accuracy of DsBE and DsCN schemes with uniform and random time
		steps: $\mm(\phi) \equiv 1$.}\label{Ex1:tab1}
	\vspace{-5pt}
	\resizebox{0.98\textwidth}{!}{{\footnotesize
			\begin{tabular}{ccccccccc}	
				\toprule
				\multicolumn{1}{c}{\multirow{2}{*}{numerical}} & \multicolumn{4}{c}{DsBE scheme} & \multicolumn{4}{c}{DsCN scheme} \\
				\cmidrule(lr){2-5}
				\cmidrule(lr){6-9}
				\multicolumn{1}{c}{\multirow{1}{*}{method}} & \multicolumn{2}{c}{uniform steps} & \multicolumn{2}{c}{random steps} & \multicolumn{2}{c}{uniform steps} & \multicolumn{2}{c}{random steps} \\
				\cmidrule(lr){2-3}
				\cmidrule(lr){4-5}
				\cmidrule(lr){6-7}
				\cmidrule(lr){8-9}
				$ N $ & $ \| e^{N} \|_{\infty} $ & order & $\| e^{N} \|_{\infty}  $ & order & $ \| e^{N} \|_{\infty} $ & order & $ \| e^{N} \|_{\infty} $ & order \\
				\midrule
				20  & $7.02 \times 10^{-2}$ & --- & $7.14 \times 10^{-2}$ & --- & $5.83 \times 10^{-3}$ & --- & $5.43 \times 10^{-3}$ & --- \\
				40  & $3.81 \times 10^{-2}$ & 0.88 & $3.93 \times 10^{-2}$ & 0.86 & $1.52 \times 10^{-3}$ & 1.94 & $2.11 \times 10^{-3}$ & 1.37 \\
				80  & $1.99 \times 10^{-2}$ & 0.94 & $2.04 \times 10^{-2}$ & 0.94 & $3.89 \times 10^{-4}$ & 1.97 & $5.20 \times 10^{-4}$ & 2.02 \\
				160 & $1.02 \times 10^{-2}$ & 0.96 & $1.06 \times 10^{-2}$ & 0.95 & $9.82 \times 10^{-5}$ & 1.99 & $1.37 \times 10^{-4}$ & 1.92 \\
				320 & $5.10 \times 10^{-3}$ & 1.00 & $5.29 \times 10^{-3}$ & 1.00 & $2.47 \times 10^{-5}$ & 1.99 & $3.27 \times 10^{-5}$ & 2.07 \\
				\bottomrule
	\end{tabular}}}
\end{table}
\begin{table}[!t]
	\caption{\small Time accuracy of DsBE and DsCN schemes with uniform and random time
		steps: $\mm(\phi) = 1 - \phi^{2}.$}\label{Ex1:tab2}%
	\vspace{-5pt}
	\resizebox{0.98\textwidth}{!}{{\footnotesize
			\begin{tabular}{ccccccccc}
				\toprule
				\multicolumn{1}{c}{\multirow{2}{*}{numerical}} & \multicolumn{4}{c}{DsBE scheme} & \multicolumn{4}{c}{DsCN scheme} \\
				\cmidrule{2-5}
				\cmidrule{6-9}
				\multicolumn{1}{c}{\multirow{1}{*}{method}} & \multicolumn{2}{c}{uniform steps} & \multicolumn{2}{c}{random steps} & \multicolumn{2}{c}{uniform steps} & \multicolumn{2}{c}{random steps} \\
				\cmidrule{2-3}
				\cmidrule{4-5}
				\cmidrule{6-7}
				\cmidrule{8-9}
				$ N $  & $ \| e^{N} \|_{\infty} $ & order & $\| e^{N} \|_{\infty}  $ & order & $ \| e^{N} \|_{\infty} $ & order & $ \| e^{N} \|_{\infty} $ & order \\
				\midrule
				20     & $4.79 \times 10^{-2}$   &  ---    &  $4.87 \times 10^{-2}$ & ---    & $3.98 \times 10^{-3}$   &  ---   & $3.72 \times 10^{-3}$   &  --- \\
				40     & $2.60 \times 10^{-2}$   &  0.88   &  $2.68 \times 10^{-2}$ & 0.86   & $1.04 \times 10^{-3}$   &  1.94  & $1.44 \times 10^{-3}$   &  1.37 \\
				80     & $1.35 \times 10^{-2}$   &  0.95   &  $1.39 \times 10^{-2}$ & 0.95   & $2.65 \times 10^{-4}$   &  1.97  & $3.55 \times 10^{-4}$   &  2.02 \\
				160    & $6.90 \times 10^{-3}$   &  0.97   &  $7.20 \times 10^{-3}$ & 0.95   & $6.70 \times 10^{-5}$   &  1.99  & $9.35 \times 10^{-5}$   &  1.92 \\
				320    & $3.50 \times 10^{-3}$   &  0.98   &  $3.60 \times 10^{-3}$ & 1.00   & $1.68 \times 10^{-5}$   &  1.99  & $2.23 \times 10^{-5}$   &  2.07 \\
				\bottomrule
	\end{tabular}}}
\end{table}
In the simulations, the spatial domain is uniformly partitioned into $ M = 400 $ elements along each spatial direction. The stabilization parameter for the DsCN scheme is set to $ S_{2} = \big( S_{1}/4 + \varepsilon^{2}/h^2 \big)^{2} \approx 0.8195$, ensuring the theoretical requirements of \eqref{MBP:S2_sCN} for both mobility functions. Given the number of temporal partition $ N $, two different time meshes are tested: uniform one with equal stepsize $\tau = T/N$; and random one with stepsizes generated by $30\%$ random perturbation of the uniform stepsize. For the given mobility functions, the numerical results are presented in Tables \ref{Ex1:tab1}--\ref{Ex1:tab2}. As observed, both schemes achieve their desired first- and second-order temporal convergence, respectively, whether with uniform or nonuniform time stepsizes.

\subsection{MBP preservation and energy stability}\label{Ex4_2}
In this subsection, we perform numerical simulations of the coarsening dynamics governed by the Allen–Cahn equation \eqref{Model:MAC1}–\eqref{Model:MAC2} using the same mobility functions as in subsection \ref{subsec:temporal}, with the aim of numerically validating the discrete MBP and energy stability of the proposed numerical schemes. The interface width parameter is set to $ \varepsilon = 0.01 $, and the computational domain is $ \Omega = (0,1)^{2} $. We discretize the spatial domain using $ M = 128 $ spatial grid points and initialize the phase field with uniformly distributed random numbers from $ -0.8 $ and $ 0.8 $, thereby generating highly oscillatory initial data.

\begin{figure}[!t]
	\vspace{-12pt}
	\centering
	\subfigure[$ \mm(\phi) \equiv 1 $]
	{
		\includegraphics[width=0.4\textwidth]{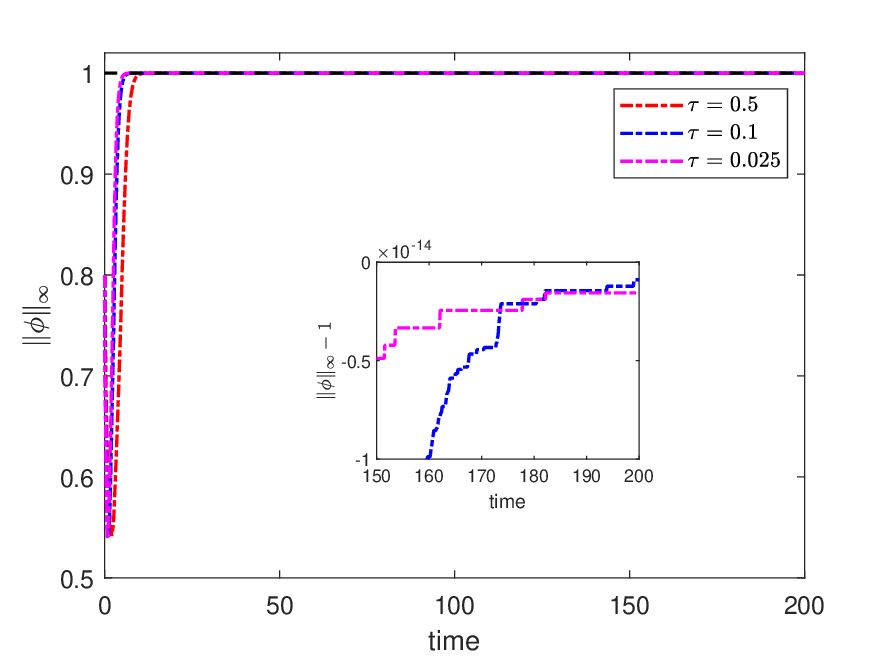}	
	}%
	\subfigure[$ \mm(\phi) = 1-\phi^{2} $]
	{
		\includegraphics[width=0.4\textwidth]{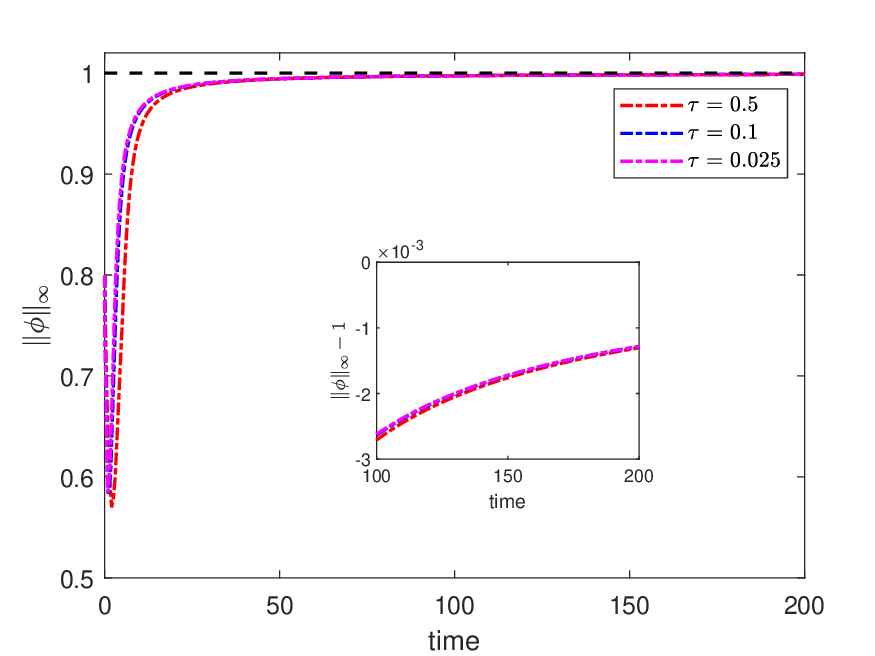}
	}\\ \vspace{-0.4cm}
	\subfigure[$ \mm(\phi) \equiv 1 $]
	{
		\includegraphics[width=0.4\textwidth]{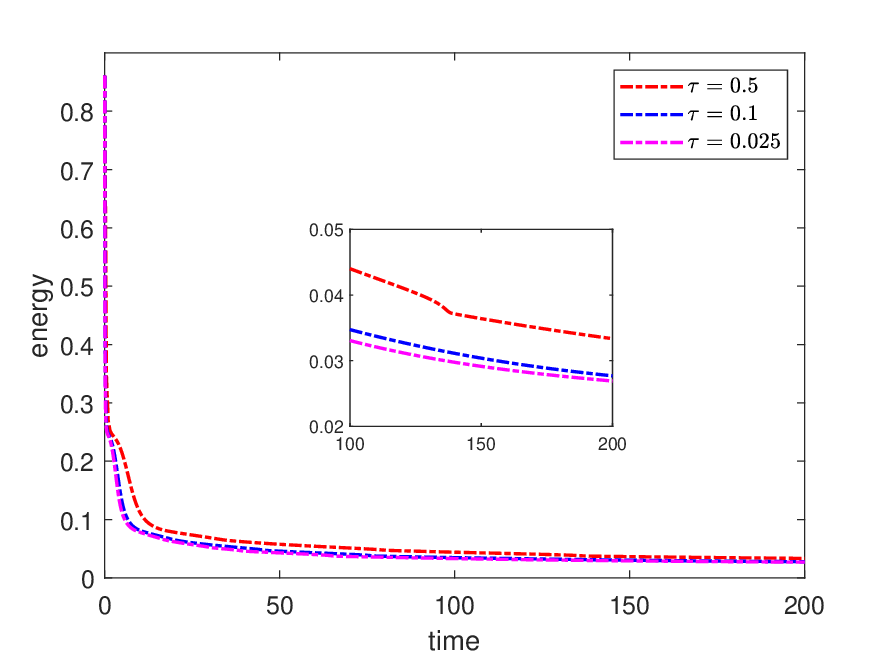}
	}%
	\subfigure[$ \mm(\phi) = 1-\phi^{2} $]
	{
		\includegraphics[width=0.4\textwidth]{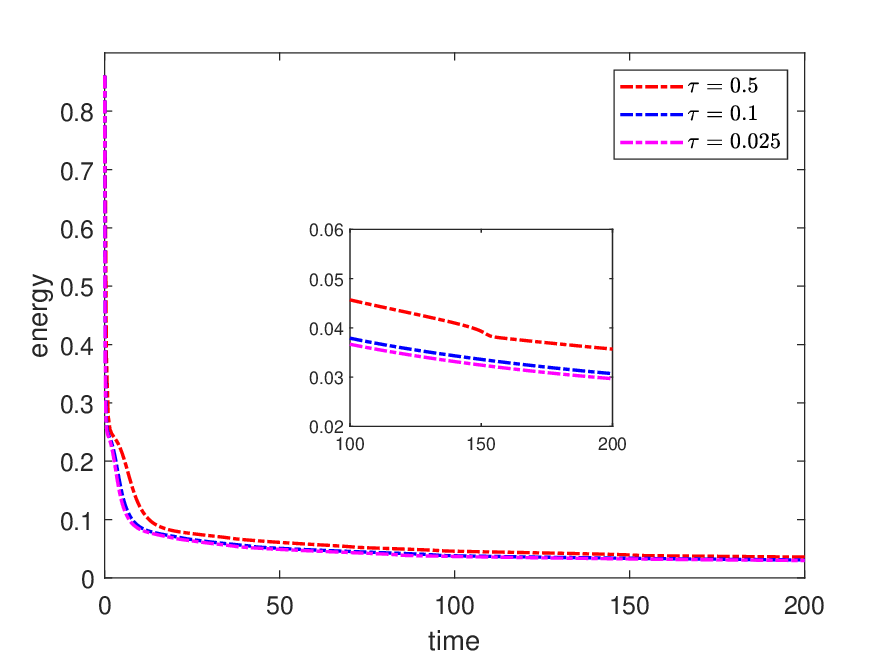}
	}%
	\setlength{\abovecaptionskip}{0.0cm}
	\setlength{\belowcaptionskip}{0.0cm}
	\caption{\small The maximum norm (top) and energy (bottom) of simulated solutions computed by the DsBE scheme for different mobilities.}
\label{figEx2_1}
\end{figure}
Figure \ref{figEx2_1} shows the time evolutions of the maximum norm and total energy computed by the DsBE scheme \eqref{sch:sBE_1}--\eqref{sch:sBE_2} with different time stepsizes $ \tau = 0.5, 0.1, 0.025 $. The results confirm that the scheme preserves both the discrete MBP and energy stability, even for relatively larger time stepsize, thereby demonstrating the unconditional structure-preserving properties established in Theorems \ref{thm:MBP_sBE}--\ref{thm:energy_1}. For the second-order DsCN scheme, the corresponding numerical results obtained with the stabilization parameter $ S_{2} = \big( S_{1}/4 + \varepsilon^{2}/h^2 \big)^{2}  \approx 4.5728 $ are depicted in Figure \ref{figEx2_3}. These results likewise confirm the structure-preserving behavior of the DsCN scheme.

\begin{figure}[!t]
\vspace{-12pt}
\centering
\subfigure[$ \mm(\phi) \equiv 1 $]
{
\includegraphics[width=0.4\textwidth]{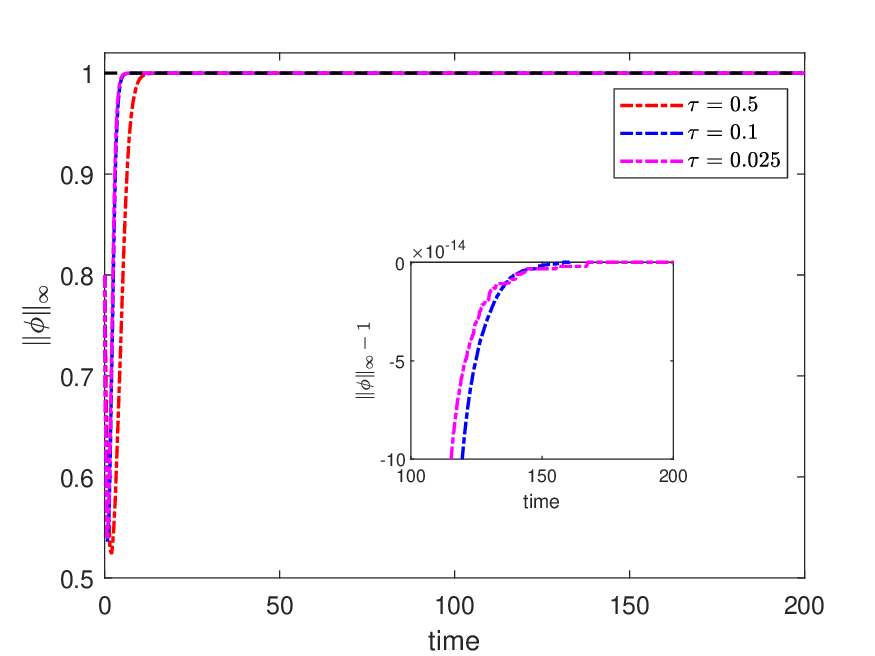}\label{figEx2_3a}
}%
\subfigure[$ \mm(\phi) = 1-\phi^{2} $]
{
\includegraphics[width=0.4\textwidth]{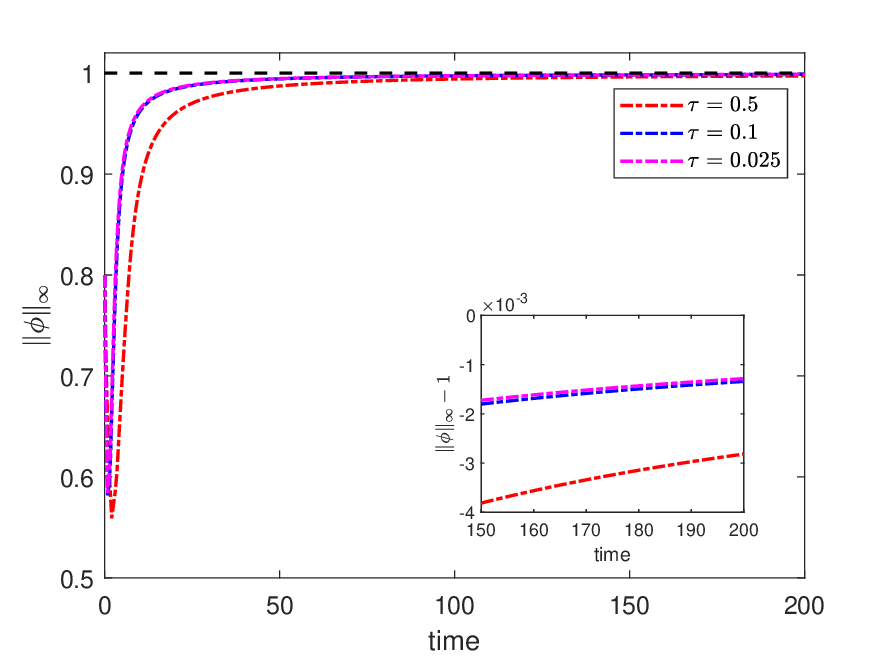}\label{figEx2_3b}
}%
\vspace{-0.4cm}
\subfigure[$ \mm(\phi) \equiv 1 $]
{
\includegraphics[width=0.4\textwidth]{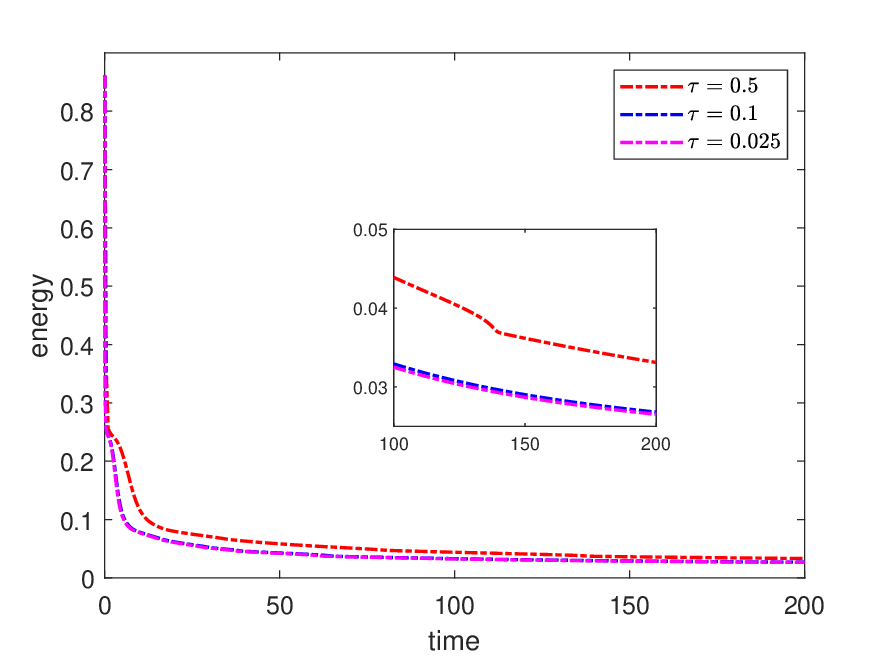}	\label{figEx2_3c}
}%
\subfigure[$ \mm(\phi) = 1-\phi^{2} $]
{
\includegraphics[width=0.4\textwidth]{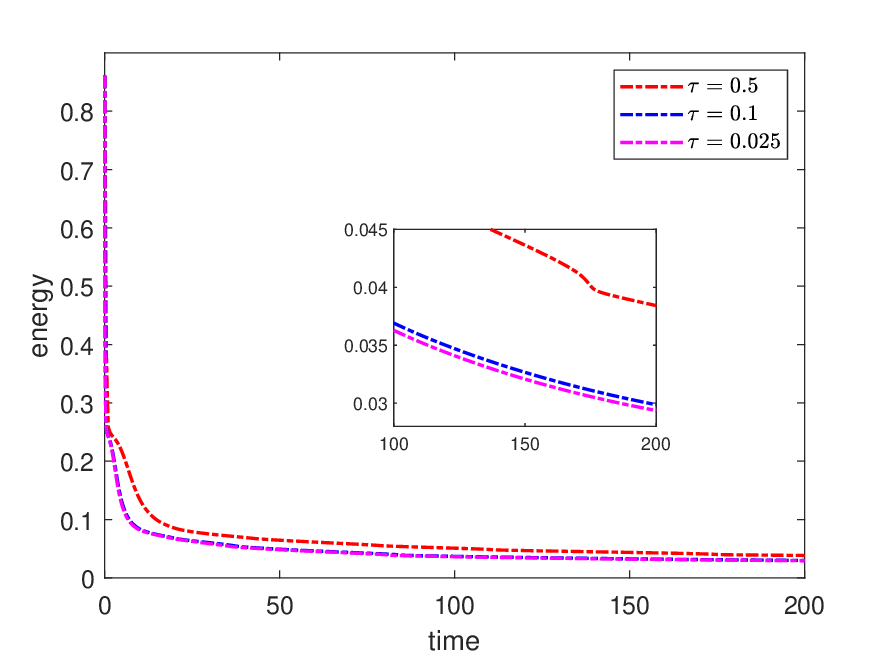}
\label{figEx2_3d}
}%
\setlength{\abovecaptionskip}{0.0cm}
\setlength{\belowcaptionskip}{0.0cm}
\caption{\small The maximum norm (top) and energy (bottom) of simulated solutions computed by the DsCN scheme with $ S_{2} = \left( S_{1}/4 + \varepsilon^{2}/h^2 \right)^{2} $ for different mobilities.}
\label{figEx2_3}
\end{figure}

\subsection{Application of adaptive time-stepping strategy}\label{sub:adaptive}
\begin{figure}[!t]
\centering
\subfigure[$t=20$]
{
\begin{minipage}[t]{0.18\linewidth}
\centering
\includegraphics[width=1.05in]{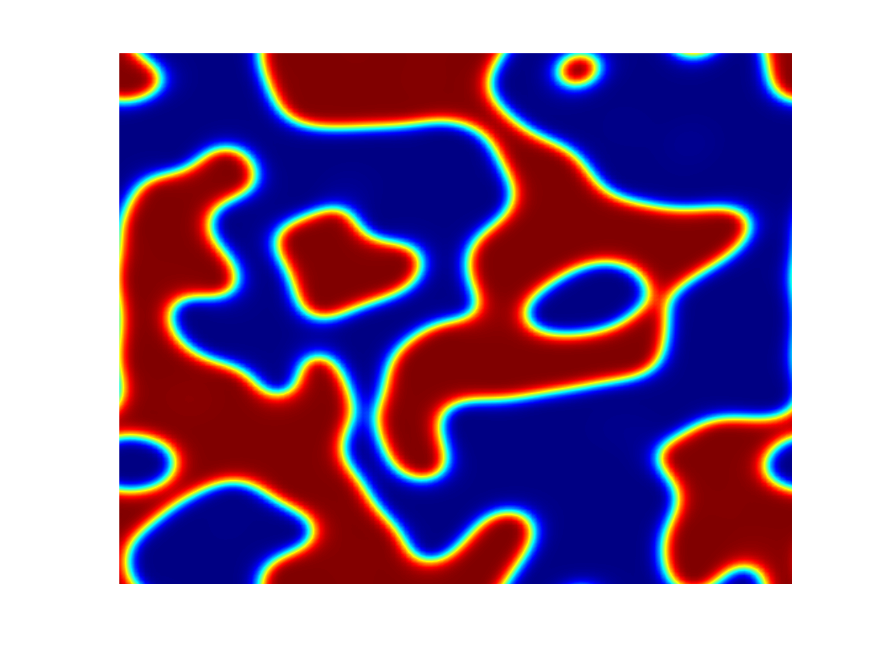}
\includegraphics[width=1.05in]{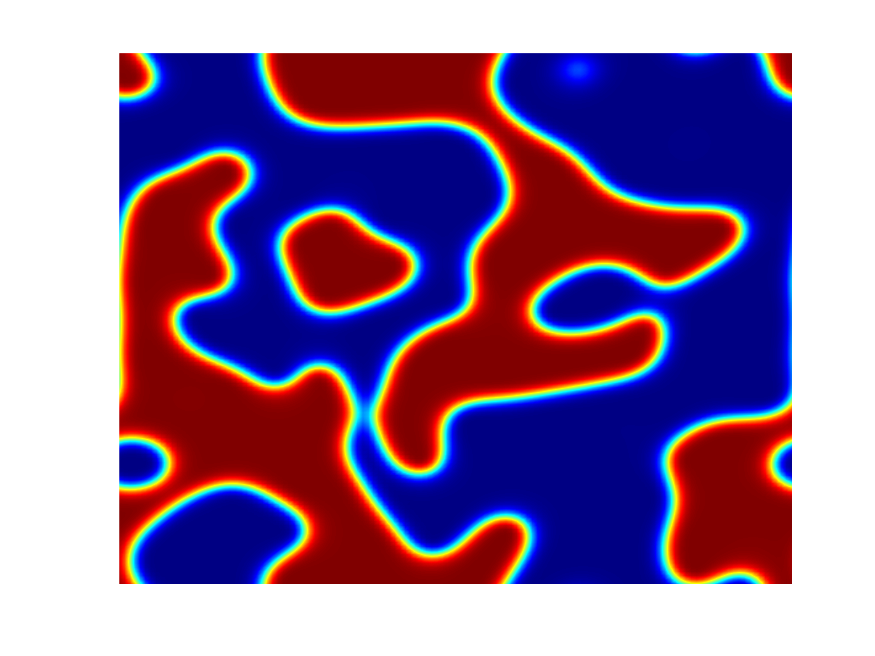}
\includegraphics[width=1.05in]{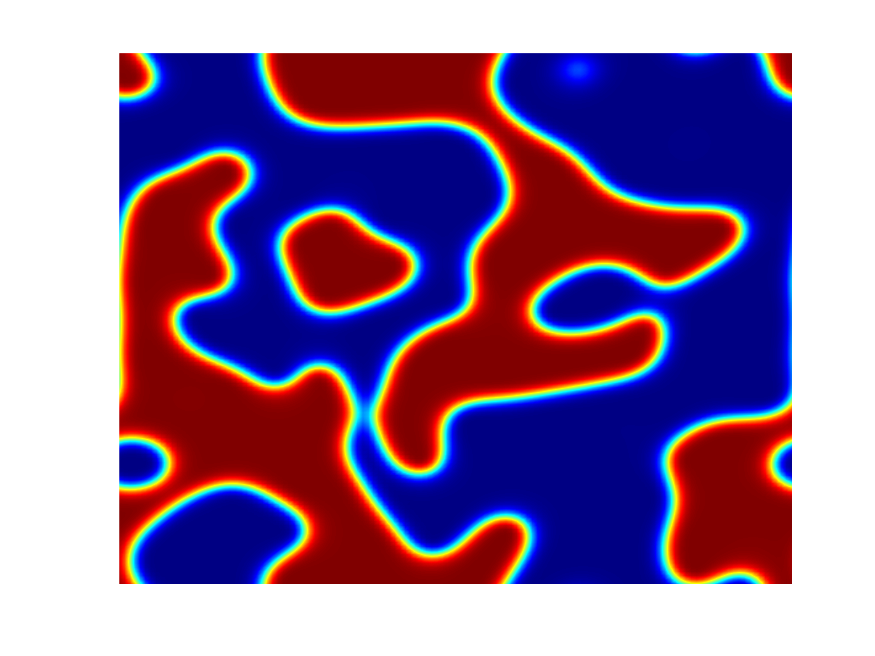}
\end{minipage}%
}%
\subfigure[$t=80$]
{
\begin{minipage}[t]{0.18\linewidth}
\centering
\includegraphics[width=1.05in]{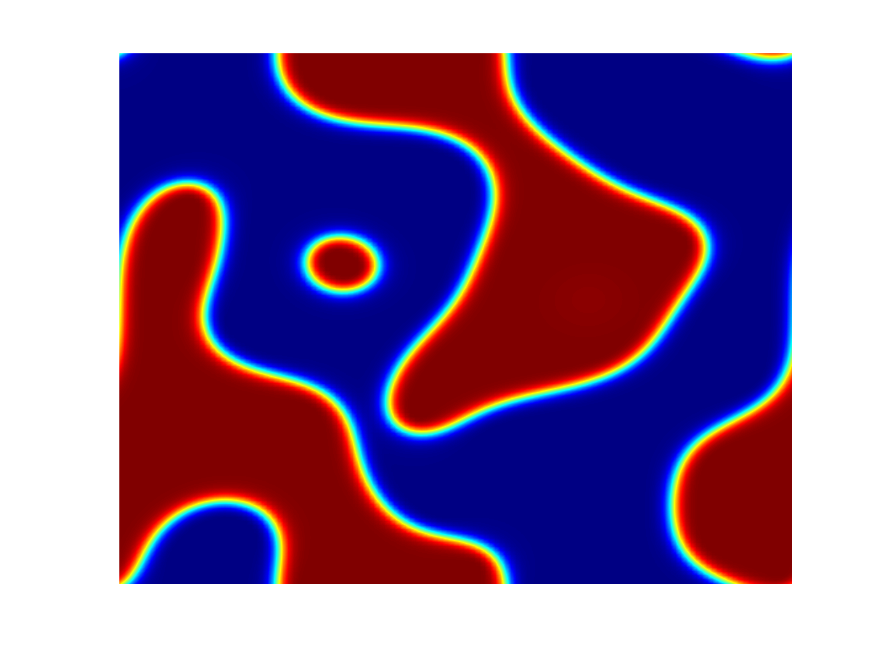}
\includegraphics[width=1.05in]{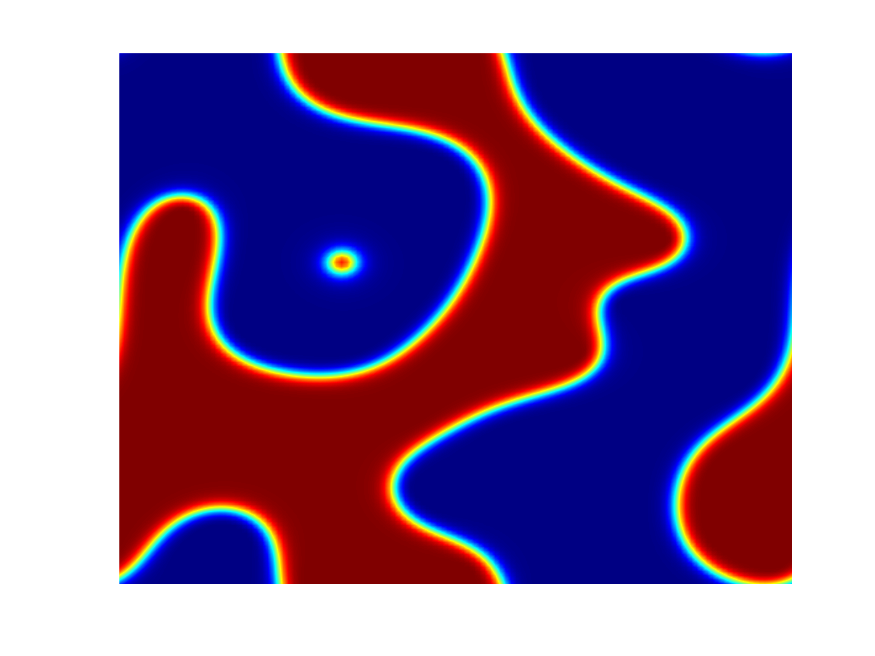}
\includegraphics[width=1.05in]{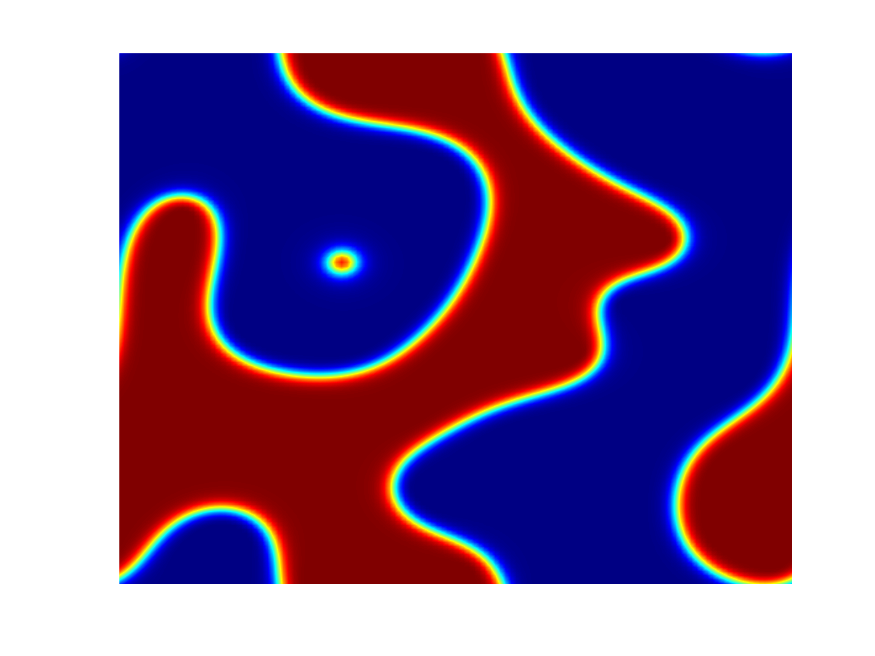}
\end{minipage}%
}%
\subfigure[$t=200$]
{
\begin{minipage}[t]{0.18\linewidth}
\centering
\includegraphics[width=1.05in]{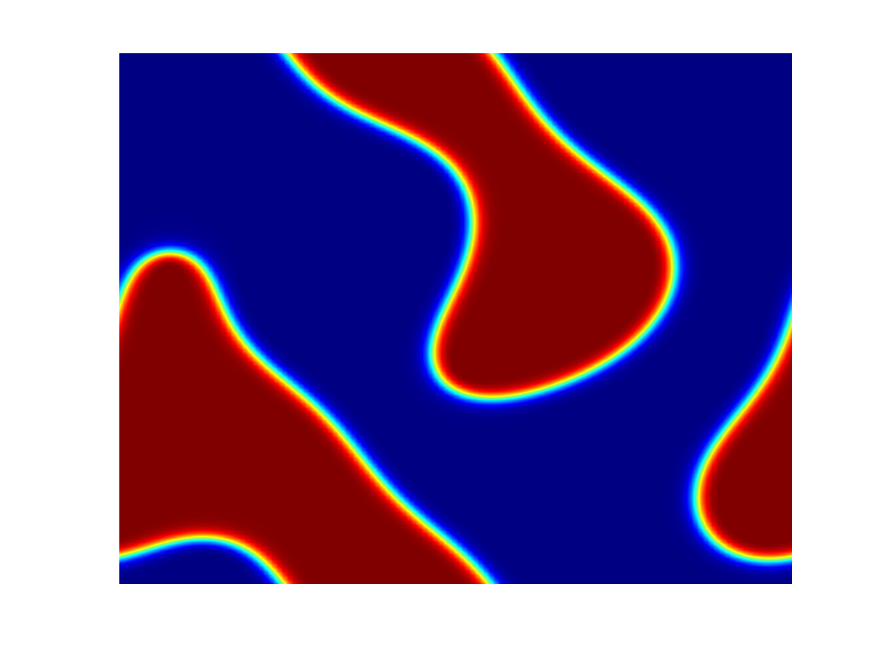}
\includegraphics[width=1.05in]{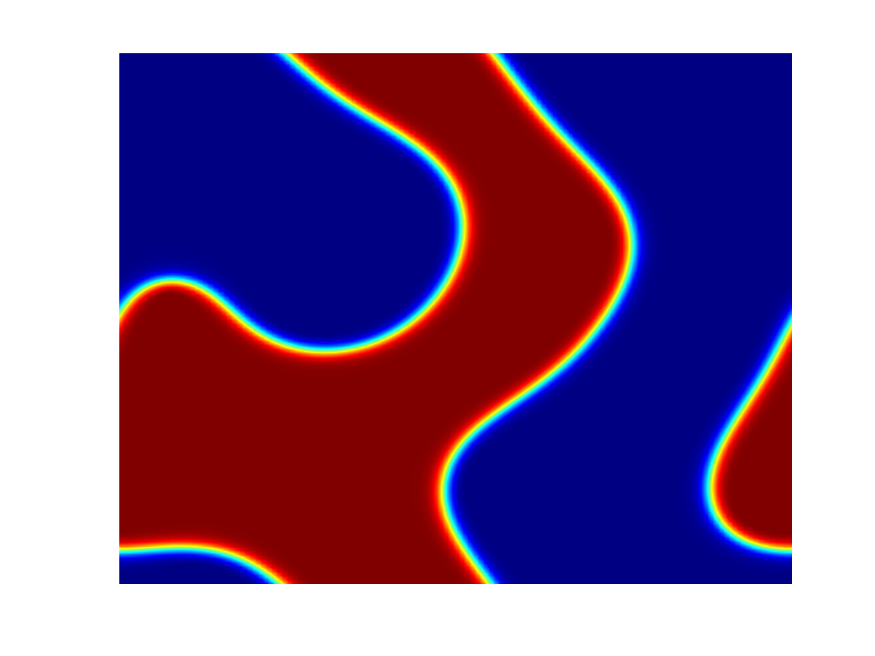}
\includegraphics[width=1.05in]{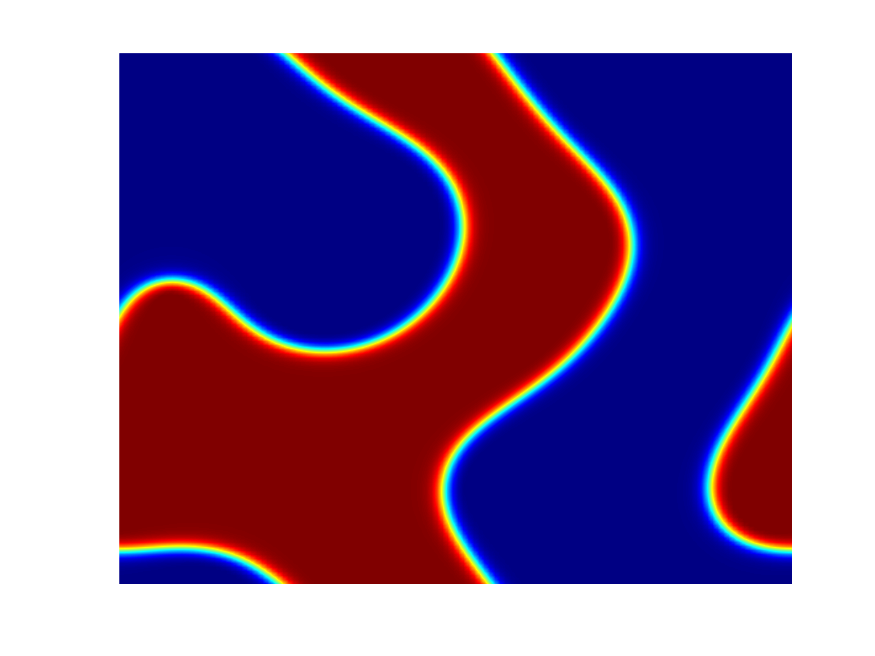}
\end{minipage}%
}%
\subfigure[$t=500$]
{
\begin{minipage}[t]{0.18\linewidth}
\centering
\includegraphics[width=1.05in]{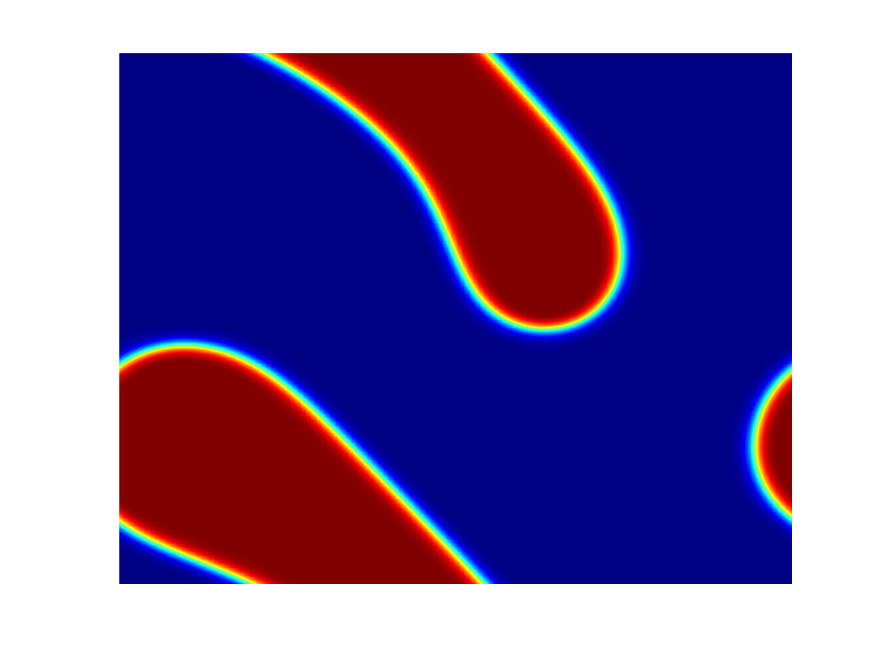}
\includegraphics[width=1.05in]{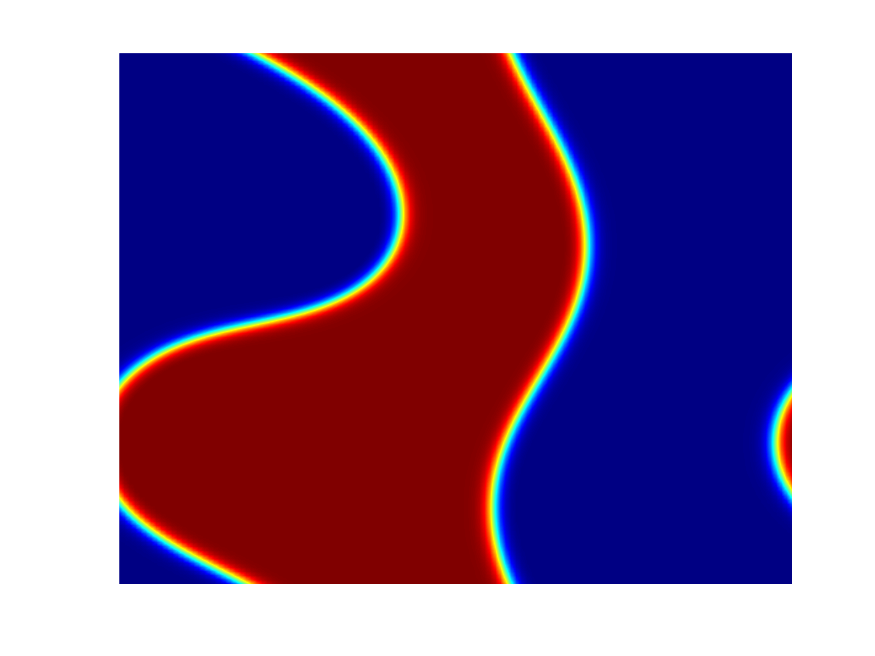}
\includegraphics[width=1.05in]{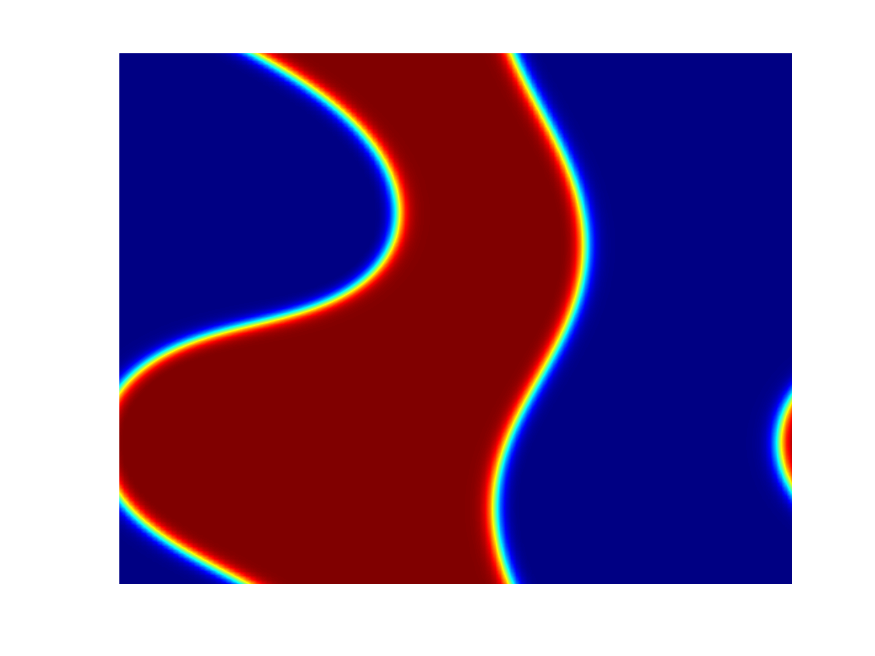}
\end{minipage}%
}%
\subfigure[$t=1000$]
{
\begin{minipage}[t]{0.18\linewidth}
\centering
\includegraphics[width=1.05in]{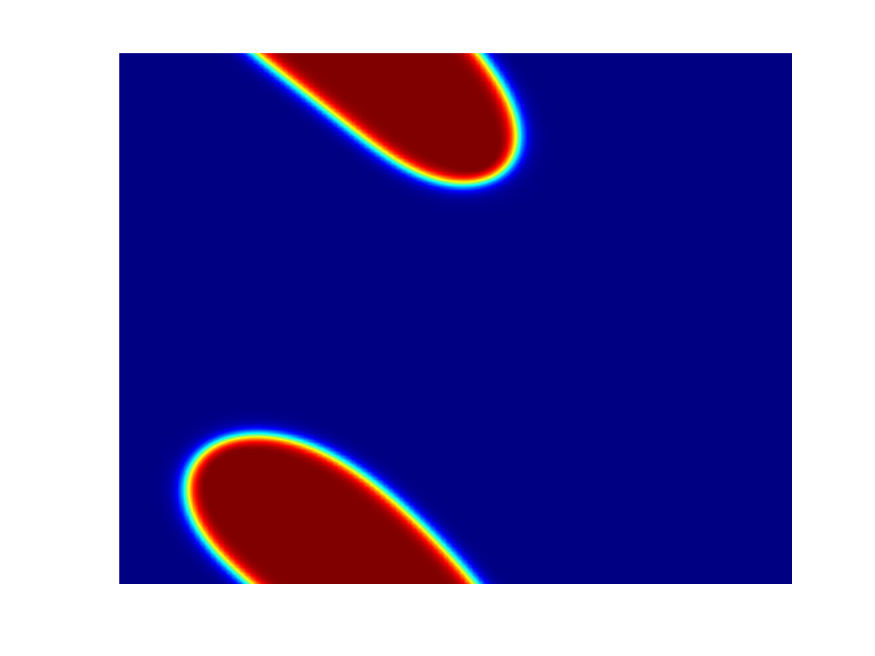}
\includegraphics[width=1.05in]{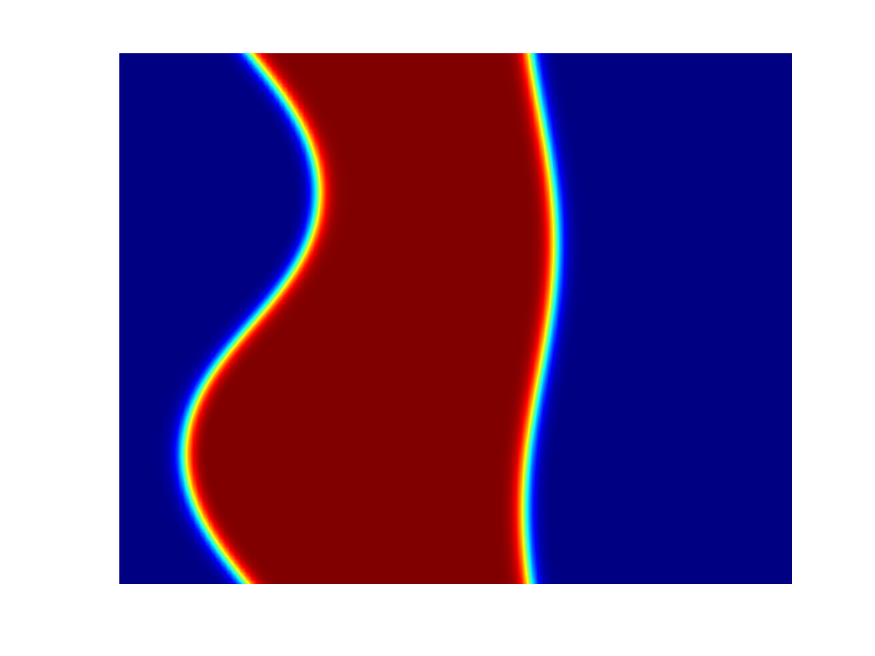}
\includegraphics[width=1.05in]{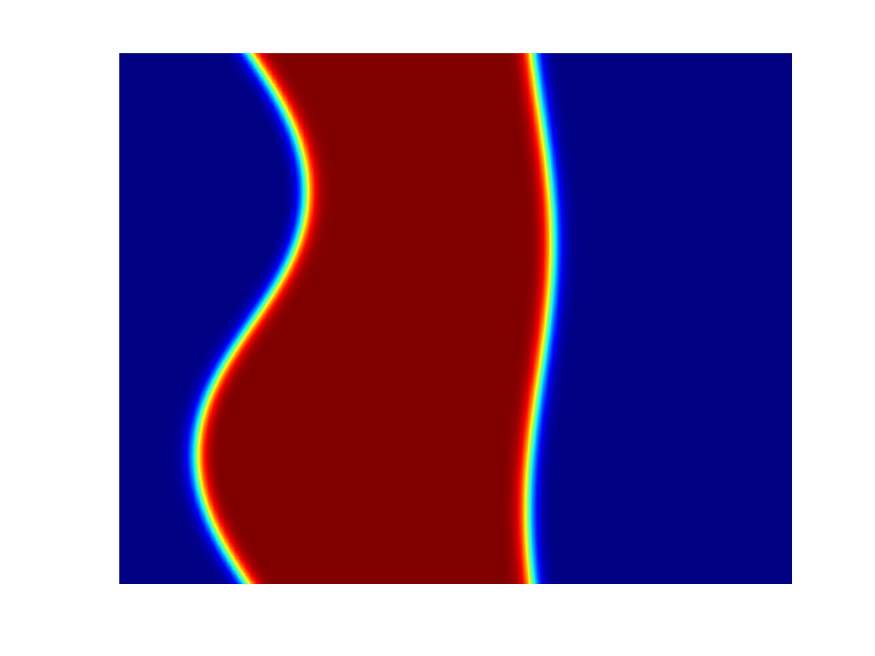}
\end{minipage}%
}%
\setlength{\abovecaptionskip}{0.0cm}
\setlength{\belowcaptionskip}{0.0cm}
\caption{\small The dynamic snapshots of the numerical solution $\phi$ at different time instants obtained by the DsCN scheme with uniform  (top, $\tau = 0.25$; bottom, $\tau = 0.025$) and adaptive (middle) time stepsizes.} \label{figEx3_1}
\end{figure}
\begin{figure}[!t]
\centering
\subfigure[maximum norm]
{
\includegraphics[width=0.32\textwidth]{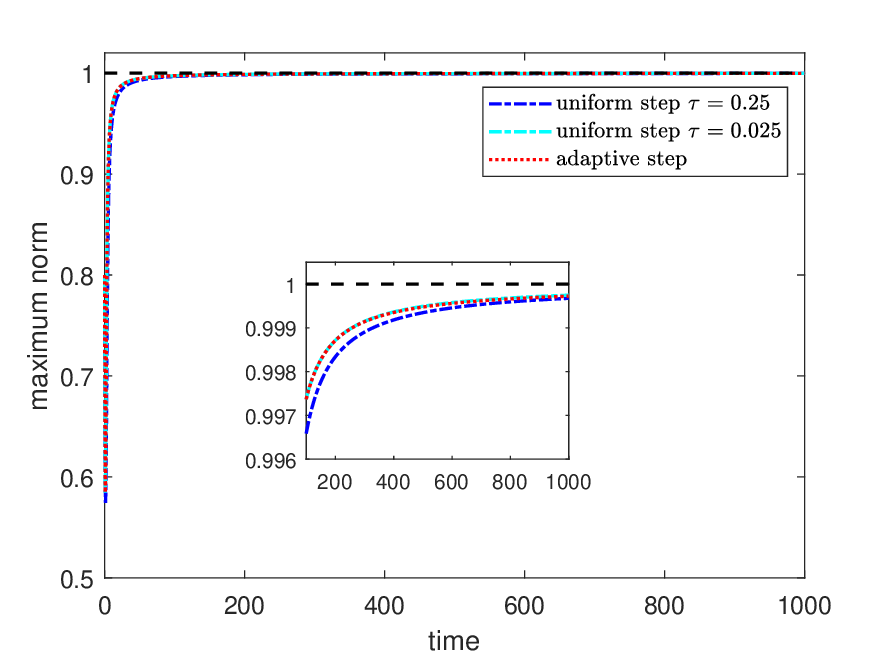}
\label{figEx3_2a}
}%
\subfigure[energy]
{
\includegraphics[width=0.32\textwidth]{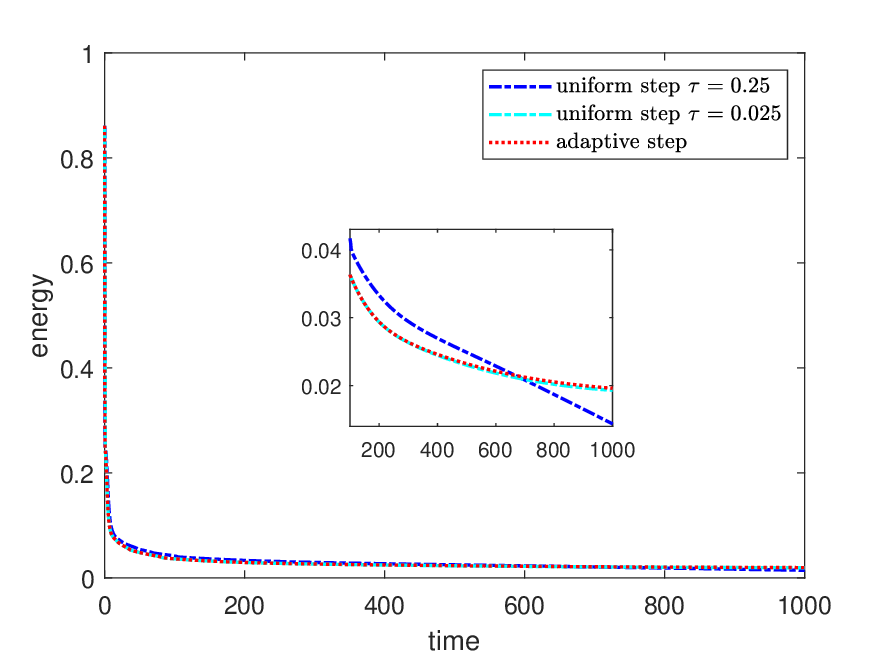}
\label{figEx3_2b}
}%
\subfigure[time stepsizes]
{
\includegraphics[width=0.32\textwidth]{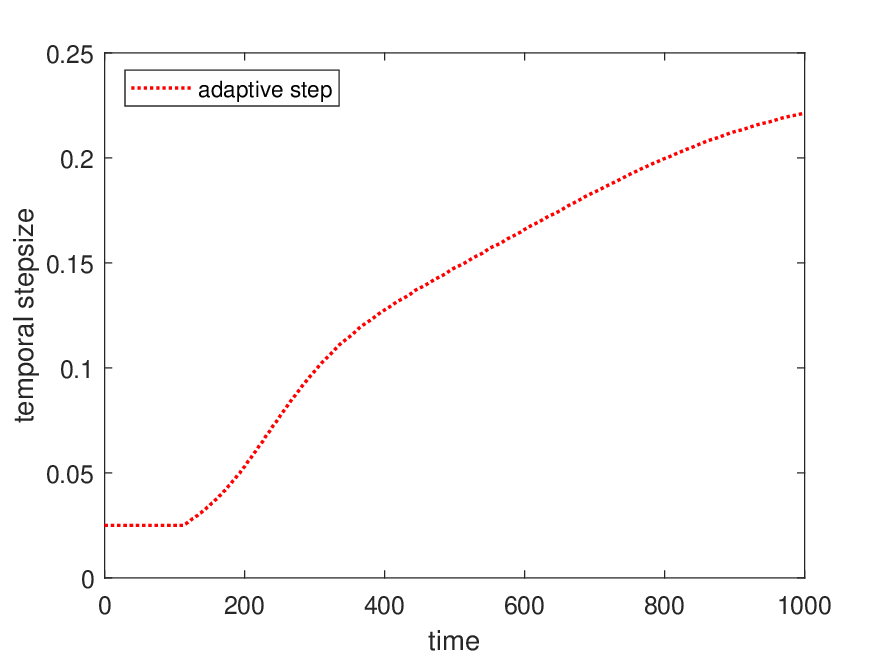}
\label{figEx3_2c}
}%
\setlength{\abovecaptionskip}{0.0cm}
\setlength{\belowcaptionskip}{0.0cm}
\caption{\small The maximum norm (left), energy (middle), and stepsizes (right) for the DsCN scheme.}
\label{figEx3_2}
\end{figure}
\begin{table}[!t]
\caption{\small CPU times and the total number of time steps yielded by the DsCN scheme.}%
\label{tab1_Ex3_1}
\setlength{\abovecaptionskip}{0.0cm}
\setlength{\belowcaptionskip}{0.0cm}
{\footnotesize\begin{tabular*}{\columnwidth}{@{\extracolsep\fill}cccc@{\extracolsep\fill}}
\toprule
time-stepping strategy  & uniform ($ \tau = 0.25$) & adaptive  & uniform ($ \tau = 0.025$)  \\
\midrule
$N$          &   4000        &  12590      &  40000           \\
CPU times    &   12 m 37 s   &  46 m 35 s   &  2 h 1 m 3 s            \\
\bottomrule
\end{tabular*}}
\end{table}
It is worth mentioning that the dynamics governed by the Allen--Cahn equation typically takes a long time to reach steady state and involves an initial rapid dynamics followed
by a slow coarsening process \cite{MOC_2023_Ju,SINUM_2020_Liao}. To enhance computational efficiency without compromising accuracy, we employ the following adaptive time-stepping strategy \cite{SISC_2011_Qiao}
\begin{equation}\label{Alg1:adaptive}
\begin{aligned}
\tau_{n+1} =  \max \bigg\{ \tau_{\min} , \f{\tau_{\max}}{ \sqrt{ 1 + \alpha \vert \delta_{t} E^{n} \vert^{2} } } \bigg\},
\end{aligned}
\end{equation}
where $ \delta_{t} E^{n} $ measures the energy variation, $ \tau_{\max} $ and $ \tau_{\min} $ are predetermined maximum and minimum time stepsizes, and $ \alpha $ is a tunable parameter.

In this simulation, we set $ \tau_{\max} = 0.25, \tau_{\min} = 0.025 $, $ \alpha = 10^{10} $, and all other settings are consistent with those described in subsection \ref{Ex4_2}. The performance of the proposed DsCN scheme with double stabilization parameters is evaluated under three time-stepping strategies: uniform stepsize with $\tau = 0.25$, uniform stepsize with $\tau = 0.025$, and an adaptive strategy determined by \eqref{Alg1:adaptive}. Figure \ref{figEx3_1} presents a comparison of the solution snapshots obtained using the different time stepsizes. It can be seen that the adaptive method yields coarsening patterns that are consistent with those generated using the small stepsize $ \tau = 0.025 $, whereas the larger uniform stepsize $ \tau = 0.25 $ results in noticeably inaccurate phase field profiles. Furthermore, as evidenced in Figures \ref{figEx3_2a}--\ref{figEx3_2b}, the adaptive DsCN scheme perfectly maintains both the discrete MBP and energy stability, and moreover, the time evolutions of both the maximum-norm of $ \phi $ and the total energy align closely with the results from the small uniform time stepsize. Finally, Table \ref{tab1_Ex3_1} highlights the computational advantage of the adaptive approach, demonstrating a reduction of more than $ 60 \% $ CPU time. This improvement is primarily due to the frequent use of larger time stepsizes, as illustrated in Figure \ref{figEx3_2c}, where only a limited number of smaller stepsizes are employed when the energy decays rapidly. Collectively, these results confirm that the adaptive DsCN scheme can significantly improve computational efficiency while maintaining the desired physical attributions and numerical accuracy.
\subsection{Effect of mobility on dynamical behavior}\label{sub:Eff_mobi}
\begin{figure}[!t]
\vspace{-10pt}
\centering
\includegraphics[width=0.45\textwidth]{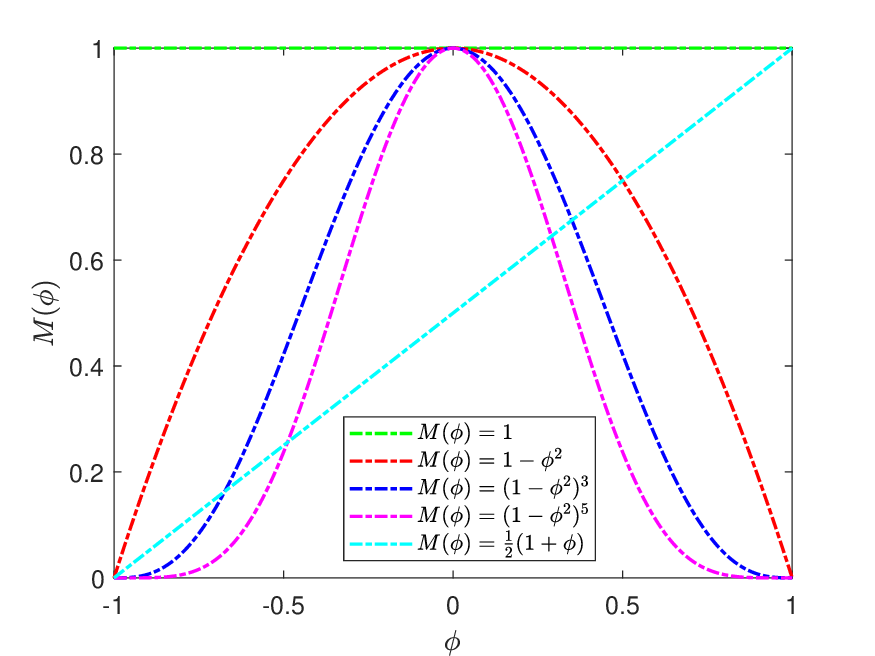}
\setlength{\abovecaptionskip}{0.0cm}
\setlength{\belowcaptionskip}{0.0cm}
\caption{\small Plot of $\mm(\phi)$ on the interval $[-1,1]$.}
\label{figEx4_1}
\end{figure}
\begin{figure}[!t]
\vspace{-12pt}
\centering
\subfigure[$t=10$]
{
\begin{minipage}[t]{0.18\linewidth}
\centering
\includegraphics[width=1.05in]{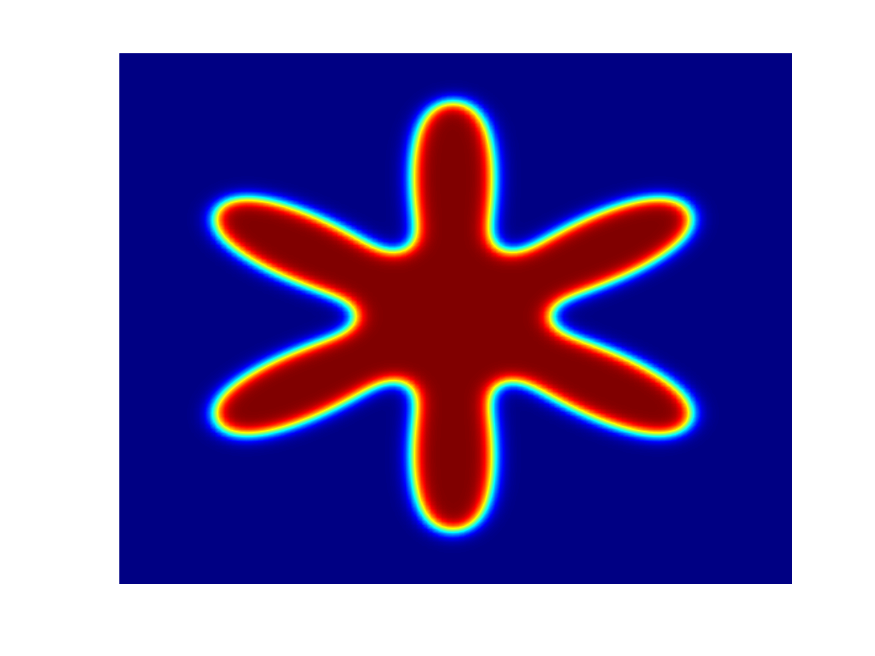}
\includegraphics[width=1.05in]{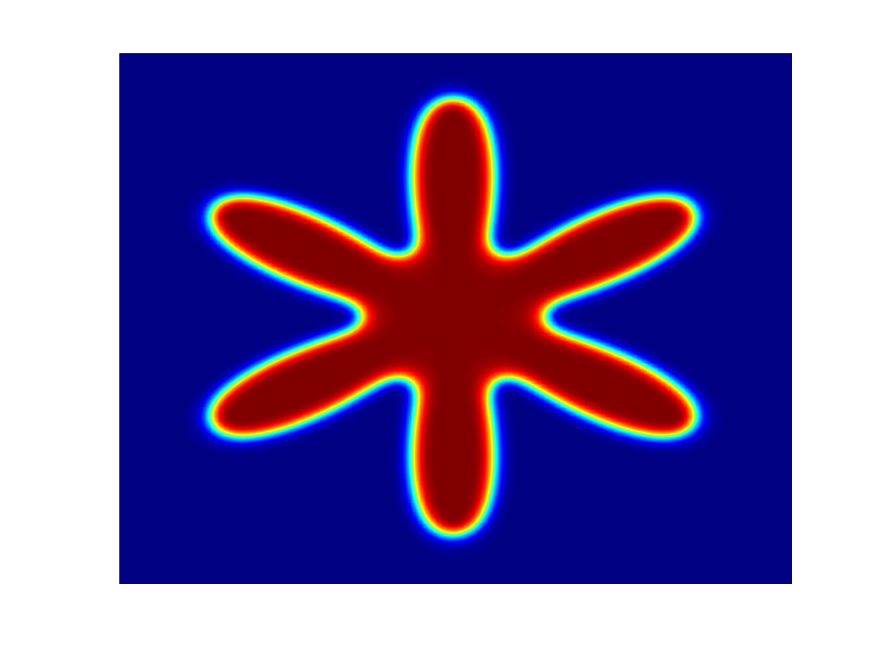}
\includegraphics[width=1.05in]{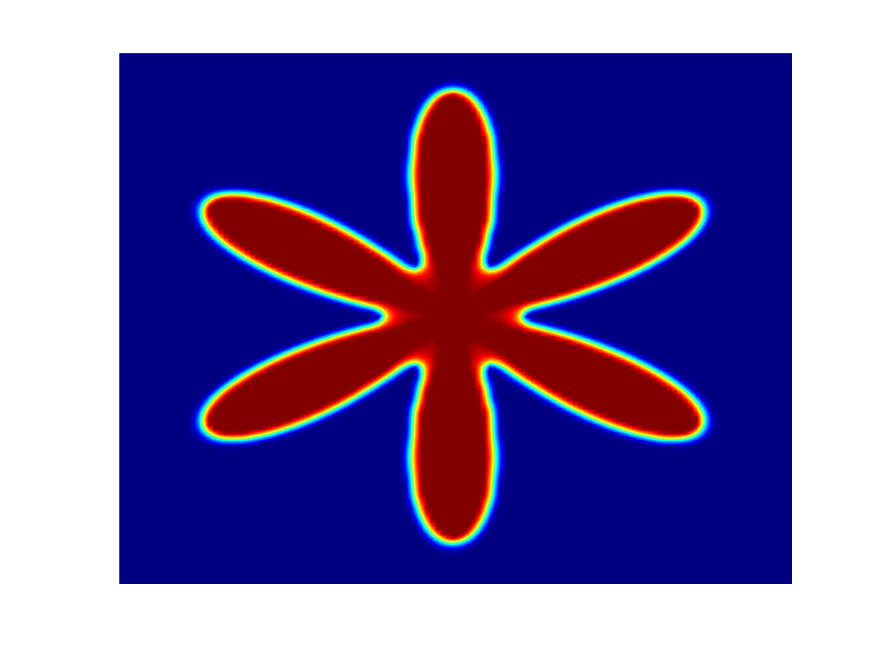}
\includegraphics[width=1.05in]{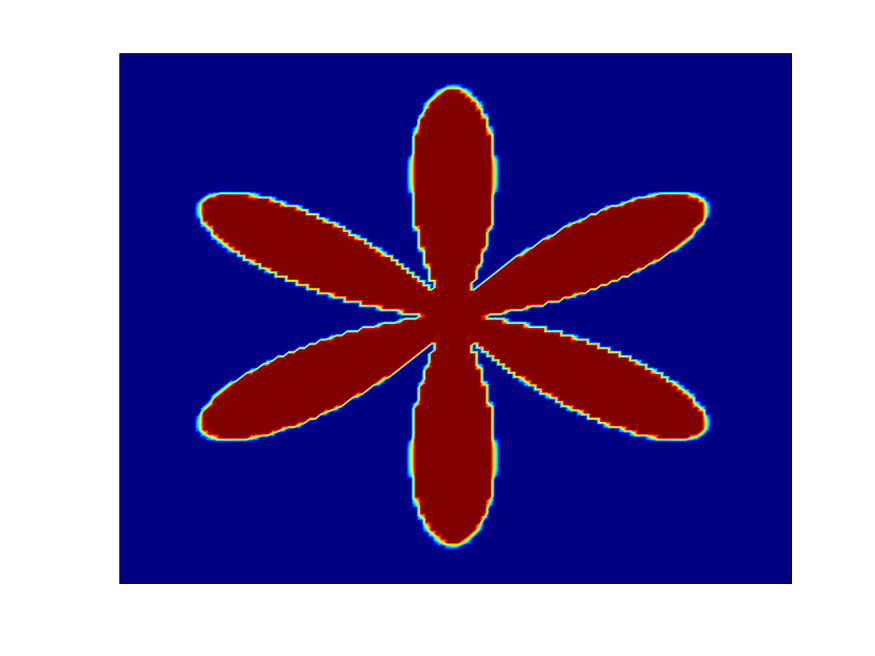}
\includegraphics[width=1.05in]{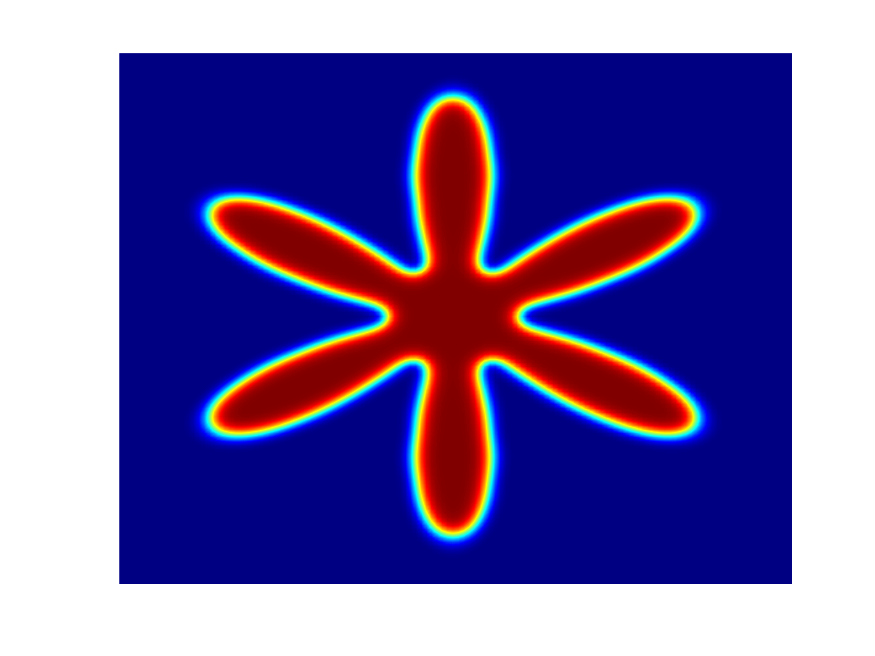}
\end{minipage}%
}%
\subfigure[$t=20$]
{
\begin{minipage}[t]{0.18\linewidth}
\centering
\includegraphics[width=1.05in]{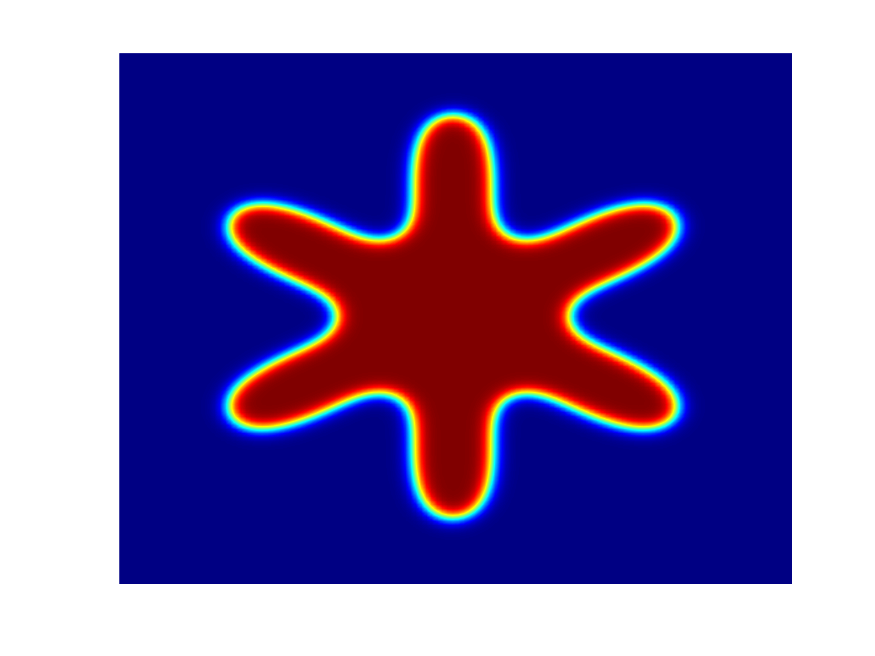}
\includegraphics[width=1.05in]{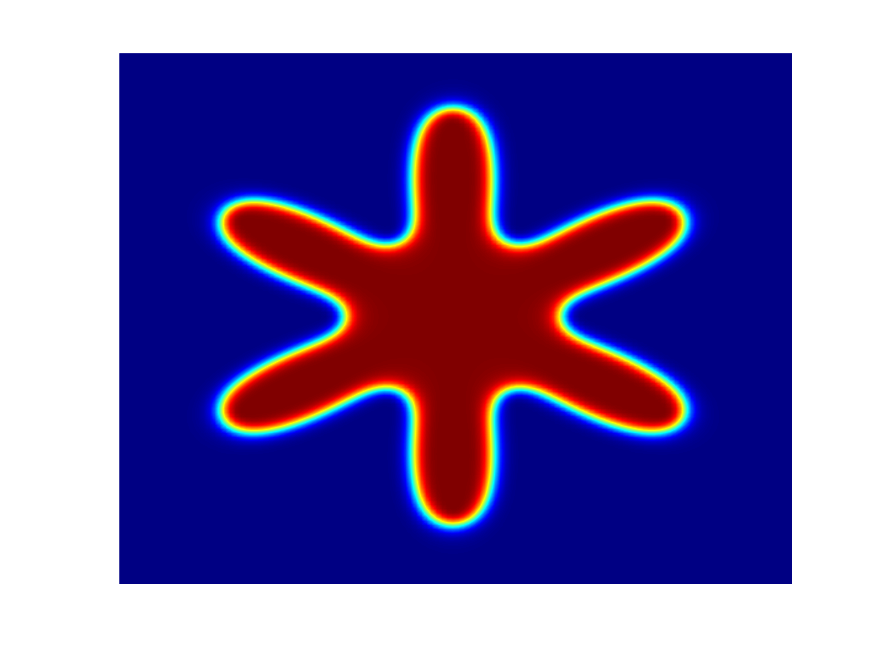}
\includegraphics[width=1.05in]{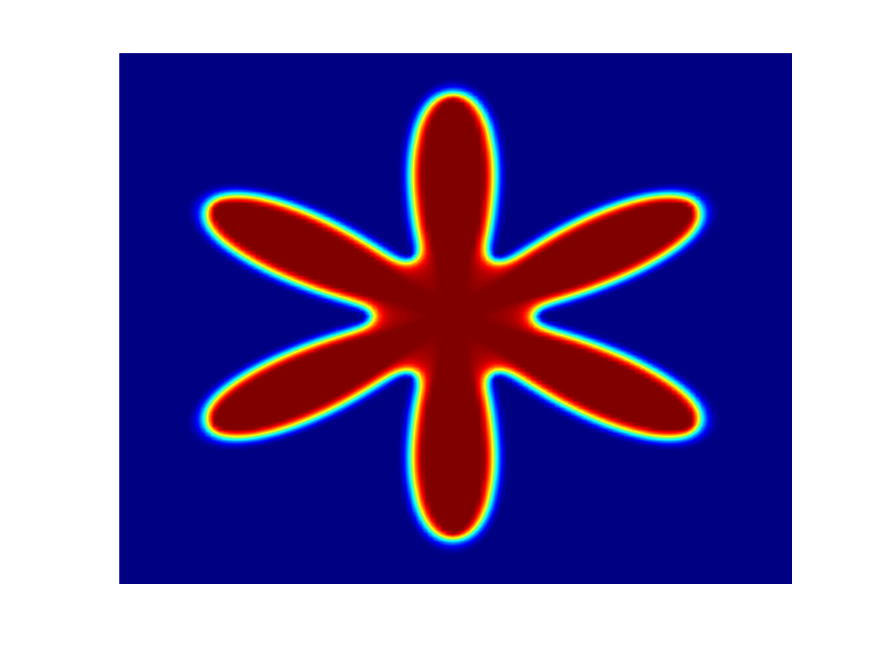}
\includegraphics[width=1.05in]{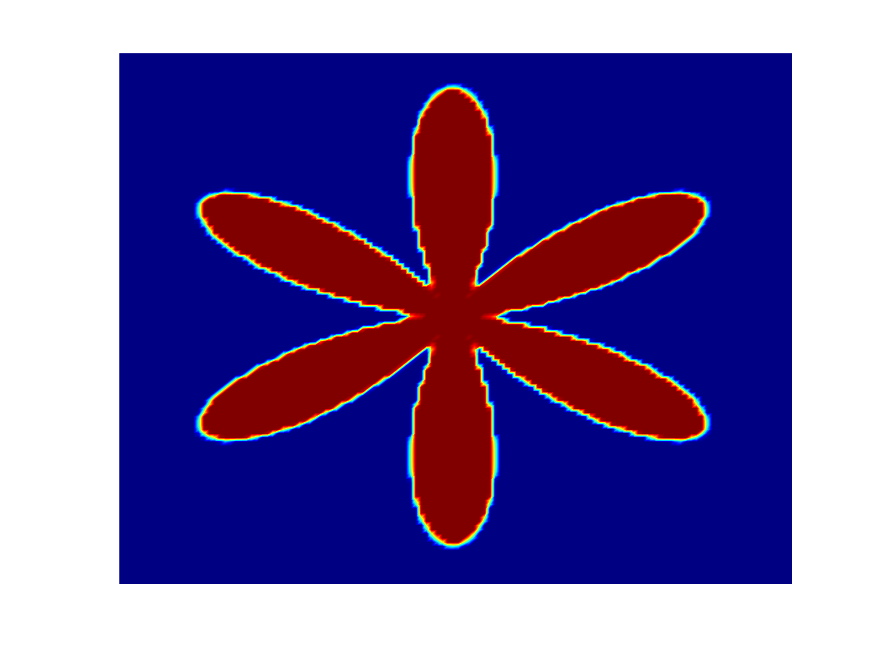}
\includegraphics[width=1.05in]{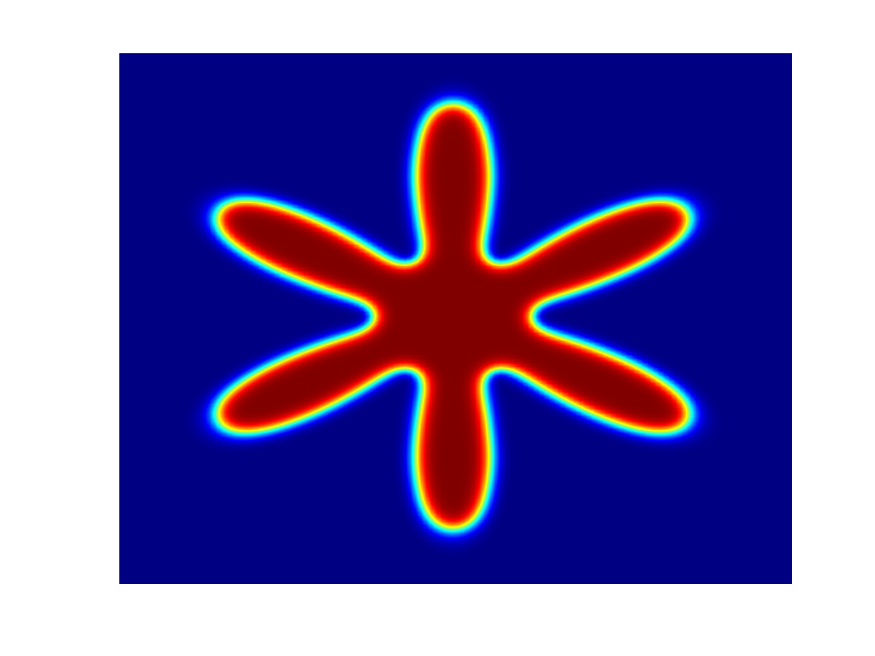}
\end{minipage}%
}%
\subfigure[$t=50$]
{
\begin{minipage}[t]{0.18\linewidth}
\centering
\includegraphics[width=1.05in]{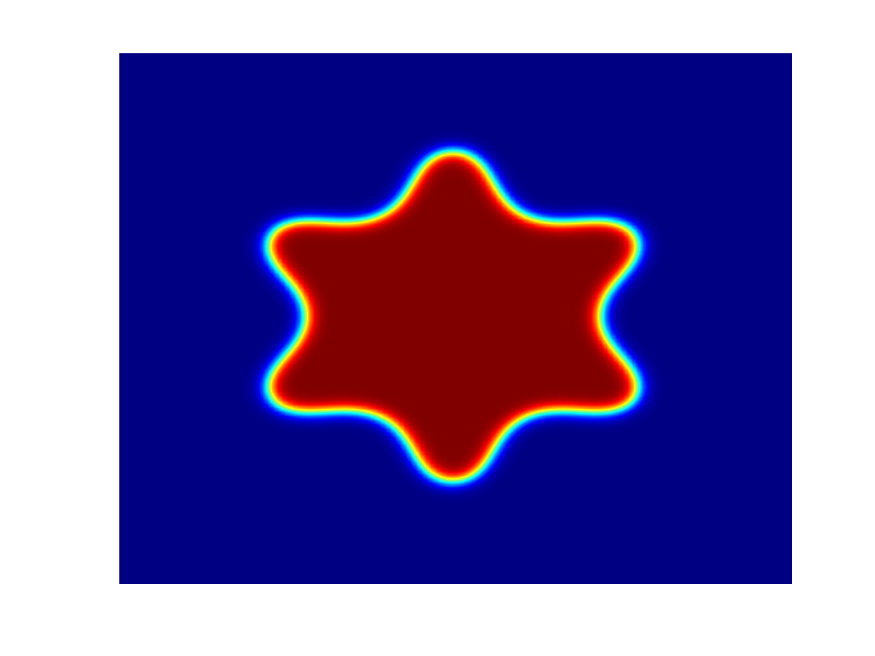}
\includegraphics[width=1.05in]{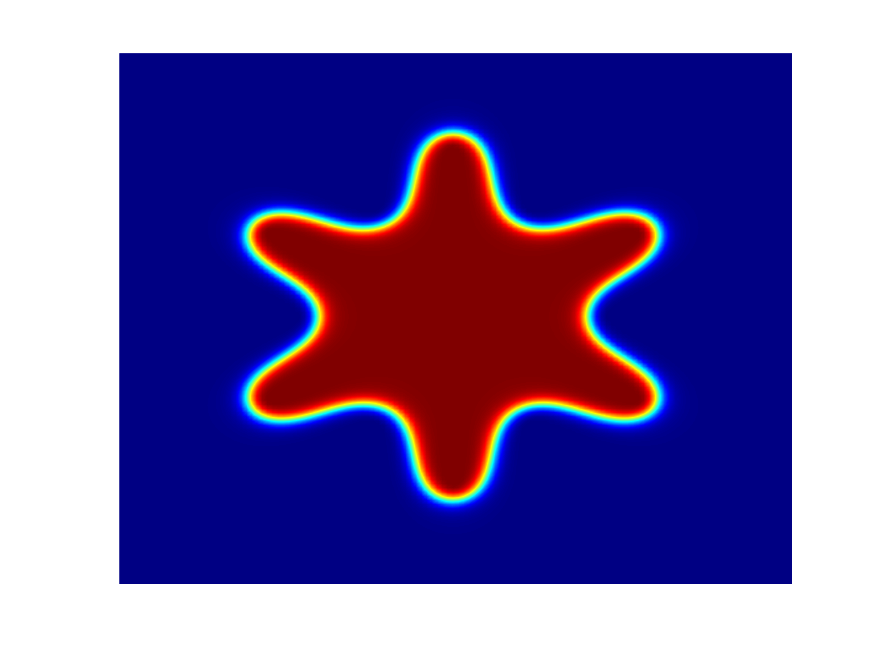}
\includegraphics[width=1.05in]{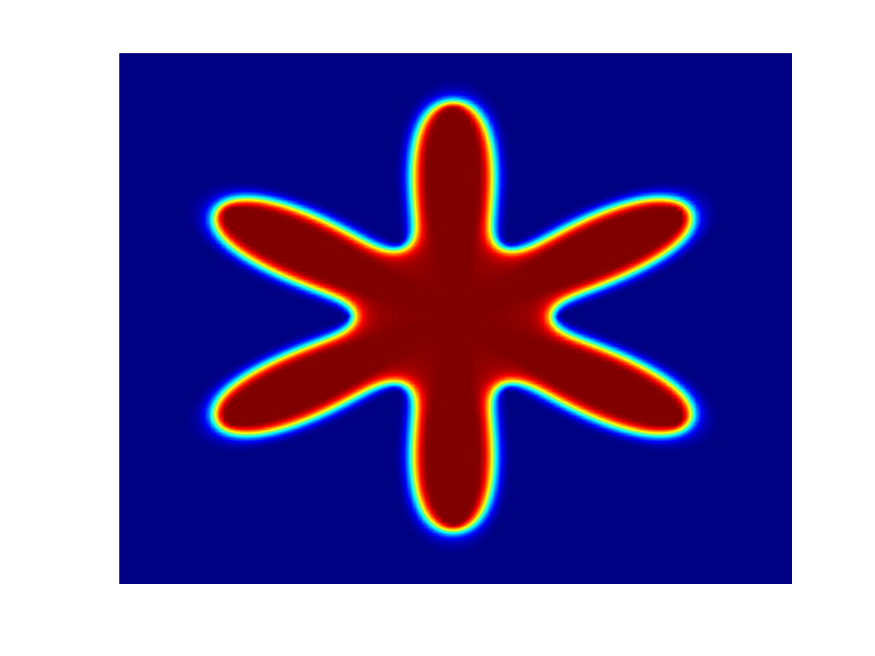}
\includegraphics[width=1.05in]{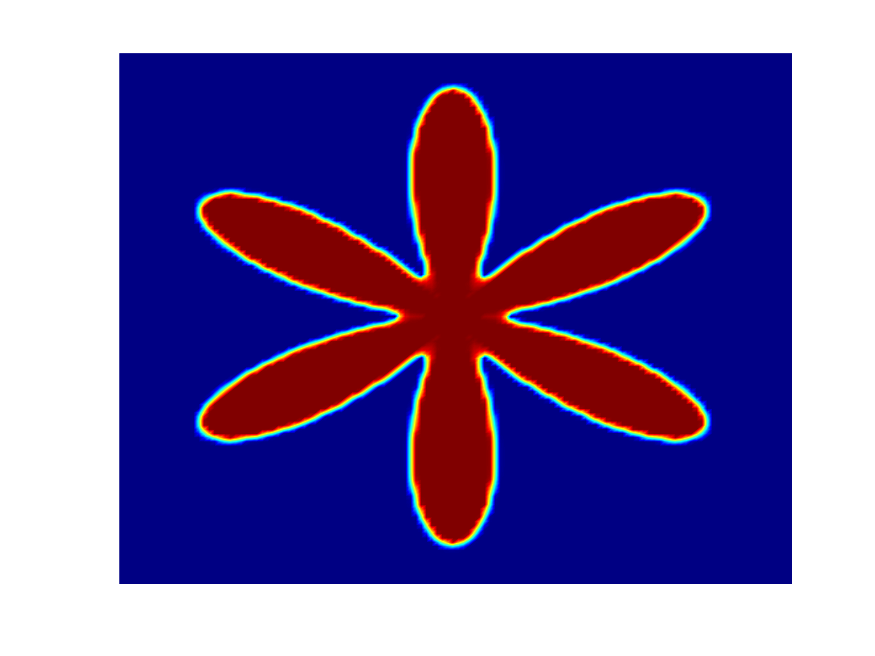}
\includegraphics[width=1.05in]{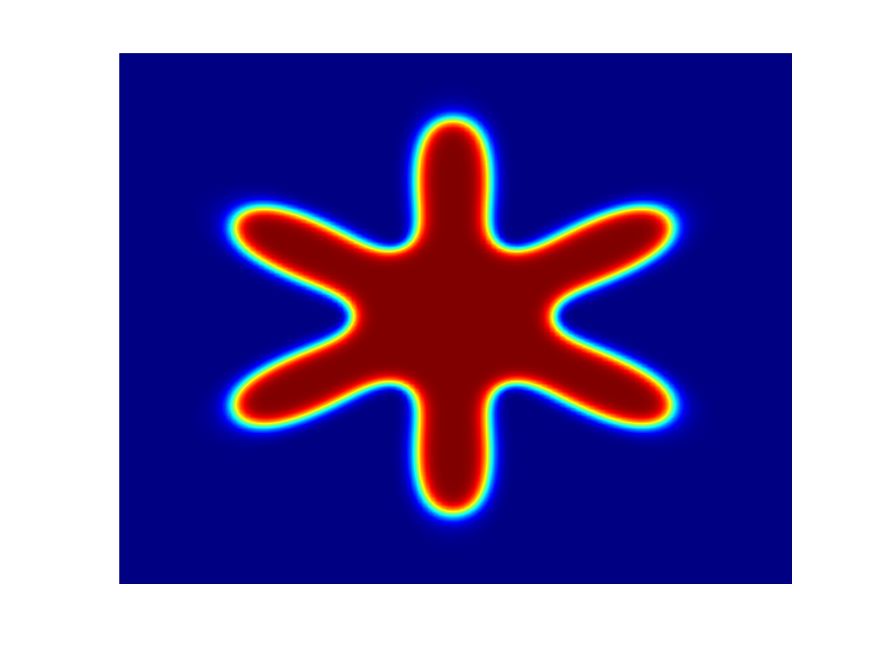}
\end{minipage}%
}%
\subfigure[$t=100$]
{
\begin{minipage}[t]{0.18\linewidth}
\centering
\includegraphics[width=1.05in]{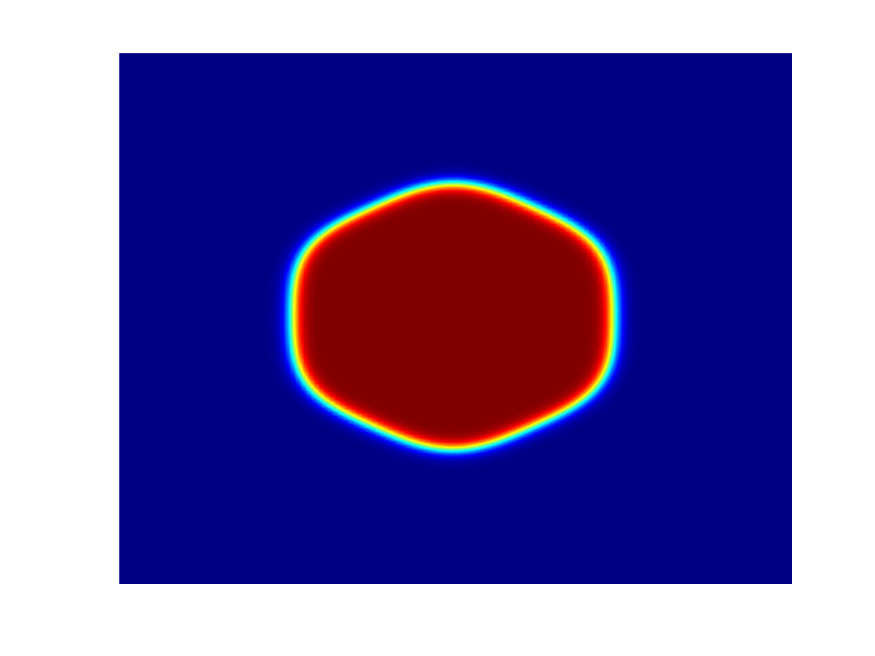}
\includegraphics[width=1.05in]{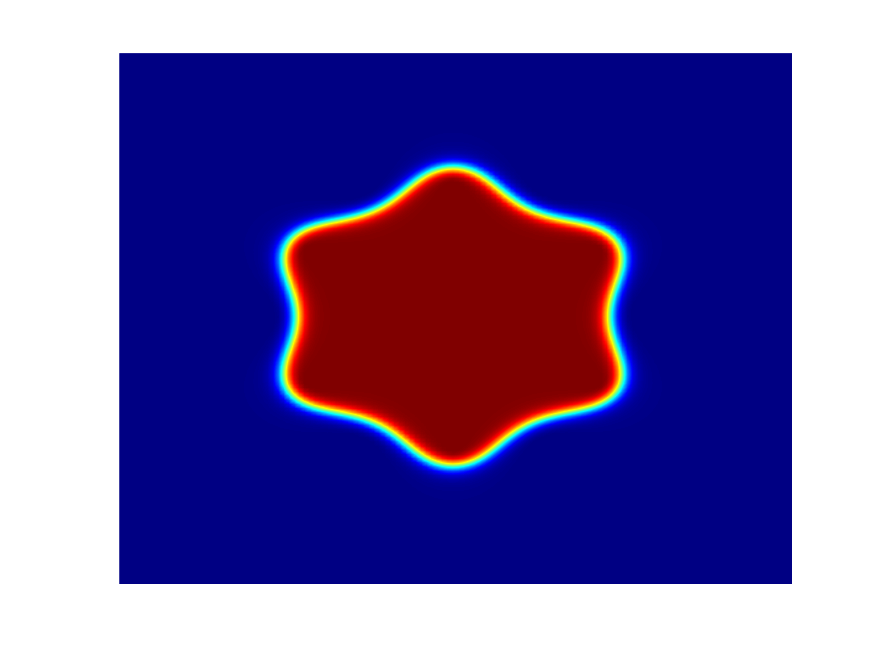}
\includegraphics[width=1.05in]{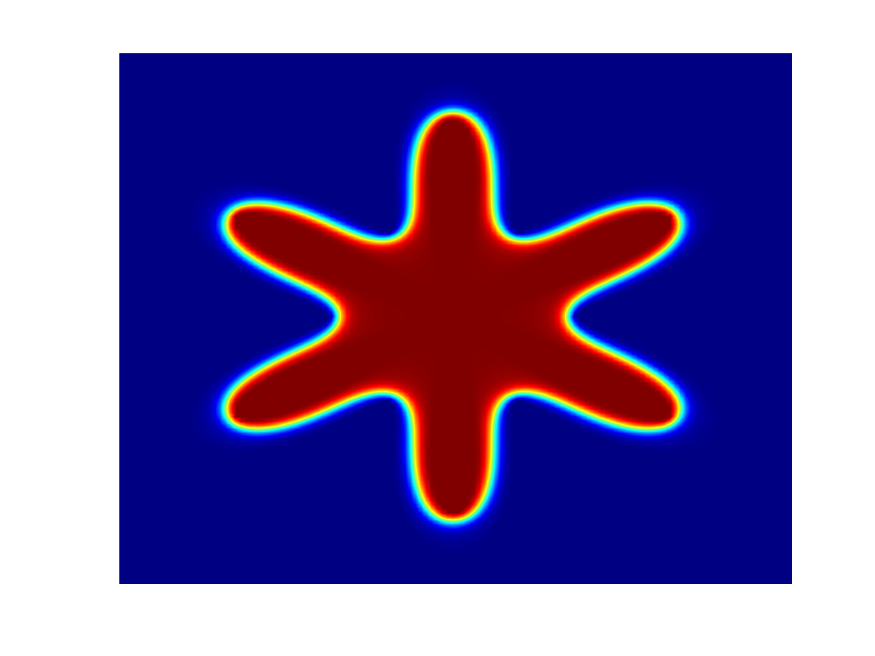}
\includegraphics[width=1.05in]{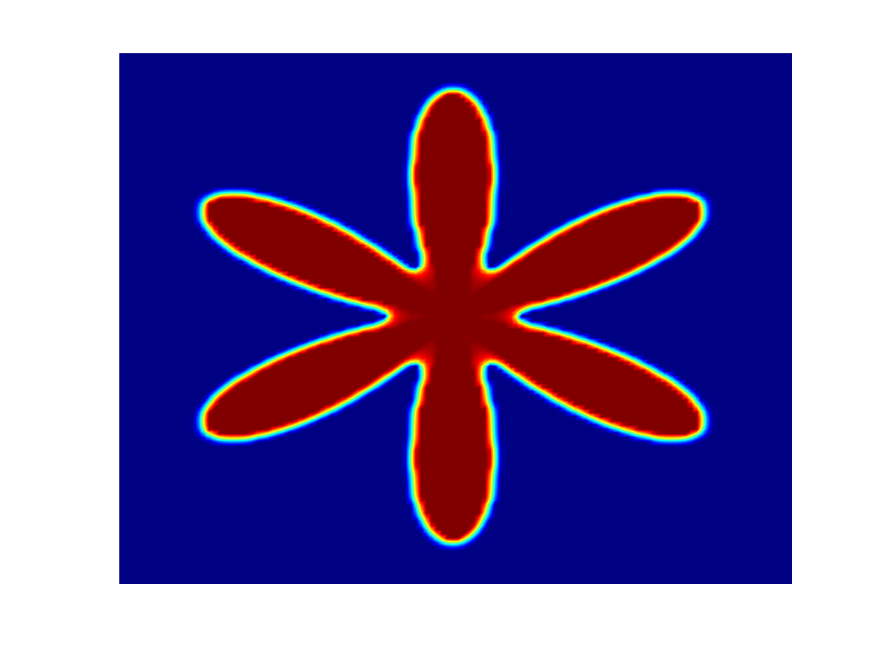}
\includegraphics[width=1.05in]{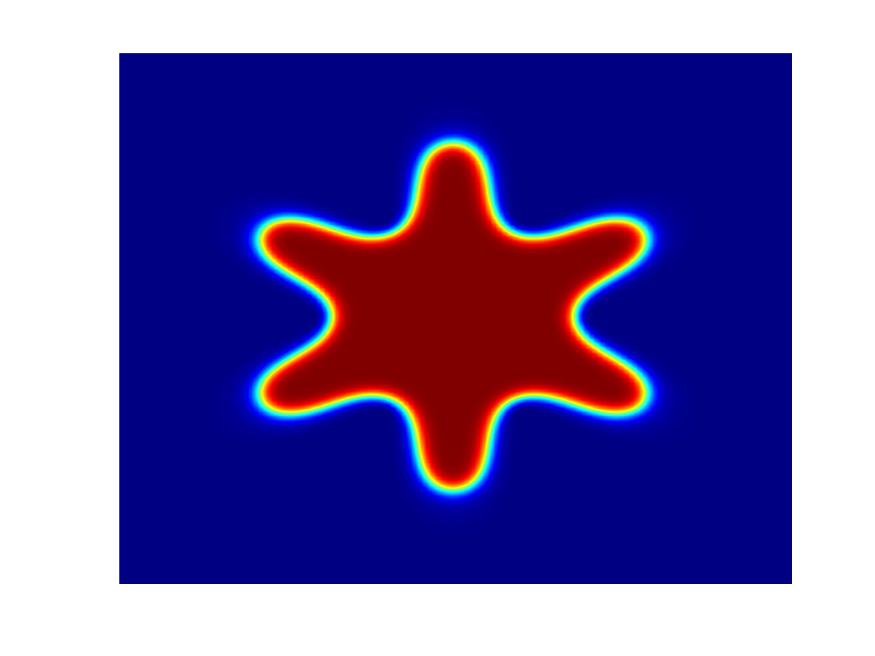}
\end{minipage}%
}%
\subfigure[$t=200$]
{
\begin{minipage}[t]{0.18\linewidth}
\centering
\includegraphics[width=1.05in]{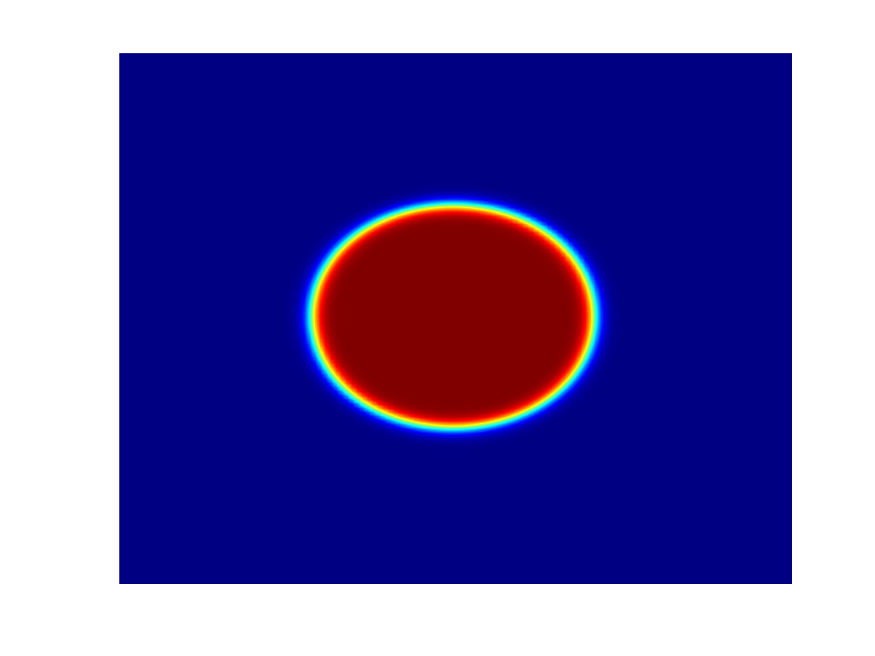}
\includegraphics[width=1.05in]{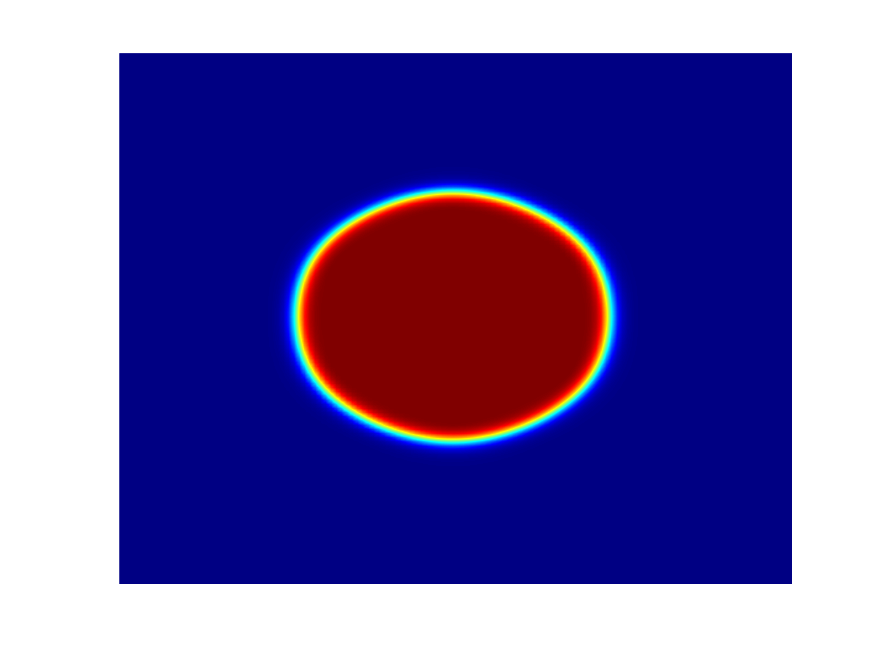}
\includegraphics[width=1.05in]{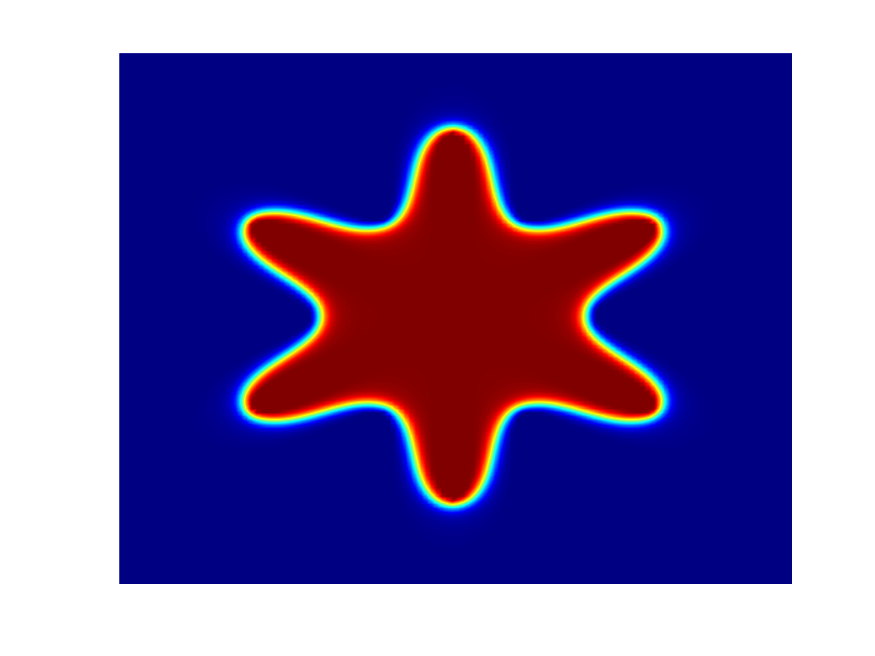}
\includegraphics[width=1.05in]{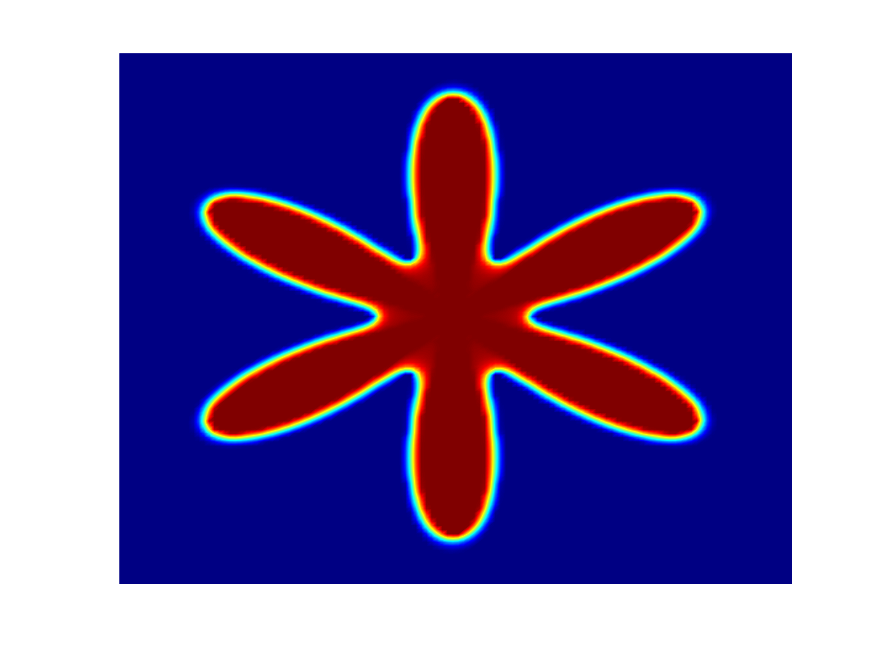}
\includegraphics[width=1.05in]{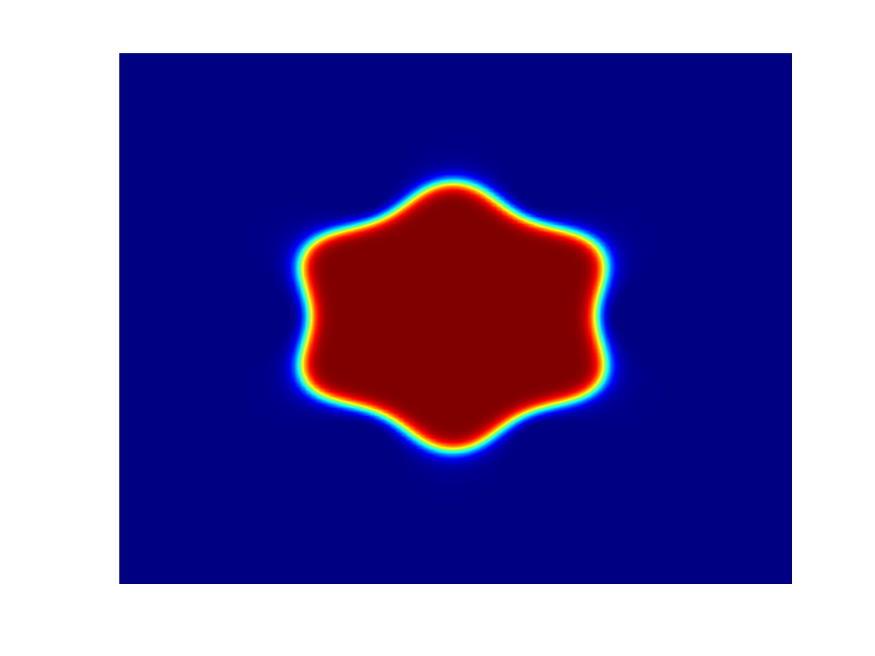}
\end{minipage}%
}%
\setlength{\abovecaptionskip}{0.0cm}
\setlength{\belowcaptionskip}{0.0cm}
\caption{\small The dynamic snapshots of the numerical solution $\phi$ at different time instants obtained by the adaptive DsCN scheme for different mobilities: $ \mm(\phi) = ( 1 - \phi^{2} )^{m} $ with $m=0, 1, 3, 5$, and $ \mm(\phi) = \f{1 + \phi}{2} $ (from top to bottom).} \label{figEx4_2}
\end{figure}
In this subsection, we investigate the influence of the phase-dependent degenerate mobility $ \mm(\phi) $ on the dynamical behavior. To this end, we consider the Allen--Cahn equation \eqref{Model:MAC1}--\eqref{Model:MAC2} with $ \varepsilon = 0.01 $ and examine several representative forms of the mobility function:
\begin{itemize}
\item[--] constant mobility $\mm(\phi) \equiv 1$, corresponding to a bulk-diffusion-dominated dynamical process;

\item[--] two-sided degenerate mobility given by $\mm(\phi) = ( 1 - \phi^{2} )^{m} $ for $  m > 0$ \cite{ARMA_2016_Du,JSP_1994_Taylor}. In this case, the mobility vanishes in both pure phases, meaning diffusion is absent over there and the dynamics is instead driven by diffusion across the interfacial region. This leads to interface-diffusion-dominated dynamics \cite{AMM_1994_Cahn,JSP_1994_Taylor}. The exponent $m$ characterizes the degree of degeneracy: the larger the value of $m$, the more pronounced the degeneracy becomes. In our study, representative cases $m=1$, 3, 5 are tested;

\item[--] one-sided degenerate mobility defined by $ \mm(\phi) = \f{1}{2} ( 1 + \phi)$ \cite{SIAM_2012_Du,JCP_2016_Du,CiCP_2010_Du}. In this setting, the mobility remains non-degenerate in the positive phase but is degenerate in the negative phase. This asymmetry induces a large mobility difference between the two phases, thereby complicating the overall dynamics.
\end{itemize}
Figure \ref{figEx4_1} shows the dependence of these mobility functions on $ \phi $, restricted to the interval $[-1,1]$ in accordance with the MBP. The initial state is chosen as
$$
\phi_{\text{init}}(x,y) = 0.9 \tanh \f{ 1.5 + 1.2\cos( 6\theta ) - 2\pi r }{ \sqrt{2 \lambda } }
$$
with
$$
\theta = \arctan \f{ y - 0.5 }{ x - 0.5 }, \quad  r = \sqrt{ \left( x - 0.5 \right)^2 + \left( y - 0.5 \right)^2 }.
$$

In this example, simulations are carried out using the developed adaptive DsCN scheme with parameters $ \tau_{\max} = 0.25, \tau_{\min} = 0.025 $ and $ \alpha = 10^{7}$. The stabilization parameter is chosen as $ S_{2} = \left( S_{1}/4 + \varepsilon^{2}/h^2 \right)^{2} \approx 4.5728 $, and the computational domain $ \Omega = (0,1)^{2} $ is uniformly discretized into $M = 128$ elements in each spatial direction. Figure \ref{figEx4_2} presents the time evolutions of the phase-field variable for different mobility functions at selected time instants. The results illustrate the dynamic relaxation of an isolated shape into a circular disk, which is consistent with previous studies \cite{JCP_2014_Xu,JSC_2024_Li,JSC_2010_Wise}. Moreover, the relaxation rate is significantly affected by the choice of the mobility function: stronger degeneracy leads to markedly slower evolution. This trend is further corroborated by the maximum norm and energy curves shown in Figure \ref{figEx4_3}. As the degree of degeneracy increases, the rate of energy decay decreases, and it takes longer for the phase-field variable to attain $\| \phi \|_{\infty} = 1$, which typically signals the emergence of pure-phase state. Of particular interest is the case with the one-sided degenerate mobility. 
Due to the huge mobility disparity between the two phases, the time required for the emergence of the positive pure phase (which evolves more quickly) and the negative pure phase (which evolves more slowly) differs significantly from those observed for other mobility cases; as illustrated in Figure \ref{figEx4_4}.
\begin{figure}[!t]
\vspace{-12pt}
\centering
\subfigure[maximum norm]
{
\includegraphics[width=0.4\textwidth]{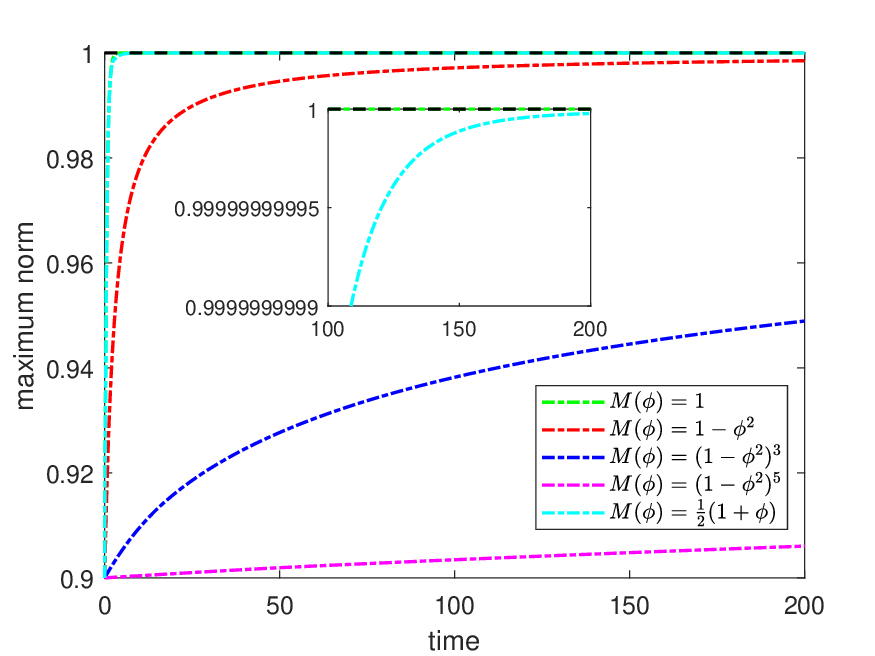}
}%
\subfigure[energy]
{
\includegraphics[width=0.4\textwidth]{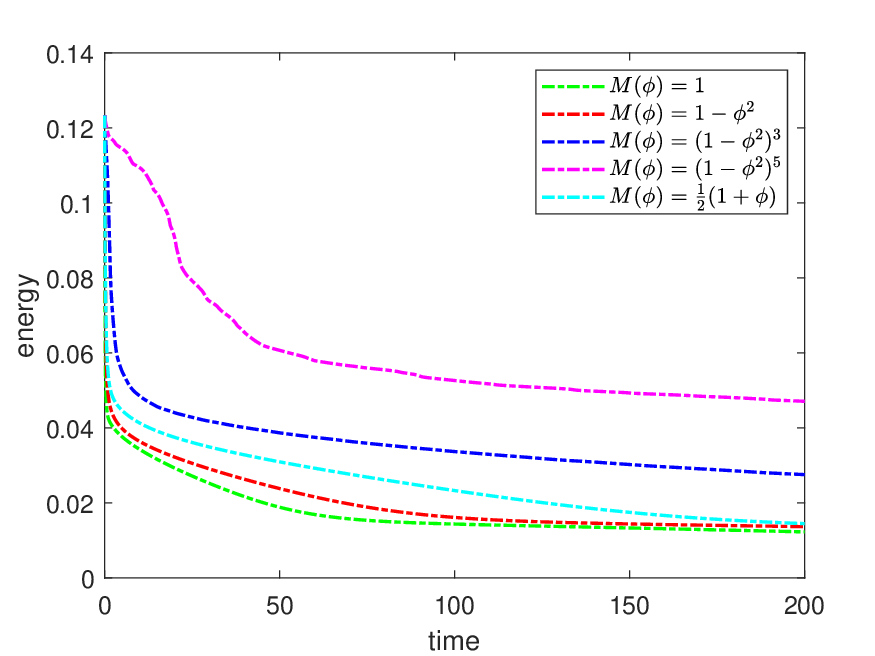}
}%
\setlength{\abovecaptionskip}{0.0cm}
\setlength{\belowcaptionskip}{0.0cm}
\caption{\small The maximum norm and energy of simulated solutions computed by the DsCN scheme for different mobilities.}
\label{figEx4_3}
\end{figure}
\begin{figure}[!t]
\vspace{-12pt}
\centering
\subfigure[maximum value]
{
\includegraphics[width=0.4\textwidth]{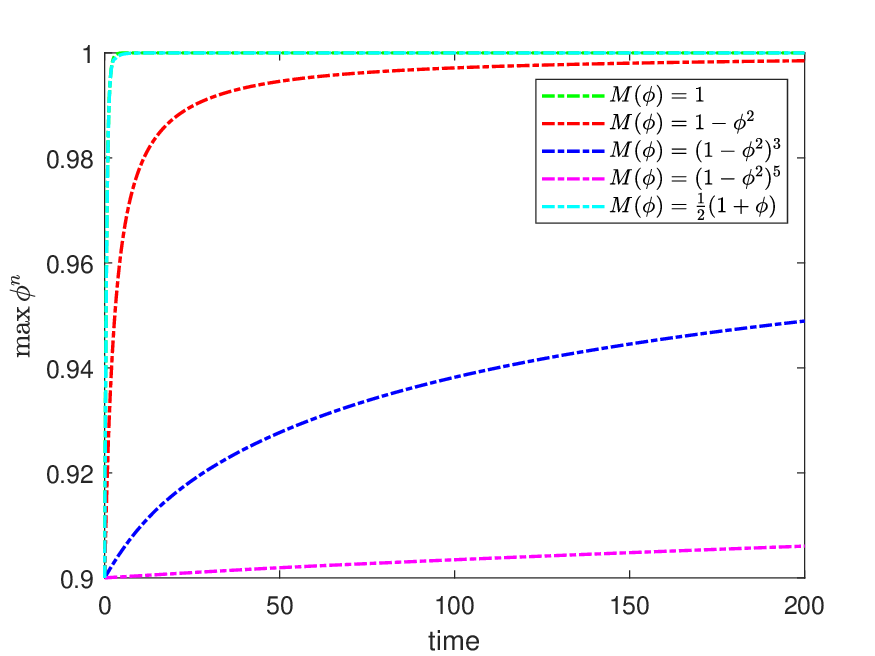}
}%
\subfigure[minimum value]
{
\includegraphics[width=0.4\textwidth]{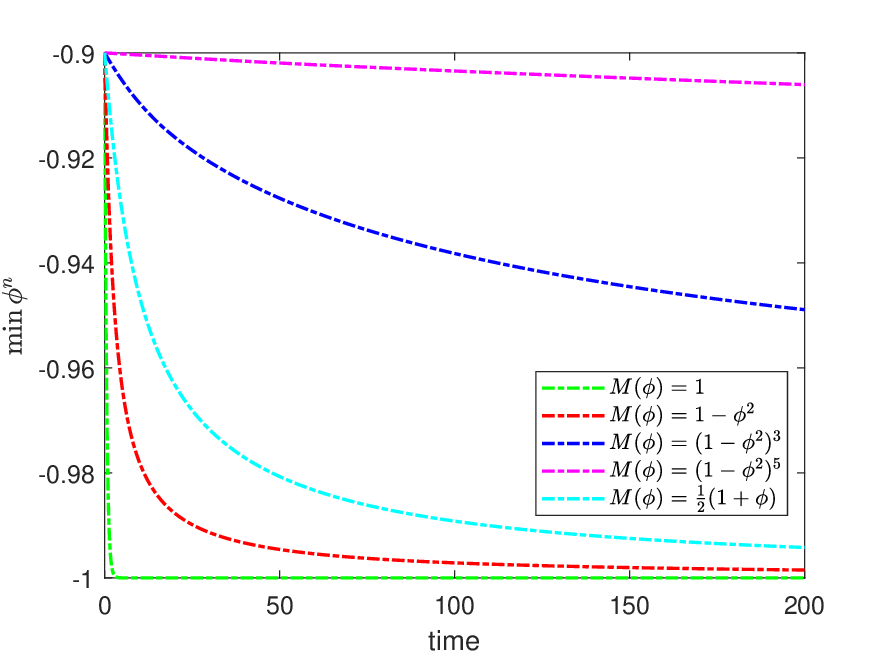}
}%
\setlength{\abovecaptionskip}{0.0cm}
\setlength{\belowcaptionskip}{0.0cm}
\caption{\small The maximum and minimum values of simulated solutions computed by the DsCN scheme for different mobilities.}
\label{figEx4_4}
\end{figure}

\begin{figure}[!t]
\vspace{-10pt}
\centering
\subfigure[iso-surfaces]
{
\begin{minipage}[t]{0.23\linewidth}
\centering
\includegraphics[width=1.3in]{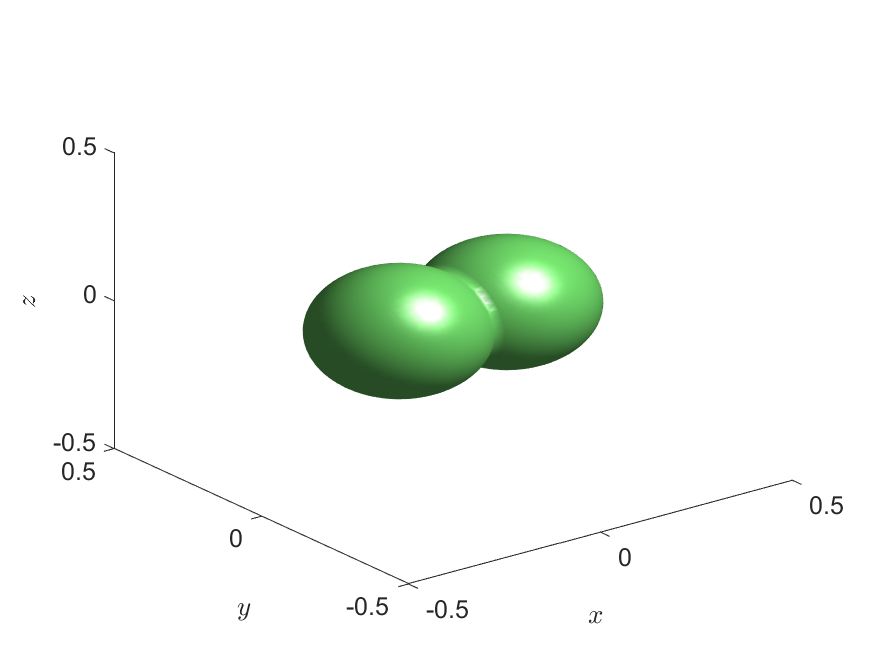}
\includegraphics[width=1.3in]{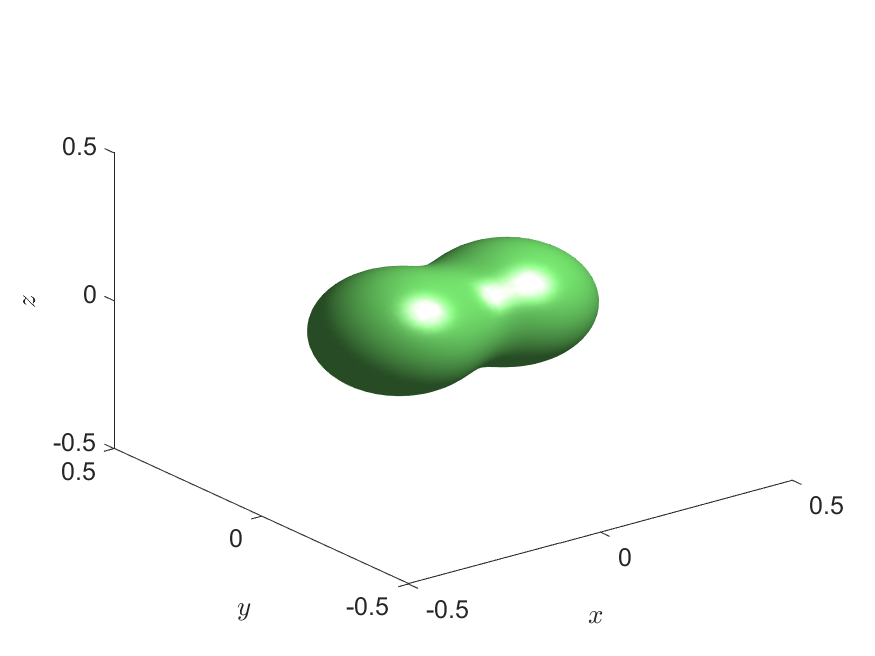}
\includegraphics[width=1.3in]{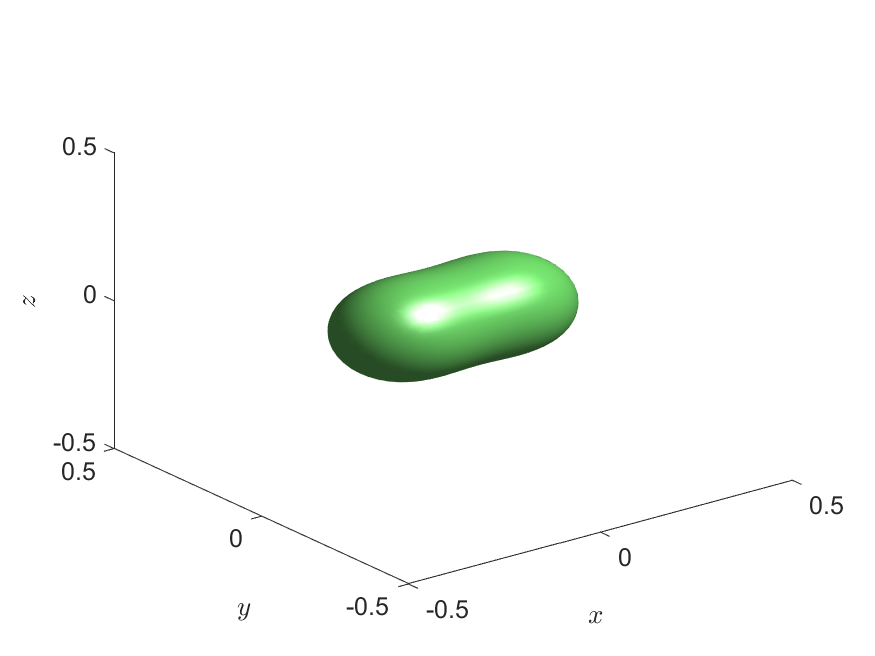}
\includegraphics[width=1.3in]{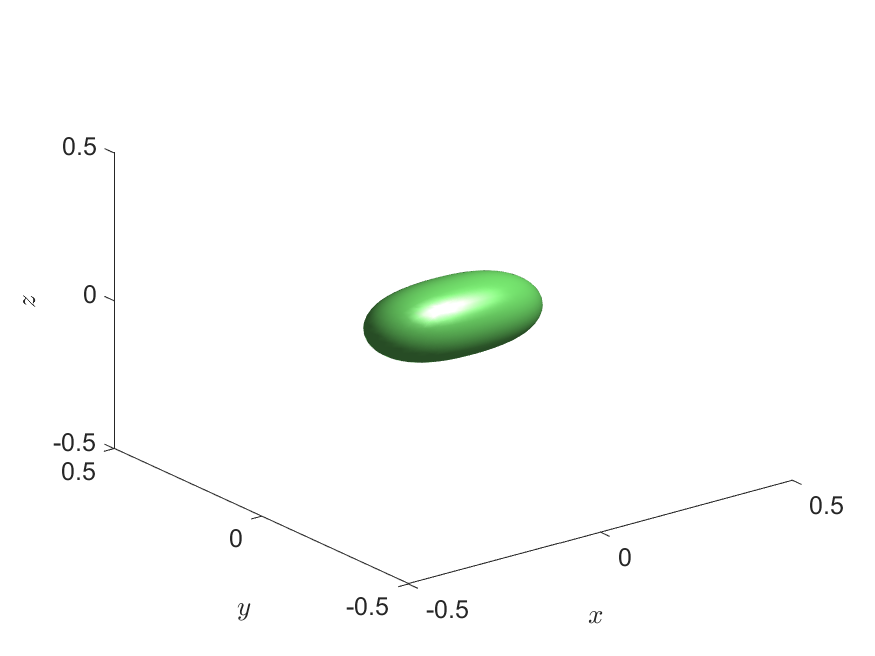}
\includegraphics[width=1.3in]{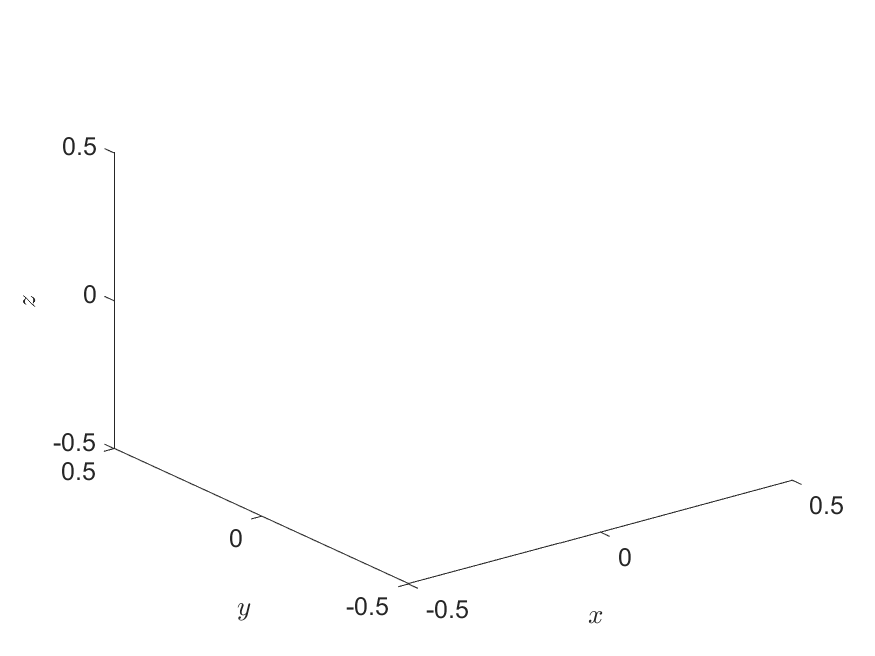}
\end{minipage}%
}%
\subfigure[slices ($x=0$)]
{
\begin{minipage}[t]{0.23\linewidth}
\centering
\includegraphics[width=1.3in]{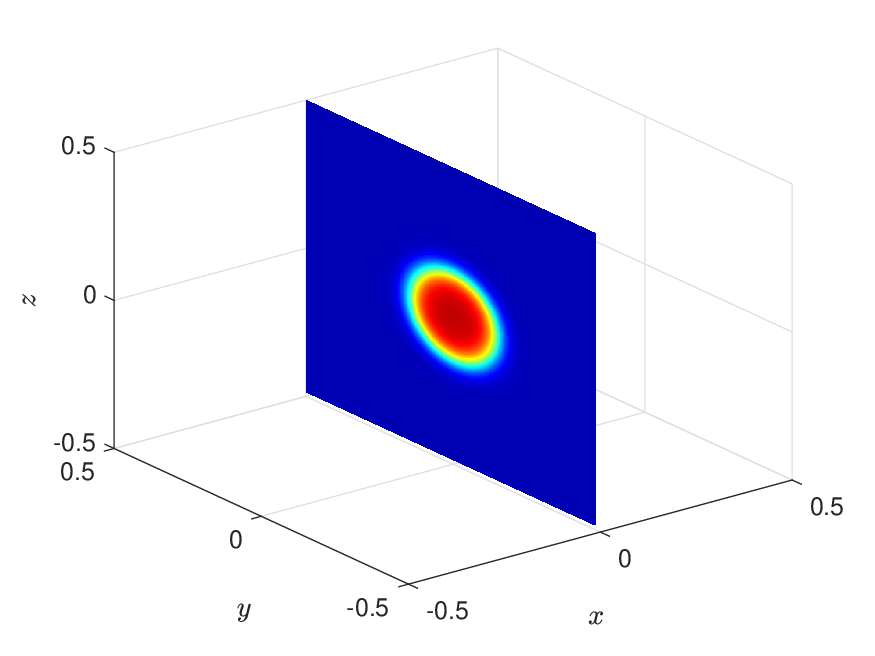}
\includegraphics[width=1.3in]{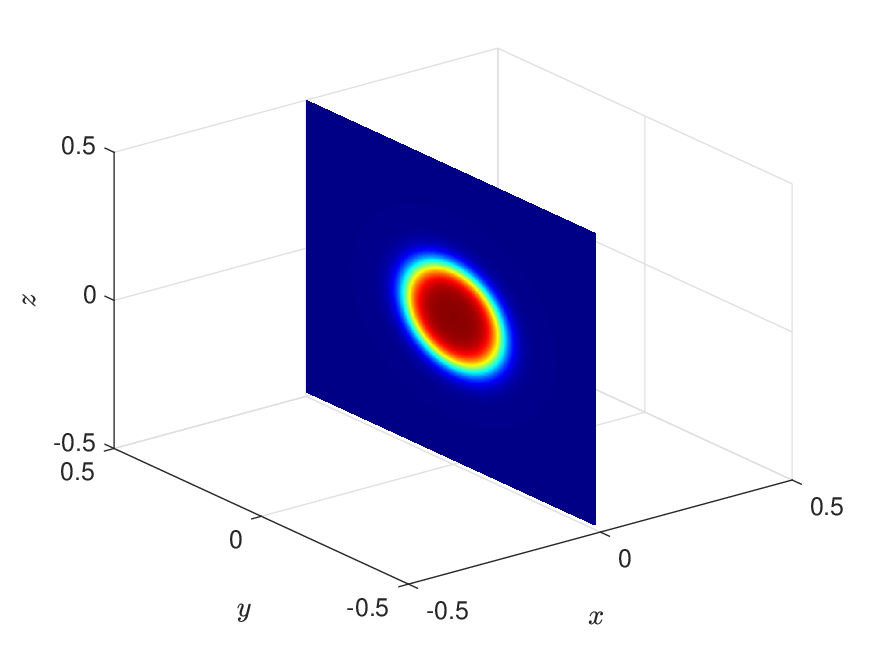}
\includegraphics[width=1.3in]{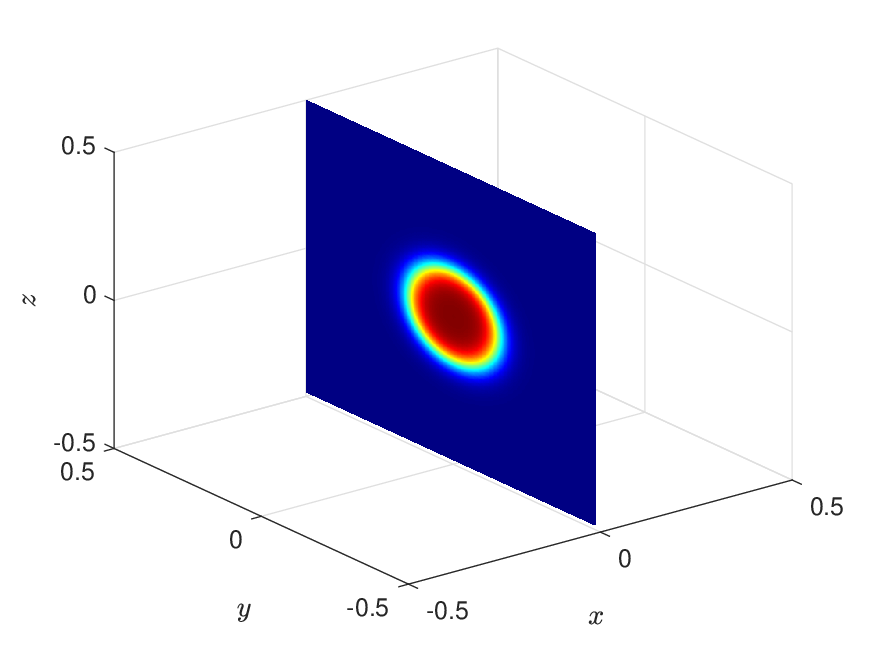}
\includegraphics[width=1.3in]{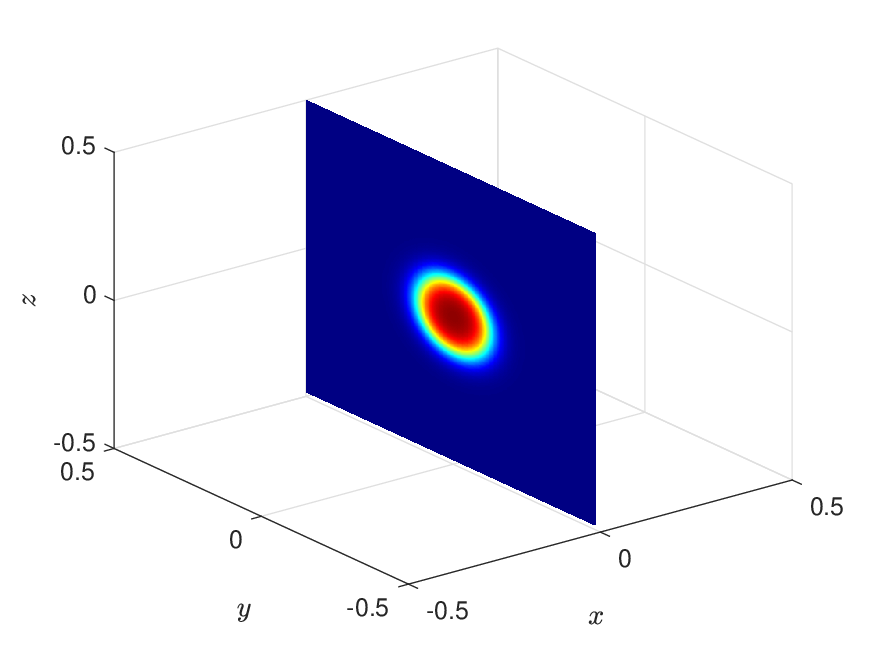}
\includegraphics[width=1.3in]{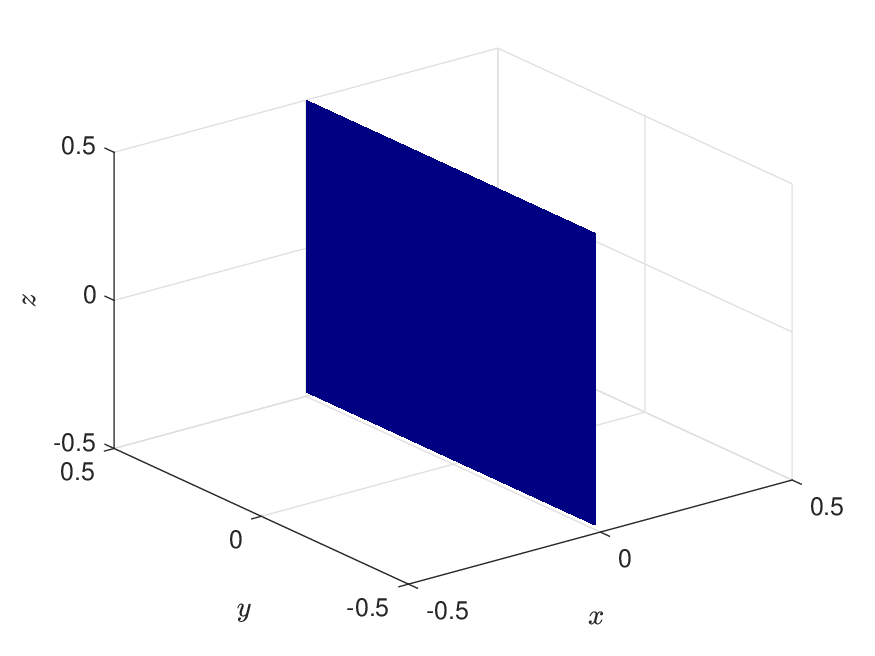}
\end{minipage}%
}%
\subfigure[slices ($y=0$)]
{
\begin{minipage}[t]{0.23\linewidth}
\centering
\includegraphics[width=1.3in]{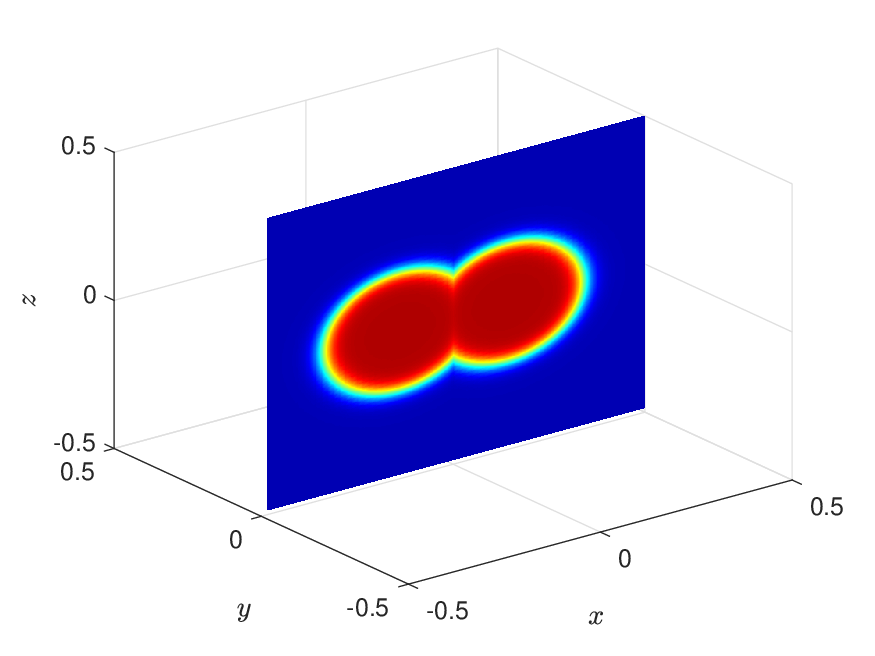}
\includegraphics[width=1.3in]{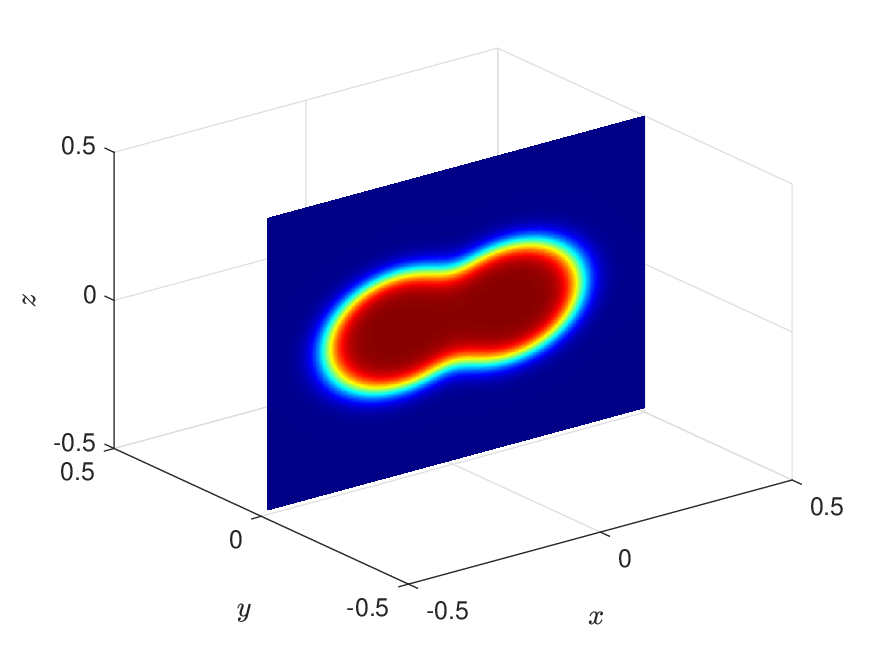}
\includegraphics[width=1.3in]{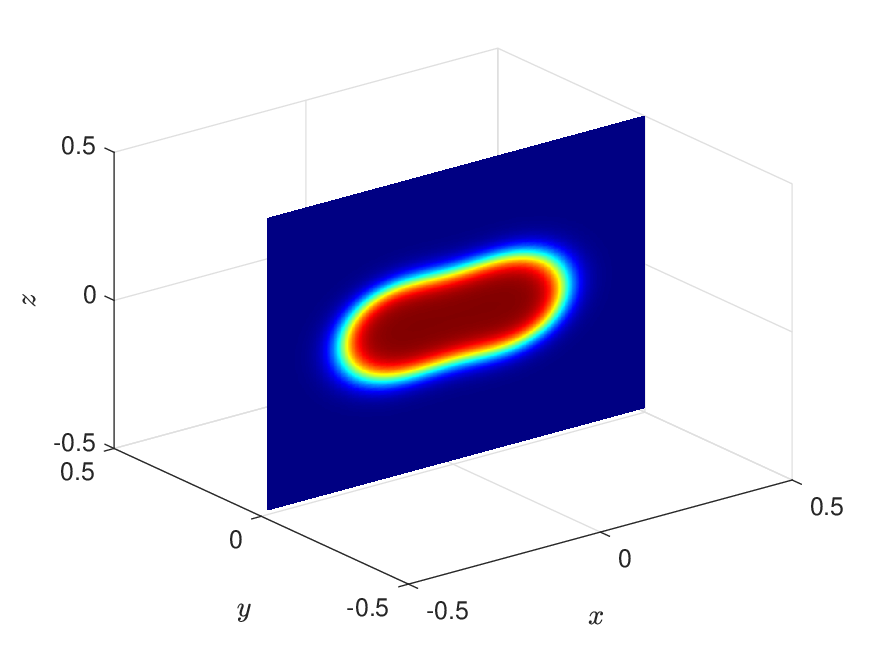}
\includegraphics[width=1.3in]{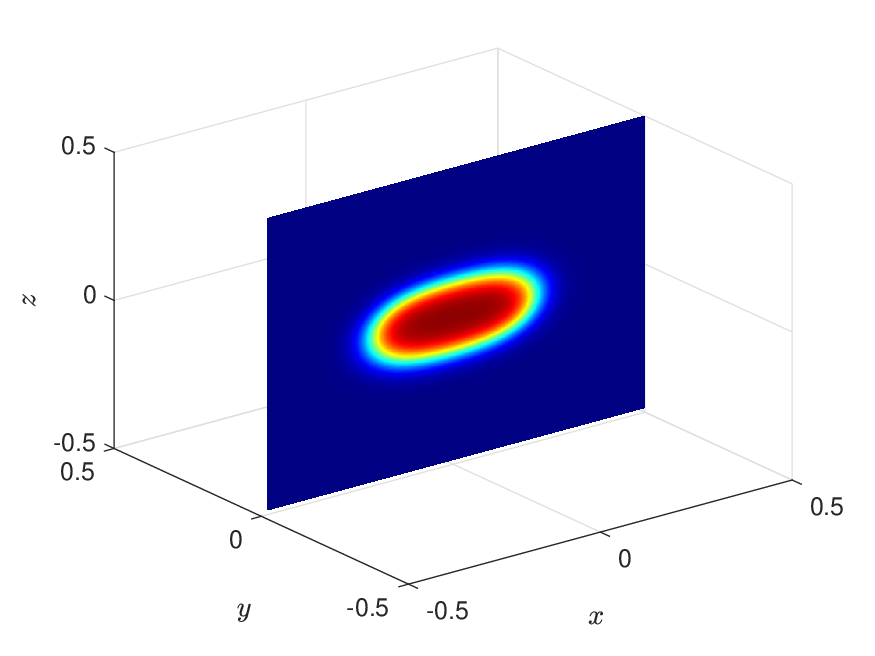}
\includegraphics[width=1.3in]{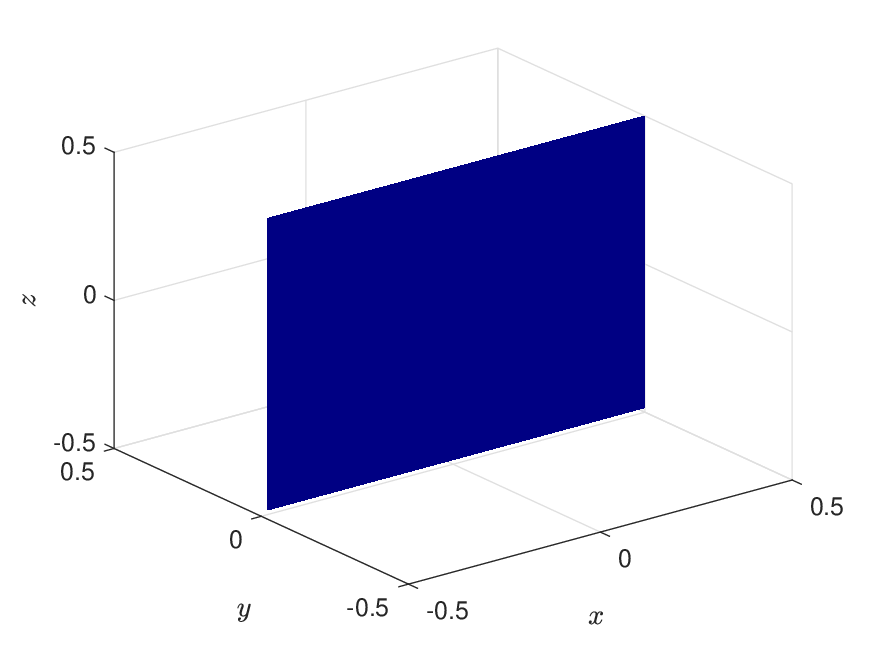}
\end{minipage}%
}%
\subfigure[slices ($z=0$)]
{
\begin{minipage}[t]{0.23\linewidth}
\centering
\includegraphics[width=1.3in]{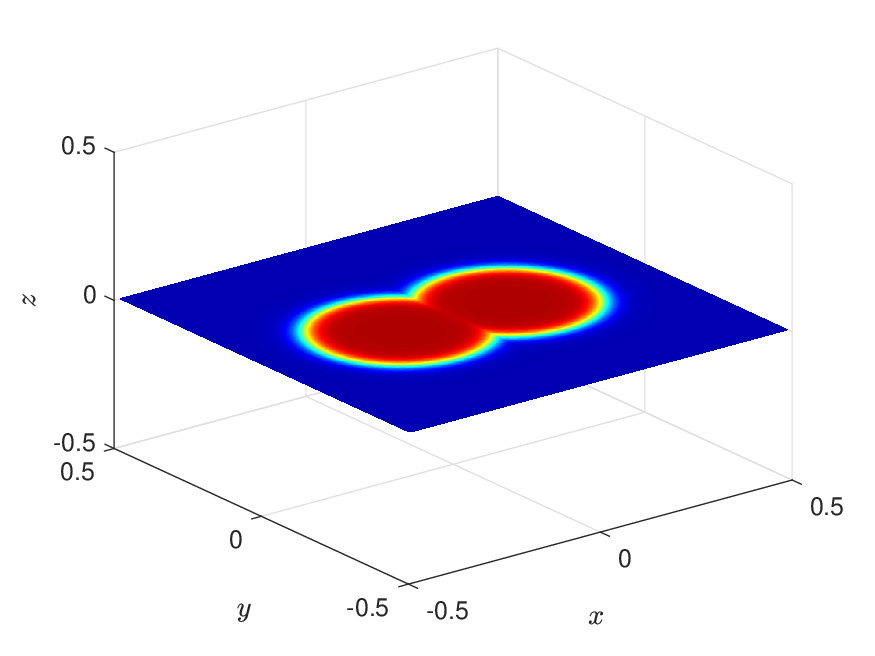}		
\includegraphics[width=1.3in]{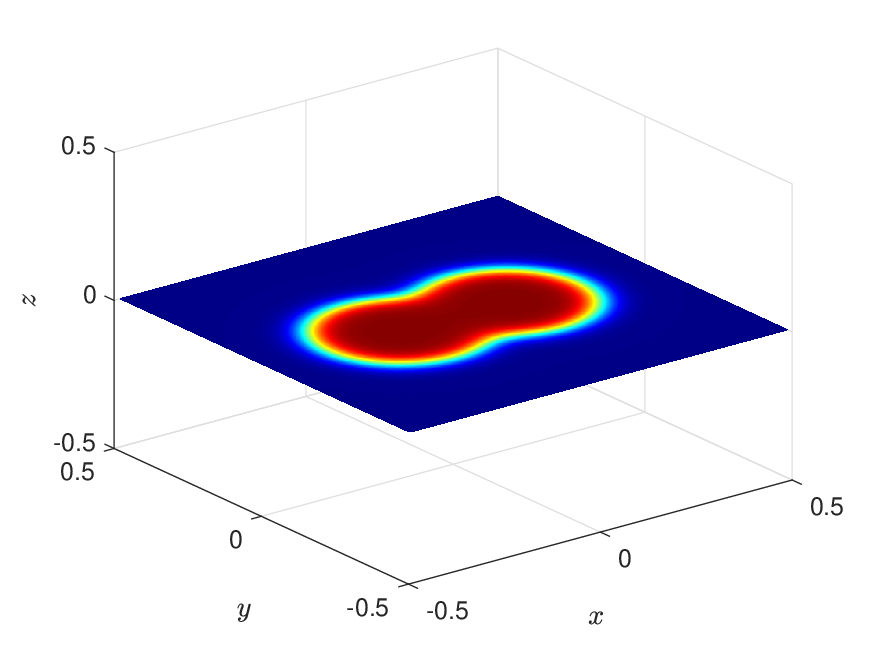}
\includegraphics[width=1.3in]{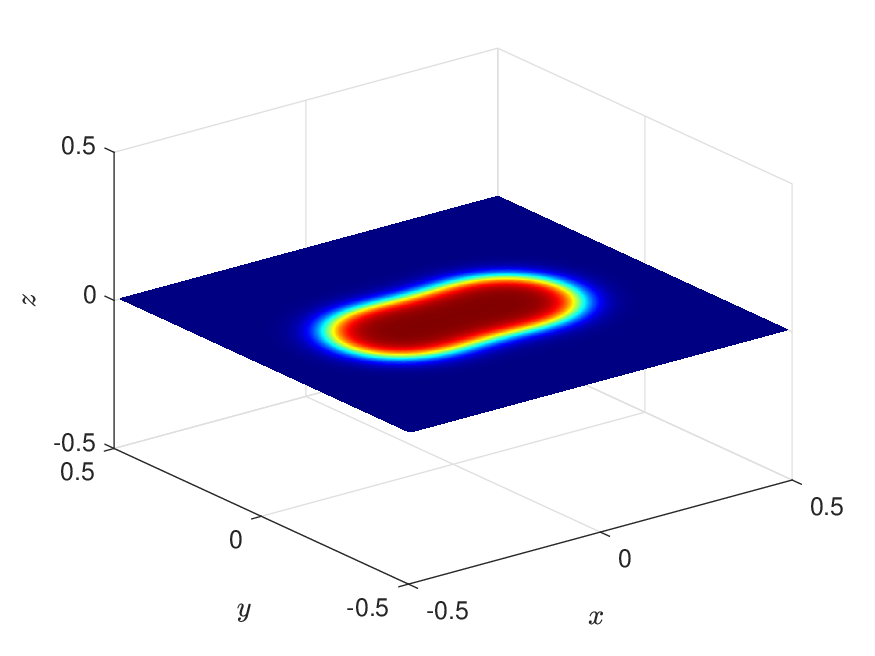}
\includegraphics[width=1.3in]{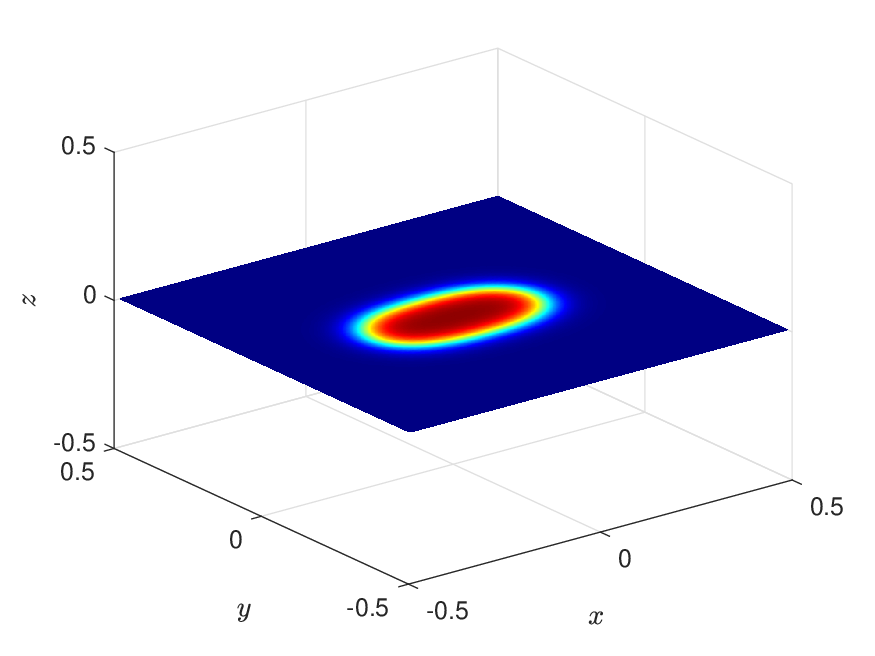}
\includegraphics[width=1.3in]{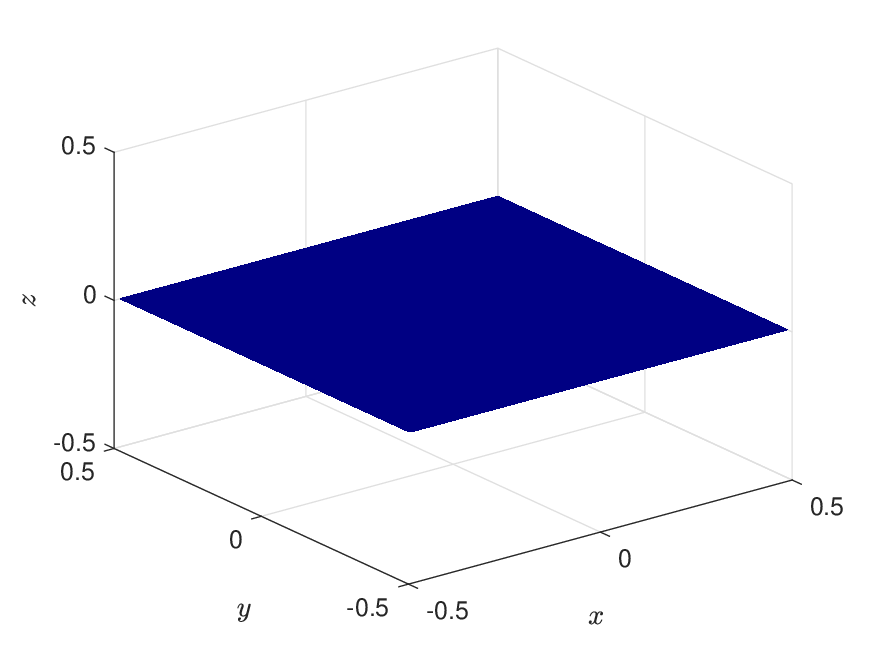}
\end{minipage}%
}
\setlength{\abovecaptionskip}{0.0cm}
\setlength{\belowcaptionskip}{0.0cm}
\caption{\small Plots of the iso-surfaces (value 0), and corresponding slices of the numerical solution $\phi$ at different time instants $ t = 0, 1.01, 5.01, 10.01, 20.08 $ (from top to bottom) obtained by the adaptive DsCN scheme: $\mm(\phi) =1$.} \label{figEx5_1}
\end{figure}
\subsection{3D bubble merging}
Finally, we consider the three-dimensional (3D) Allen--Cahn model on the domain $ \Omega = (-0.5, 0.5)^{3} $, with $ \varepsilon = 0.03 $ and two types of mobility functions: $ \mm(\phi) = 1 $ and $ \mm(\phi) = 1 - \phi^{2} $. The adaptive DsCN scheme, equipped with stabilization parameter $ S_{2} = \left( S_{1}/4 + 3 \varepsilon^{2}/2h^2 \right)^{2} \approx 36.3561 $, is employed to simulate the 3D bubble merging with the following initial condition
$$
\phi_{\text{init}}( x, y, z ) = \max\{ \phi_{1}, \phi_{2} \}; \ \phi_{i} = 0.9 \tanh\Big( \f{ 0.2 - \sqrt{ ( x \pm 0.14 )^2 + y^2 + z^2 } }{ \varepsilon } \Big), ~ i=1,2.
$$

\begin{figure}[!t]
\vspace{-12pt}
\centering
\subfigure[iso-surfaces]
{
\begin{minipage}[t]{0.23\linewidth}
\centering
\includegraphics[width=1.3in]{Ex5_u_T0.eps}
\includegraphics[width=1.3in]{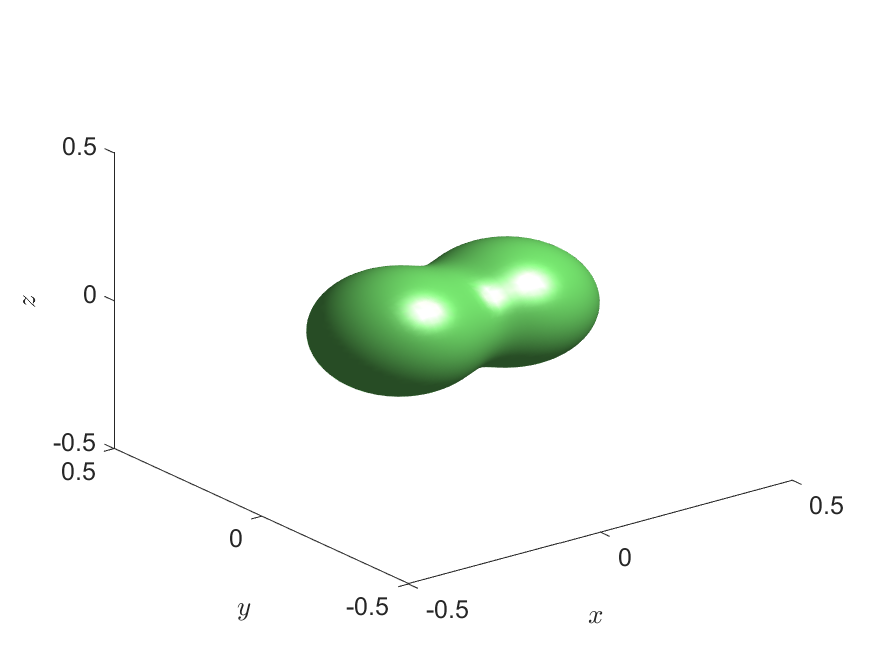}
\includegraphics[width=1.3in]{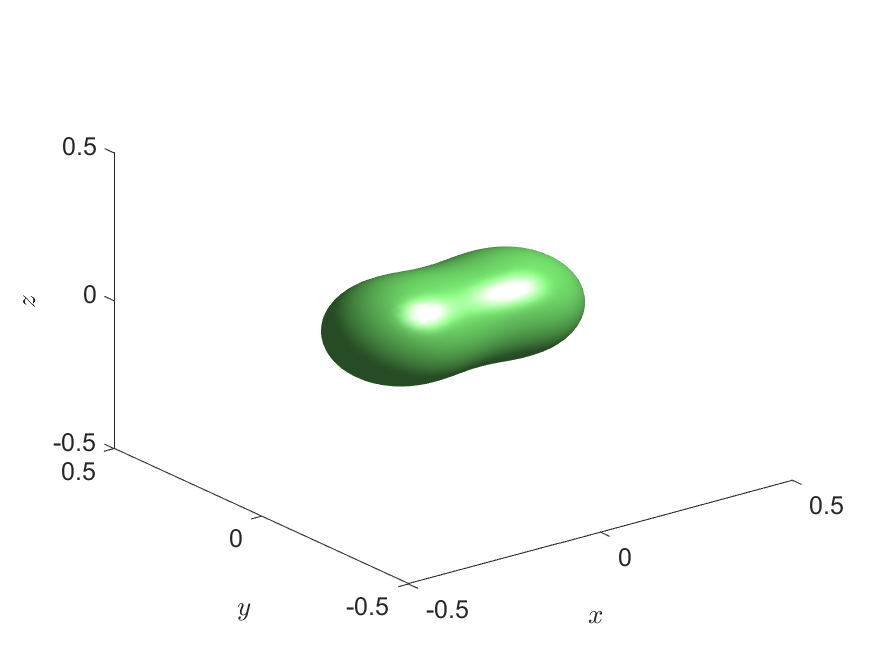}
\includegraphics[width=1.3in]{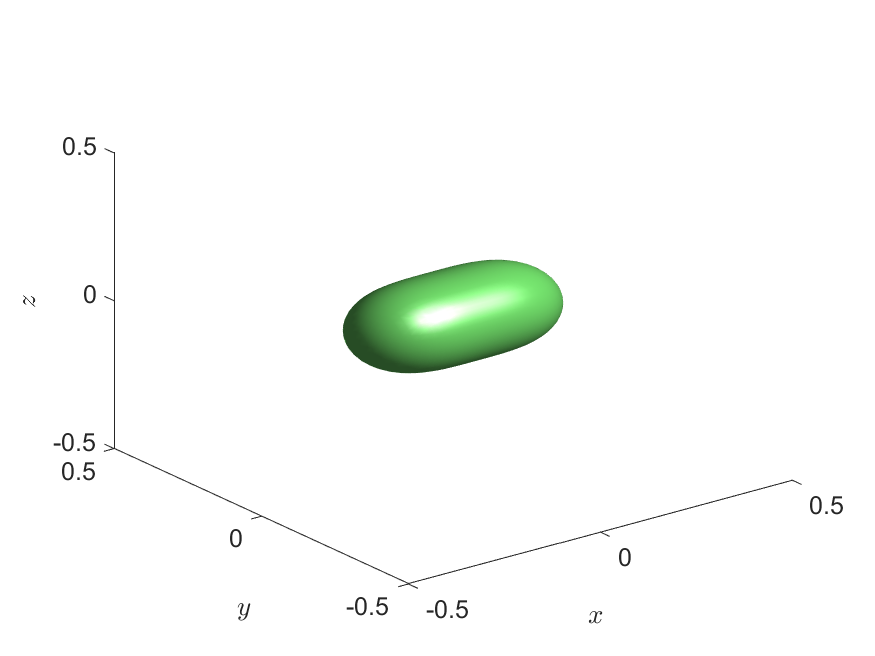}
\includegraphics[width=1.3in]{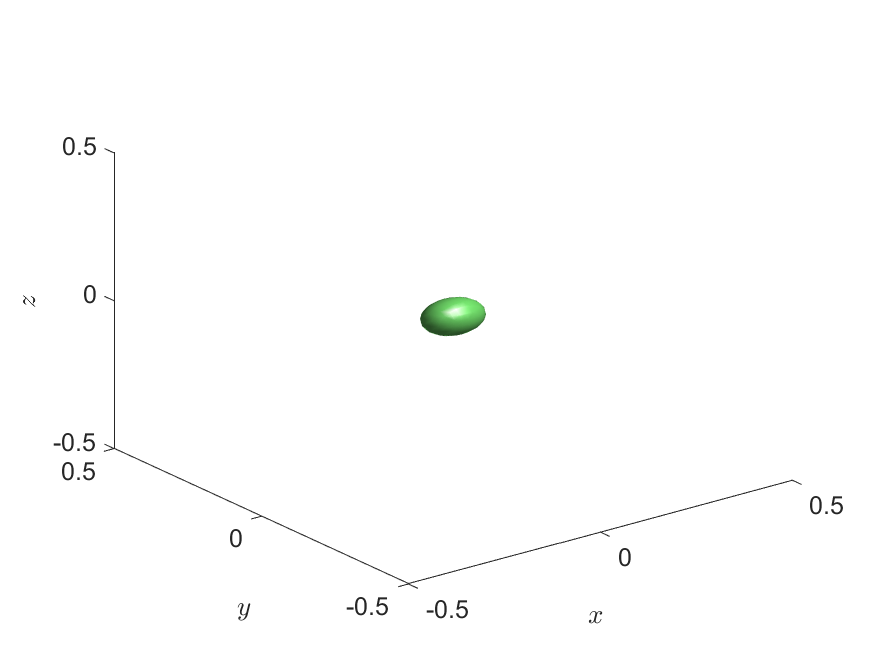}
\end{minipage}%
}%
\subfigure[slices ($x=0$)]
{
\begin{minipage}[t]{0.23\linewidth}
\centering
\includegraphics[width=1.3in]{u_x0_T0.eps}
\includegraphics[width=1.3in]{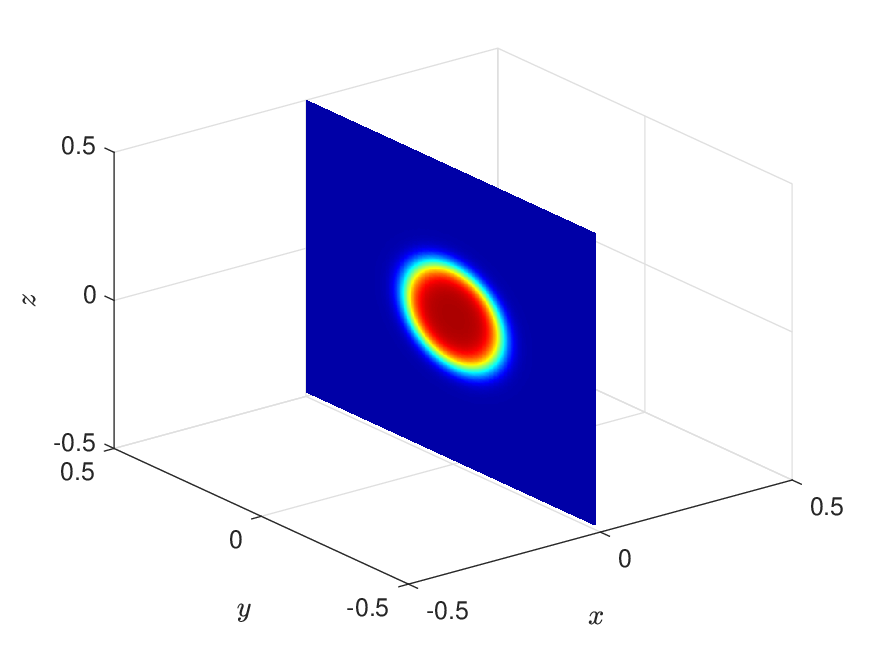}
\includegraphics[width=1.3in]{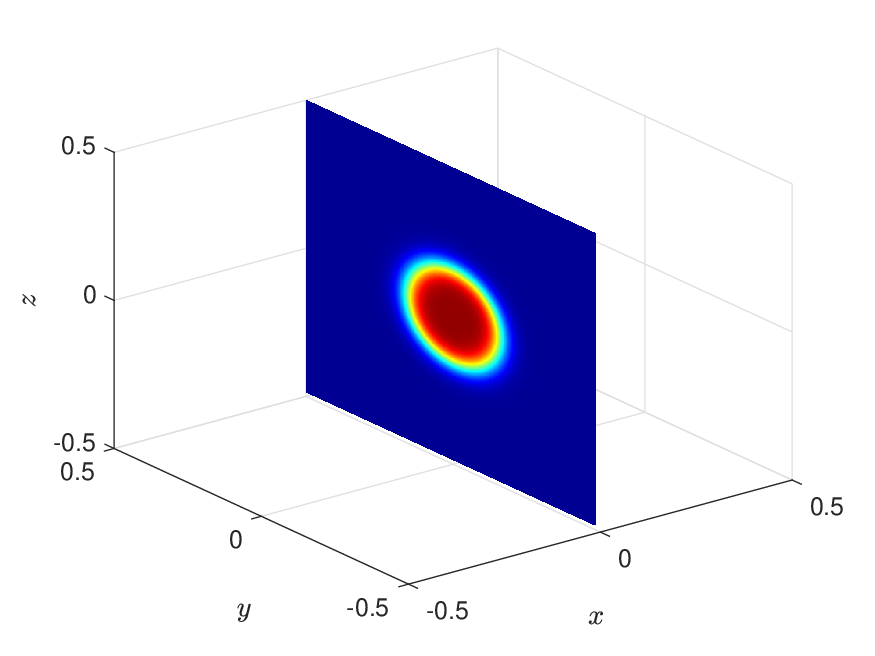}
\includegraphics[width=1.3in]{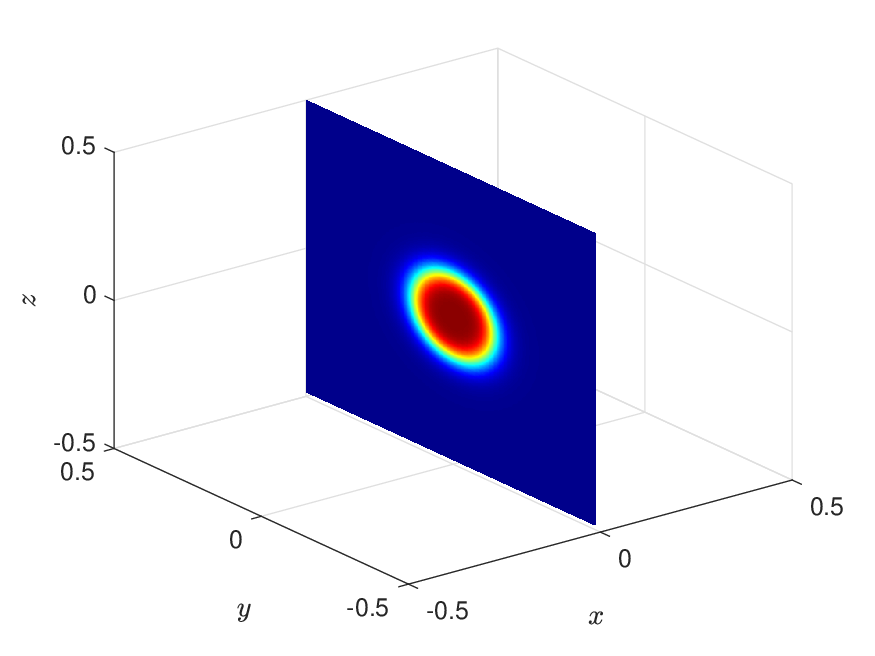}
\includegraphics[width=1.3in]{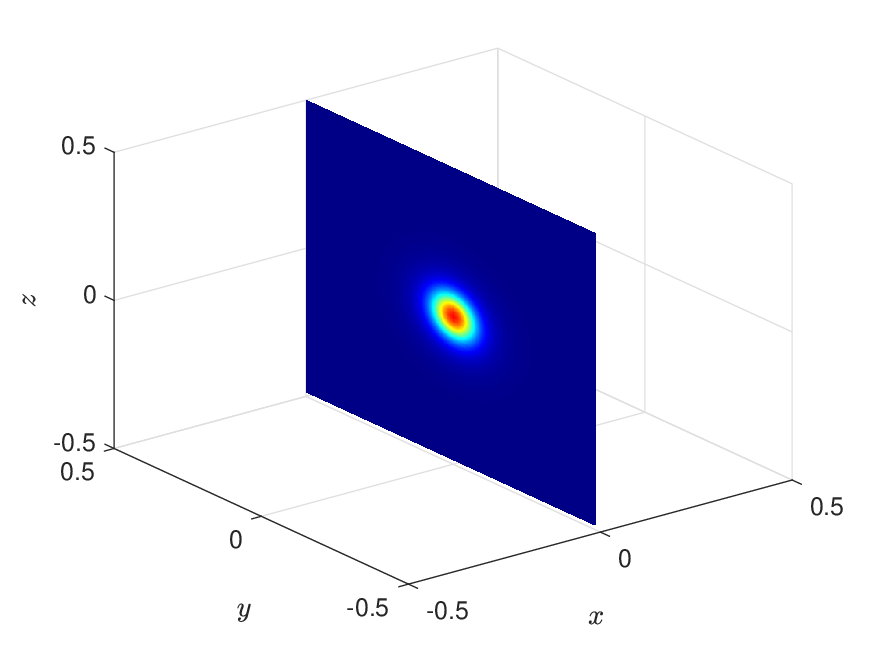}
\end{minipage}%
}%
\subfigure[slices ($y=0$)]
{
\begin{minipage}[t]{0.23\linewidth}
\centering
\includegraphics[width=1.3in]{u_y0_T0.eps}
\includegraphics[width=1.3in]{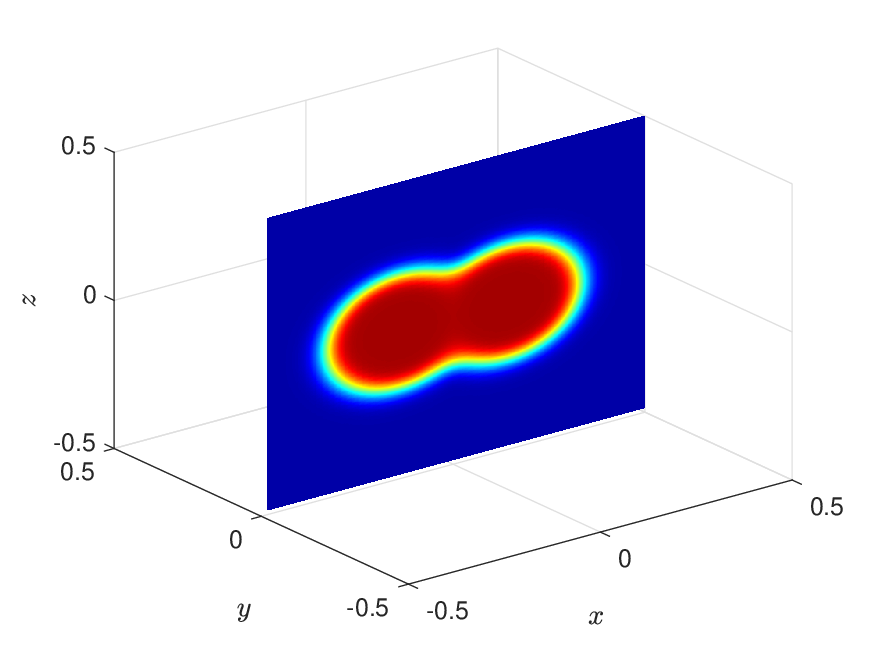}
\includegraphics[width=1.3in]{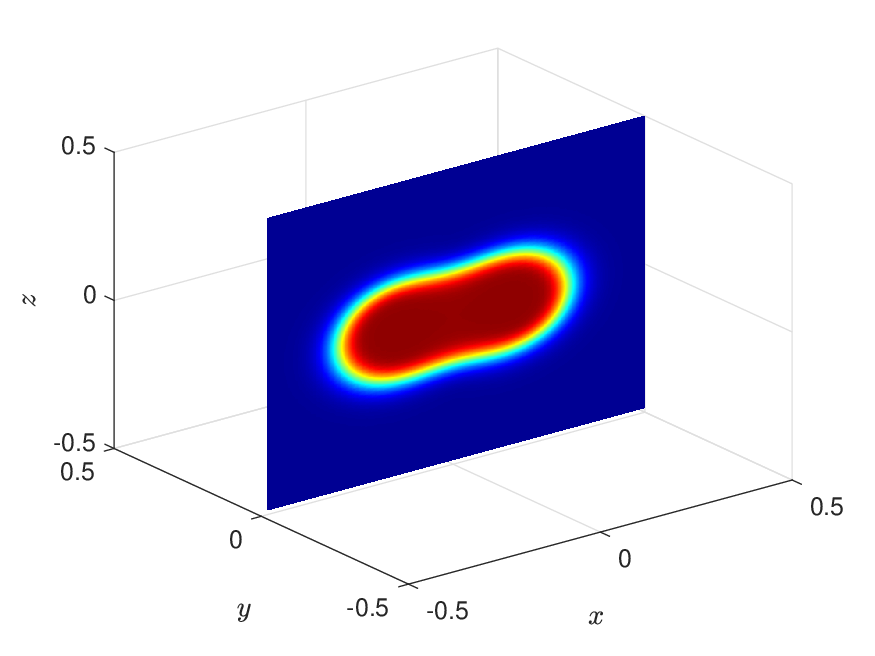}
\includegraphics[width=1.3in]{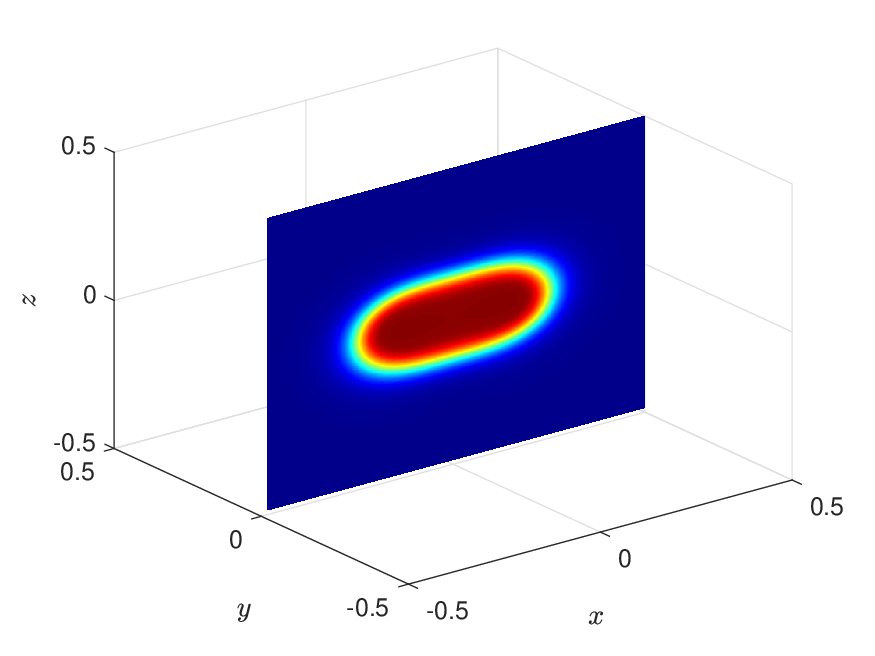}
\includegraphics[width=1.3in]{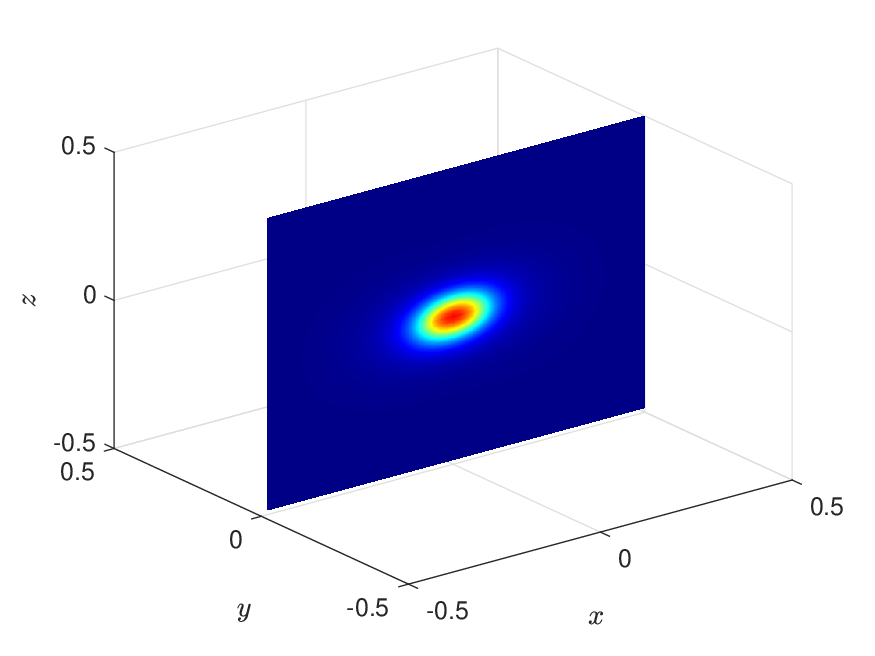}
\end{minipage}%
}%
\subfigure[slices ($z=0$)]
{
\begin{minipage}[t]{0.23\linewidth}
\centering
\includegraphics[width=1.3in]{u_z0_T0.eps}		
\includegraphics[width=1.3in]{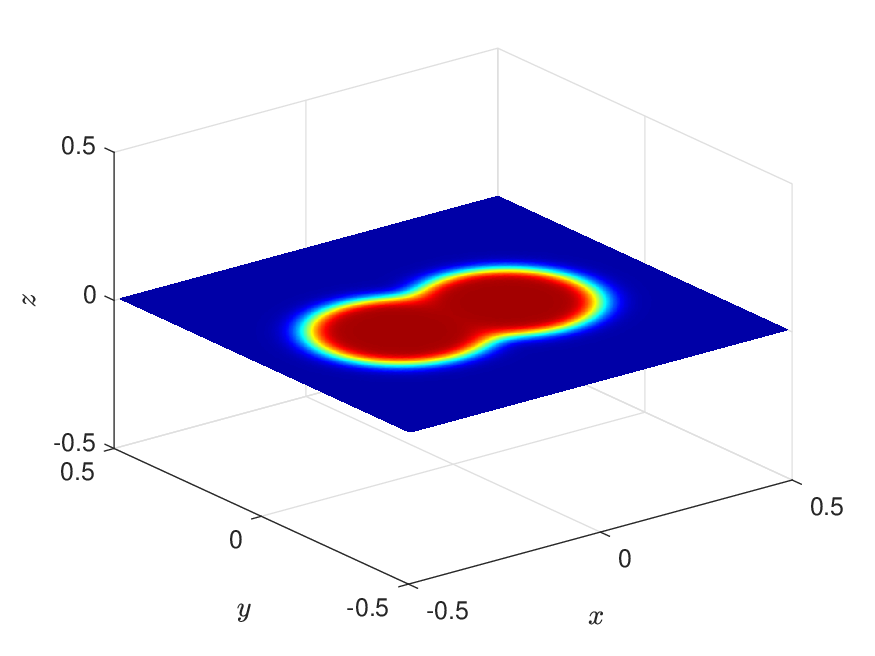}
\includegraphics[width=1.3in]{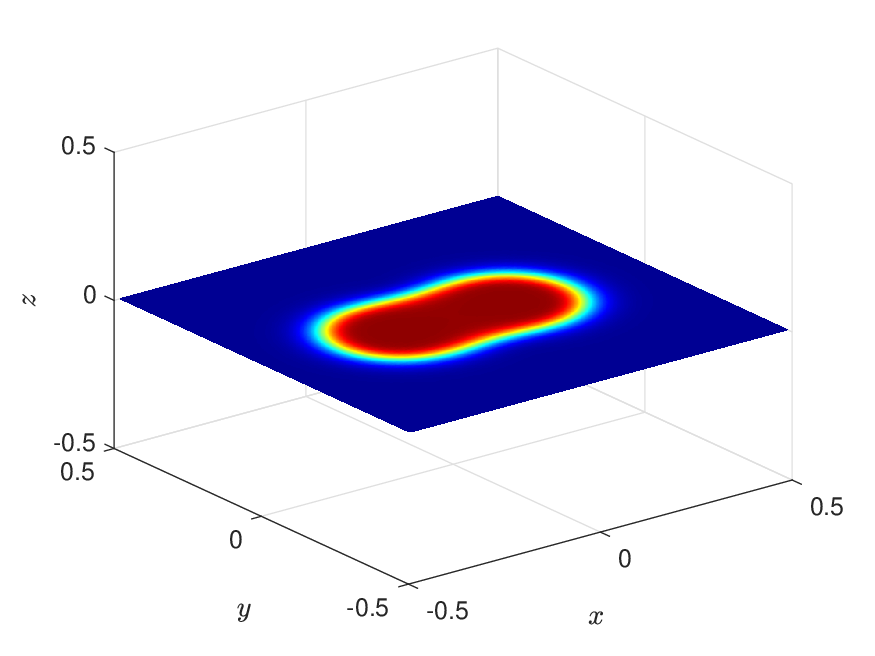}
\includegraphics[width=1.3in]{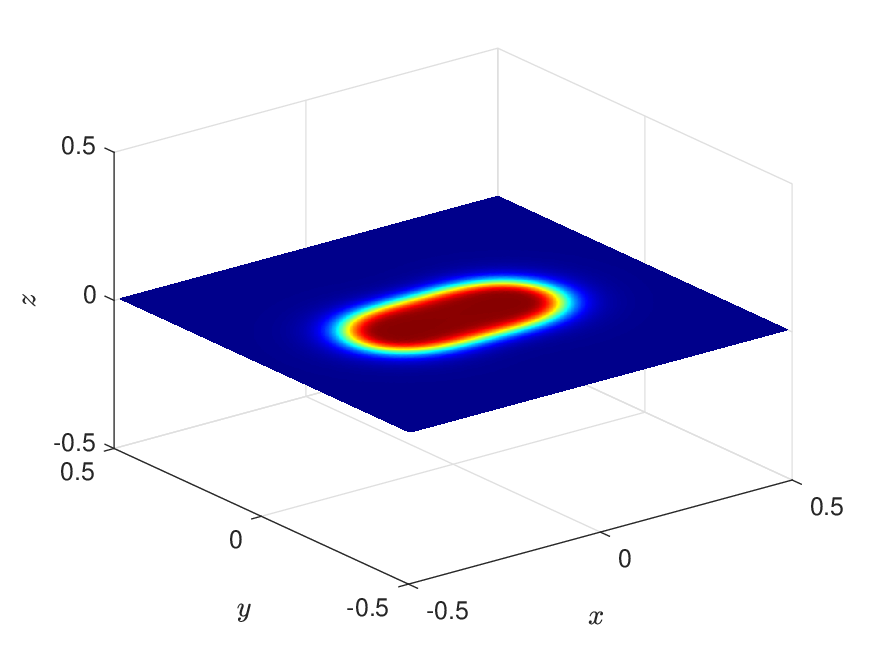}
\includegraphics[width=1.3in]{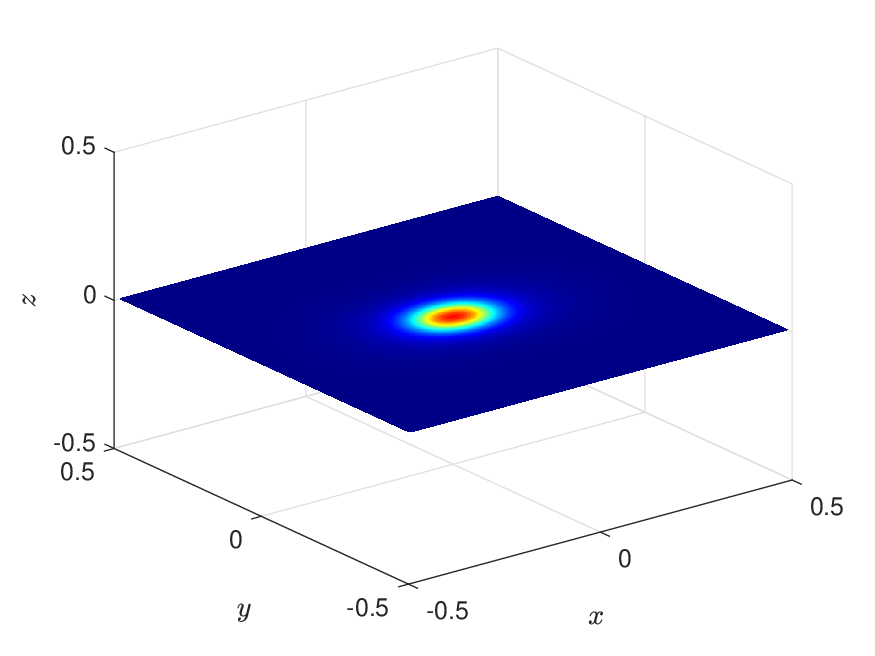}
\end{minipage}%
}%
\setlength{\abovecaptionskip}{0.0cm}
\setlength{\belowcaptionskip}{0.0cm}
\caption{\small Plots of the iso-surfaces (value $0$), and corresponding slices of the numerical solution $\phi$ at different time instants $ t = 0, 1, 5.02, 10.02, 20.02 $ (from top to bottom) obtained by the adaptive DsCN scheme: $\mm(\phi) = 1 - \phi^{2}$.} \label{figEx5_2}
\end{figure}
In this test, the adaptive parameters in \eqref{Alg1:adaptive} are selected as $ \tau_{\max} = 0.1$, $\tau_{\min} = 0.01 $ and $ \alpha = 10^{7} $. Spatial discretization is performed using a uniform mesh with $ 64 \times 64 \times 64 $ grid points. The iso-surfaces (value $0$) and corresponding slices of the numerical solution $\phi$ at different time instants are presented in Figure \ref{figEx5_1} for the constant mobility and in Figure \ref{figEx5_2} for the degenerate mobility. These results clearly demonstrate that: (i) the two initial bubbles gradually shrink and merge into one smaller single bubble; and (ii) the evolution process proceeds more slowly under the degenerate mobility setting than under the constant mobility, which is totally consistent with the two-dimensional case (see subsection \ref{sub:Eff_mobi}). Furthermore, as evidenced by the time evolution curves in Figure \ref{figEx5_3}, the scheme also preserves both the discrete energy stability and MBP for 3D simulation.

\begin{figure}[!t]
\vspace{-10pt}
\centering
\subfigure[maximum norm]
{
\includegraphics[width=0.4\textwidth]{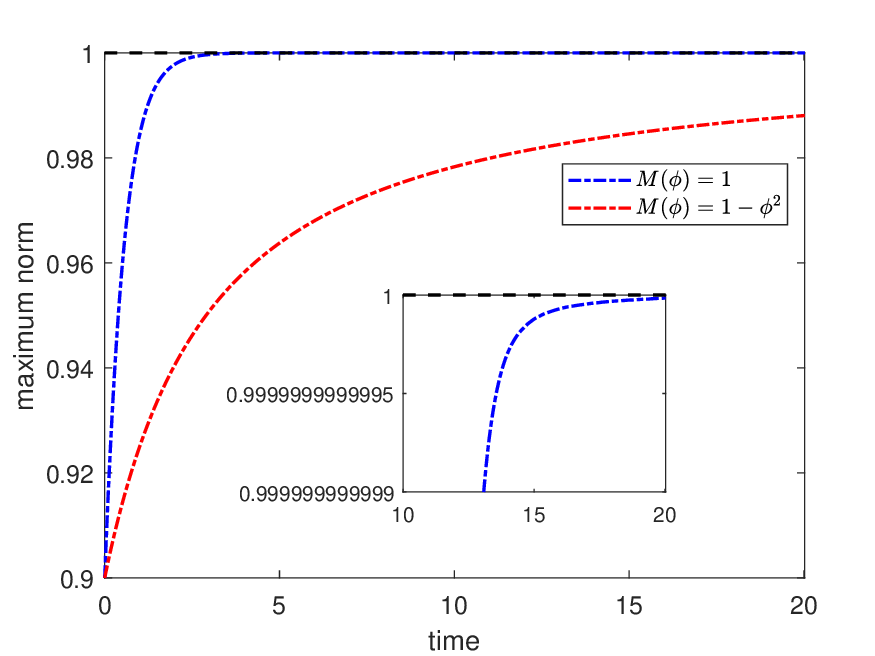}
}%
\subfigure[energy]
{
\includegraphics[width=0.4\textwidth]{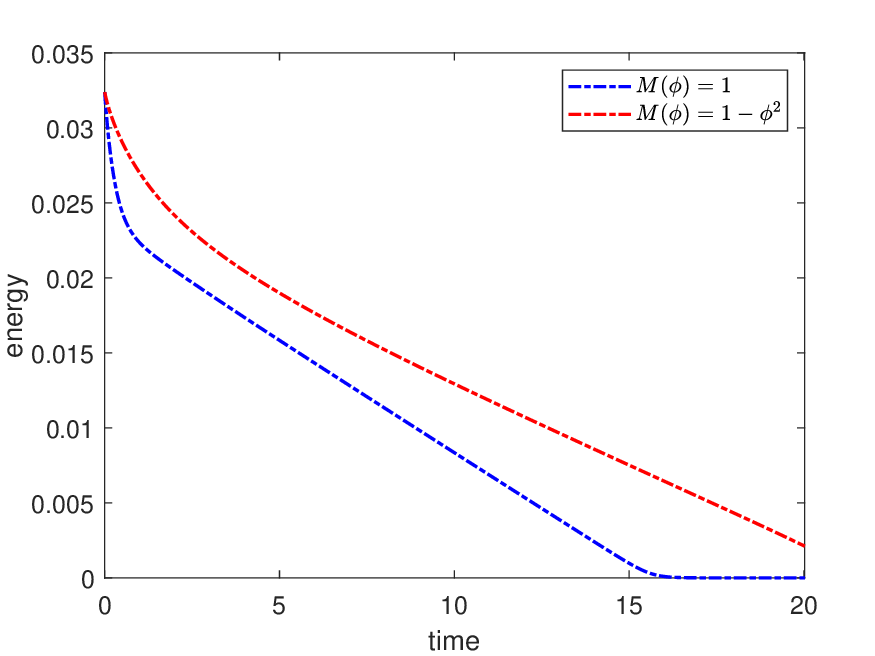}
}%
\setlength{\abovecaptionskip}{0.0cm}
\setlength{\belowcaptionskip}{0.0cm}
\caption{\small The maximum norm and energy of simulated solutions computed by the DsCN scheme for different mobilities.}
\label{figEx5_3}
\vspace{-8pt}
\end{figure}

\section{Concluding remarks}\label{sec:Conclu}
In this work, we introduce a dynamic stabilization term and put forward a first-order linear DsBE scheme for the Allen--Cahn equation with general mobility. We prove that this scheme unconditionally preserves the MBP and adheres to the original energy dissipation law. To obtain second-order, linear, energy-stable MBP-preserving schemes, we develop a novel prediction–correction framework that requires solving only a single linear system per time step, thereby reducing computational cost. Rigorous maximum-norm error estimates are established for both schemes. A notable distinction from prior studies \cite{CMS_2023_Cai,MOC_2023_Ju,ANM_2025_Hou} is that our theoretical results remain valid even in the presence of degenerate mobility, underscoring the robustness of the analysis. Finally, numerical experiments for the Allen--Cahn equation with various mobility laws corroborate the theoretical findings and demonstrate the strong performance of the proposed schemes. In particular, when coupled with an adaptive time-stepping strategy, the schemes efficiently resolve multiscale dynamics without compromising accuracy, making them well-suited for long-time simulations.


\bibliographystyle{amsplain}
\bibliography{Ref_MAC_CN}

\providecommand{\bysame}{\leavevmode\hbox to3em{\hrulefill}\thinspace}
\providecommand{\MR}{\relax\ifhmode\unskip\space\fi MR }
\providecommand{\MRhref}[2]{%
  \href{http://www.ams.org/mathscinet-getitem?mr=#1}{#2}
}
\providecommand{\href}[2]{#2}
\begin{thebibliography}{10}

\bibitem{Acta_Allen_1979}
A.M. Allen and J.W. Cahn, \emph{A microscopic theory for antiphase boundary
  motion and its application to antiphase domain coarsening}, Acta Metall.
  \textbf{27} (1979), 1085--1095.

\bibitem{ANM_2004_Benes}
M.~Benes, V.~Chalupecky, and K.~Mikula, \emph{Geometrical image segmentation by
  the {A}llen--{C}ahn equation}, Appl. Numer. Math. \textbf{51} (2004),
  187--205.

\bibitem{AMM_1994_Cahn}
J.W. Cahn and J.E. Taylor, \emph{Surface motion by surface diffusion}, Acta.
  Metall. Mater. \textbf{42} (1994), 1045--1063.

\bibitem{CMS_2023_Cai}
Y.~Cai, L.~Ju, R.~Lan, and J.~Li, \emph{Stabilized exponential time
  differencing schemes for the convective {A}llen--{C}ahn equation}, Commun.
  Math. Sci. \textbf{21} (2023), 127--150.

\bibitem{CMAME_2022_Huang}
Q.~Cheng and J.~Shen, \emph{A new {L}agrange multiplier approach for
  constructing structure preserving schemes, {I}. {P}ositivity preserving},
  Comput. Methods Appl. Mech. Engrg. \textbf{391} (2022), 114585.

\bibitem{SINUM_2022_Huang}
Q.~Cheng and J.~Shen, \emph{A new {L}agrange multiplier approach for
  constructing structure preserving schemes, {II}. {B}ound preserving}, SIAM J.
  Numer. Anal. \textbf{60} (2022), 970--998.

\bibitem{SIAM_2012_Du}
S.~Dai and Q.~Du, \emph{Motion of interfaces governed by the {C}ahn--{H}illiard
  equation with highly disparate diffusion mobility}, SIAM J. Appl. Math.
  \textbf{72} (2012), 1818--1841.

\bibitem{JCP_2016_Du}
S.~Dai and Q.~Du, \emph{Computational studies of coarsening rates for the
  {C}ahn--{H}illiard equation with phase-dependent diffusion mobility}, J.
  Comput. Phys. \textbf{310} (2016), 85--108.

\bibitem{ARMA_2016_Du}
S.~Dai and Q.~Du, \emph{Weak solutions for the {C}ahn--{H}illiard equation with
  degenerate mobility}, Arch. Rational Mech. Anal. \textbf{219} (2016),
  1161--1184.

\bibitem{SINUM_2019_Du}
Q.~Du, L.~Ju, X.~Li, and Z.~Qiao, \emph{Maximum principle preserving
  exponential time differencing schemes for the nonlocal {A}llen--{C}ahn
  equation}, SIAM J. Numer. Anal. \textbf{57} (2019), 875--898.

\bibitem{SIREV_Du_2021}
Q.~Du, L.~Ju, X.~Li, and Z.~Qiao, \emph{Maximum bound principles for a class of
  semilinear parabolic equations and exponential time-differencing schemes},
  SIAM Rev. \textbf{63} (2021), 317--359.

\bibitem{NM_2003_Feng}
X.B. Feng and A.~Prohl, \emph{Numerical analysis of the {A}llen--{C}ahn
  equation and approximation for mean curvature flows}, Numer. Math.
  \textbf{94} (2003), 33--65.

\bibitem{JCP_2011_Gomez}
H.~Gomez and T.J.R. Hughes, \emph{Provably unconditionally stable, second-order
  time-accurate, mixed variational methods for phase-field models}, J. Comput.
  Phys. \textbf{230} (2011), 5310--5327.

\bibitem{JCP_2014_Xu}
R.~Guo, Y.~Xia, and Y.~Xu, \emph{An efficient fully-discrete local
  discontinuous {G}alerkin method for the {Cahn--Hilliard--Hele--Shaw} system},
  J. Comput. Phys. \textbf{264} (2014), 23--40.

\bibitem{MOC_2023_Ju}
D.~Hou, L.~Ju, and Z.~Qiao, \emph{A linear second-order maximum bound
  principle-preserving {BDF} scheme for the {A}llen--{C}ahn equation with a
  general mobility}, Math. Comput. \textbf{92} (2023), 2515--2542.

\bibitem{AAMM_2024_Hou}
D.~Hou, L.~Ju, and Z.~Qiao, \emph{A linear doubly stabilized {C}rank-{N}icolson
  scheme for the {A}llen--{C}ahn equation with a general mobility}, Adv. Appl.
  Math. Mech. \textbf{16} (2024), 1009--1038.

\bibitem{CMA_2025_Hou}
D.~Hou, H.~Liu, and L.~Ju, \emph{Unconditionally original energy-dissipative
  and {MBP}-preserving {C}rank-{N}icolson scheme for the {A}llen--{C}ahn
  equation with general mobility}, Comput. Math. Appl. \textbf{191} (2025),
  86--104.

\bibitem{ANM_2025_Hou}
D.~Hou, T.~Zhang, and H.~Zhu, \emph{A linear second order unconditionally
  maximum bound principle-preserving scheme for the {A}llen--{C}ahn equation
  with general mobility}, Appl. Numer. Math. \textbf{207} (2025), 222--243.

\bibitem{AML_2020_Hou}
T.~Hou and H.~Leng, \emph{Numerical analysis of a stabilized
  {Crank-Nicolson}/{Adams-Bashforth} finite difference scheme for {Allen--Cahn}
  equations}, Appl. Math. Lett. \textbf{102} (2020), 106150.

\bibitem{JDG_1993_Ilmanen}
T.~Ilmanen, \emph{Convergence of the {A}llen--{C}ahn equation to brakke's
  motion by mean curvature}, J. Differential Geom. \textbf{38} (1993),
  417--461.

\bibitem{JCP_2021_Ju}
L.~Ju, X.~Li, and J.~Qiao, Z.and~Yang, \emph{Maximum bound principle preserving
  integrating factor {Runge-Kutta} methods for semilinear parabolic equations},
  J. Comput. Phys. \textbf{439} (2021), 110405.

\bibitem{SINUM_2022_Ju}
L.~Ju, X.~Li, and Z.~Qiao, \emph{Generalized {SAV}-exponential integrator
  schemes for {A}llen--{C}ahn type gradient flows}, SIAM J. Numer. Anal.
  \textbf{60} (2022), 1905--1931.

\bibitem{JSC_2022_Ju}
L.~Ju, X.~Li, and Z.~Qiao, \emph{Stabilized exponential-{SAV} schemes
  preserving energy dissipation law and maximum bound principle for the
  {A}llen--{C}ahn type equations}, J. Sci. Comput. \textbf{92} (2022), 66.

\bibitem{CiCP_2012_Kim}
J.~Kim, \emph{Phase-field models for multi-component fluid flows}, Commun.
  Comput. Phys. \textbf{12} (2012), 613--661.

\bibitem{JCP_Lan_2023}
R.~Lan, J.~Li, Y.~Cai, and L.~Ju, \emph{Operator splitting based
  structure-preserving numerical schemes for the mass-conserving convective
  {Allen--Cahn} equation}, J. Comput. Phys. \textbf{472} (2023), 111695.

\bibitem{SISC_2020_LiBuyang}
B.~Li, J.~Yang, and Z.~Zhou, \emph{Arbitrarily high-order exponential cut-off
  methods for preserving maximum principle of parabolic equations}, SIAM J.
  Sci. Comput. \textbf{42} (2020), A3957--A3978.

\bibitem{SISC_2021_Li}
J.~Li, X.~Li, L.~Ju, and X.~Feng, \emph{Stabilized integrating factor
  {R}unge-{K}utta method and unconditional preservation of maximum bound
  principle}, SIAM J. Sci. Comput. \textbf{43} (2021), A1780--A1802.

\bibitem{SINUM_2020_Liao}
H.~Liao, T.~Tang, and T.~Zhou, \emph{On energy stable, maximum-principle
  preserving, second-order {BDF} scheme with variable steps for the
  {A}llen--{C}ahn equation}, SIAM J. Numer. Anal. \textbf{58} (2020),
  2294--2314.

\bibitem{JCP_2024_Liao}
H.~Liao, X.~Wang, and C.~Wen, \emph{Original energy dissipation preserving
  corrections of integrating factor {R}unge-{K}utta methods for gradient flow
  problems}, J. Comput. Phys. \textbf{519} (2024), 113456.

\bibitem{LiuQiaoZhang}
C.~Liu, Z.~Qiao, and Q.~Zhang, \emph{Two-phase segmentation for intensity
  inhomogeneous images by the {A}llen--{C}ahn local binary fitting model}, SIAM
  J. Sci. Comput. \textbf{44} (2022), no.~1, B177--B196.

\bibitem{JSC_2024_Li}
Z.~Liu, Y.~Zhang, and X.~Li, \emph{A novel energy-optimized technique of
  {SAV}-based ({EOP-SAV}) approaches for dissipative systems}, J. Sci. Comput.
  \textbf{101} (2024), 38.

\bibitem{SISC_2016_Lv}
C.~Lv and C.~Xu, \emph{Error analysis of a high order method for
  time-fractional diffusion equations}, SIAM J. Sci. Comput. \textbf{38}
  (2016), A2699--A2724.

\bibitem{PRE_1997_Puri}
S.~Puri, A.J. Bray, and J.L. Lebowitz, \emph{Phase-separation kinetics in a
  model with order-parameter-dependent mobility}, Phys. Rev. E \textbf{56}
  (1997), 758--765.

\bibitem{SISC_2011_Qiao}
Z.~Qiao, Z.~Zhang, and T.~Tang, \emph{An adaptive time-stepping strategy for
  the molecular beam epitaxy models}, SIAM J. Sci. Comput. \textbf{33} (2011),
  1395--1414.

\bibitem{CMS_Shen_2016}
J.~Shen, T.~Tang, and J.~Yang, \emph{On the maximum principle preserving
  schemes for the generalized {A}llen--{C}ahn equation}, Commun. Math. Sci.
  \textbf{14} (2016), 1517--1534.

\bibitem{JCP_2018_Shen}
J.~Shen, J.~Xu, and J.~Yang, \emph{The scalar auxiliary variable ({SAV})
  approach for gradient flows}, J. Comput. Phys. \textbf{353} (2018), 407--416.

\bibitem{SIREV_2019_Shen}
J.~Shen, J.~Xu, and J.~Yang, \emph{A new class of efficient and robust energy
  stable schemes for gradient flows}, SIAM Rev. \textbf{61} (2019), 474--506.

\bibitem{SISC_2010_Shen}
J.~Shen and X.~Yang, \emph{A phase-field model and its numerical approximation
  for two-phase incompressible flows with different densities and viscosities},
  SIAM J. Sci. Comput. \textbf{32} (2010), 1159--1179.

\bibitem{CiCP_2010_Du}
G.~Sheng, T.~Wang, Q.~Du, K.G. Wang, Z.K. Liu, and L.Q. Chen, \emph{Coarsening
  kinetics of a two phase mixture with highly disparate diffusion mobility},
  Commun. Comput. Phys. \textbf{8} (2010), 249--264.

\bibitem{JCM_Tang_2016}
T.~Tang and J.~Yang, \emph{Implicit-explicit scheme for the {A}llen--{C}ahn
  equation preserves the maximum principle}, J. Comput. Math. \textbf{34}
  (2016), 471--481.

\bibitem{JSP_1994_Taylor}
J.E. Taylor and J.W. Cahn, \emph{Linking anisotropic sharp and diffuse surface
  motion laws via gradient flows}, J. Stat. Phys. \textbf{77} (1994), 183--197.

\bibitem{JSC_2010_Wise}
S.M. Wise, \emph{Unconditionally stable finite difference, nonlinear multigrid
  simulation of the {C}ahn--{H}illiard--{H}ele--{S}haw system of equations}, J.
  Sci. Comput. \textbf{44} (2010), 38--68.

\bibitem{JSC_2022_Yang}
J.~Yang, Z.~Yuan, and Z.~Zhou, \emph{Arbitrarily high-order maximum bound
  preserving schemes with cut-off postprocessing for {A}llen--{C}ahn
  equations}, J. Sci. Comput. \textbf{90} (2022), 76.

\bibitem{JCP_2006_Yang}
X.~Yang, J.J. Feng, C.~Liu, and J.~Shen, \emph{Numerical simulations of jet
  pinching-off and drop formation using an energetic variational phase-field
  method}, J. Comput. Phys. \textbf{218} (2006), 417--428.

\bibitem{JSC_2025_Zhang}
B.~Zhang, H.~Fu, R.~Lan, and S.~Xie, \emph{Energy dissipation law and maximum
  bound principle-preserving linear {BDF}2 schemes with variable steps for the
  {A}llen--{C}ahn equation}, J. Sci. Comput. \textbf{105} (2025), 51.

\bibitem{CMAME_2022_Zhang}
H.~Zhang, J.~Yan, X.~Qian, and S.~Song, \emph{Up to fourth-order
  unconditionally structure-preserving parametric single-step methods for
  semilinear parabolic equations}, Comput. Methods Appl. Mech. Engrg.
  \textbf{393} (2022), 114817.

\end{thebibliography}


\end{document}